\documentclass[11pt, oneside]{article}

\usepackage{epsfig}
\usepackage{tikz}
  \usetikzlibrary{calc}
\usetikzlibrary{patterns.meta}   % modern, flexible patterns
\usepackage{subcaption}

\usepackage{url}
\usepackage{amsmath,amstext,amssymb,amsfonts,amsthm}
\usepackage{mathrsfs}
\usepackage{xcolor}
\usepackage{mathtools}
\usepackage{graphicx}
\usepackage{relsize}   % for \mathlarger

\usetikzlibrary{fadings,shapes,arrows.meta,positioning}
\usepackage{hhline}

\usepackage{booktabs}
\usepackage{multicol}
\usepackage{multirow}
\usepackage[labelfont=bf]{caption}
\usepackage{subcaption}
\usepackage{siunitx}
\usepackage[title]{appendix}

\usepackage[authormarkup=none]{changes}
\definechangesauthor[name=reviewer0, color=red]{reviewer0}
\definechangesauthor[name=reviewer1, color=blue]{reviewer1}
\definechangesauthor[name=reviewer2, color=brown]{reviewer2}
\newtheorem{thm}{Theorem}[section]

\newtheorem{prop}{Proposition}[section]
\newtheorem{rem}{Remark}[section]

\newtheorem{exmp}{Example}

\numberwithin{equation}{section}
\numberwithin{figure}{section}

\newcommand{\dd}{\mathrm{d}}

\newcommand{\tildeminus}{%
  \mathrel{\vcenter{\offinterlineskip
    \halign{\hfil##\hfil\cr
      \(\sim\)\cr
      \noalign{\vskip -0.6ex}
      \(-\)\cr}}}}

\newcommand*\vimin{1.2}   % v_{i-1/2}
\newcommand*\viplus{2.2}  % v_{i+1/2}
\newcommand*\vmax{3.2}    % v_{\max}

\title{A Conservative and Positivity-Preserving Discontinuous Galerkin Method for the Population Balance Equation}

\author{Ziyao Xu\footnote{Department of Applied and Computational Mathematics and Statistics, University of Notre Dame, Notre Dame, IN 46556, USA. E-mail: zxu25@nd.edu. Also with the Department of Mathematics and Statistics, Binghamton University, Binghamton, NY 13902, USA. E-mail: zxu24@binghamton.edu} , Guanyang Liu\footnote{Process Engineering \& Modeling, CMC Synthetics, Sanofi, Cambridge, MA 02141, USA. E-mail: Guanyang.Liu@sanofi.com.}\; and\; Yong-Tao Zhang\footnote{Department of Applied and Computational Mathematics and Statistics,
University of Notre Dame, Notre Dame, IN 46556, USA. E-mail: yzhang10@nd.edu. Research of Yong-Tao Zhang is partially supported by Simons Foundation MPS-TSM-00007854}}
\date{}

\begin{document}

\maketitle

\textbf{Abstract.} 
We develop a conservative, positivity-preserving discontinuous Galerkin (DG) method for the population balance equation (PBE), which models the distribution of particle numbers across particle sizes due to growth, nucleation, aggregation, and breakage.
To ensure number conservation in growth and mass conservation in aggregation and breakage, we design a DG scheme that applies standard treatment for growth and nucleation, and introduces a novel discretization for aggregation and breakage. The birth and death terms are discretized in a symmetric double-integral form, evaluated using a common refinement of the integration domain and carefully selected quadrature rules.
Beyond conservation, we focus on preserving the positivity of the number density in aggregation–breakage. Since local mass corresponds to the first moment, the classical Zhang–Shu limiter, which preserves the zeroth moment (cell average), is not directly applicable. We address this by proving the positivity of the first moment on each cell and constructing a moment-conserving limiter that enforces nonnegativity across the domain.
To our knowledge, this is the first work to develop a positivity-preserving algorithm that conserves a prescribed moment. Numerical results verify the accuracy, conservation, and robustness of the proposed method.

\bigskip

\noindent{\bf Key Words:} Discontinuous Galerkin method; Population balance equation; Conservation; Positivity-preserving

\bigskip

\section{Introduction}
The population balance equation (PBE) provides a unified framework for modeling the time evolution of dispersed particle systems within a continuous environment, such as aerosols in air \cite{ramabhadran1976dynamics}, crystals in solution \cite{Myerson2002}, cells in bioreactors \cite{kolewe2012population}, and dust in disks \cite{lombart2020grain}.
In a spatially homogeneous setting, the number density $n(v, t)$ of particles, distributed over the internal coordinate $v$ (particle size), evolves through four key mechanisms: growth, nucleation, aggregation, and breakage.
Taking solution crystallization as an example, growth refers to the addition of solute molecules from the liquid phase to crystal surfaces, resulting in an increase in particle size; nucleation denotes the spontaneous formation of new crystals from a supersaturated solution; aggregation describes the collision and adhesion of smaller crystals into larger clusters; and breakage involves the fragmentation of larger crystals into smaller pieces due to hydrodynamic stress or mechanical forces.

The limited availability of analytical solutions has motivated the development of a variety of numerical methods for solving the population balance equation. 
These methods fall broadly into three categories: Monte Carlo methods \cite{maisels2004direct, patterson2011stochastic, xu2015accelerating}, moment methods \cite{diemer2002moment, marchisio2003quadrature, attarakih2009solution}, and discretization methods \cite{kumar1996solution, rigopoulos2003finite, filbet2004numerical}.
Monte Carlo methods simulate the stochastic evolution of individual particles and are advantageous for handling multi-dimensional distributions, but they often suffer from low accuracy and generate noisy results.
Moment methods track only the low-order moments of the number density distribution and are computationally efficient, but they do not reconstruct the full distribution without additional assumptions or closures.
Discretization methods, which are the focus of this paper, approximate the full distribution directly on a grid (also referred to as sections or bins in the literature) and are the most accurate and versatile for representing general PBEs.
However, the underlying physical mechanisms in the population balance equation (PBE) introduce specific numerical challenges for discretization methods. As an integro-differential equation, the PBE involves two primary sources of difficulty:
First, the aggregation and breakage terms appear as integral operators, whose discretization requires particular care to ensure the conservation of mass. These terms contain structurally distinct birth and death contributions. If the discretizations of these two components are not carefully coordinated, the resulting truncation errors may not cancel out, leading to global mass conservation errors in the numerical scheme.
Second, the growth term acts as a convective transport in the internal coordinate space, and thus inherits the numerical challenges typical of hyperbolic equations. In particular, sharp gradients in the distribution can cause spurious oscillations or numerical instability. Moreover, low-order schemes may introduce excessive numerical diffusion, especially toward the large-size region of the particle spectrum. Since number densities of large particles are often several orders of magnitude smaller than those of smaller particles, even mild over-diffusion can significantly contaminate the solution in this regime. This issue is frequently observed in low-resolution schemes.

One of the earliest discretization methods for approximating the number density distribution in the PBE with aggregation alone was presented by Hidy and Brock \cite{hidy1970dynamics}. Their method computes number densities on a uniform grid and closely resembles the discrete Smoluchowski equation \cite{smoluchowski1918}, an early form of the PBE that models aggregation processes over discrete particle sizes.
Bleck \cite{bleck1970fast} proposed a mass-weighted finite volume scheme for the aggregation PBE over logarithmic bins, yielding a sparse ODE system that facilitates fast computation.
In 1996, Kumar and Ramkrishna \cite{kumar1996solution} proposed the notable fixed-pivot technique to approximate number density on arbitrary grids. The method redistributes each newly formed aggregate onto the two nearest pivots in such a way that any two specified moments (typically the zeroth and first) are preserved exactly. They later generalized this idea to the moving-pivot technique for adaptive resolution \cite{kumar1996solution2}, and in a third paper \cite{kumar1997solution}, coupled it with a method-of-characteristics discretization for convective growth, thereby minimizing numerical diffusion and resolving stability issues arising from the hyperbolic nature of the growth term.
Lim et al. subsequently adopted a high-resolution finite difference WENO scheme to discretize the hyperbolic growth term, achieving sharper front tracking without sacrificing the fixed-pivot treatment of aggregation \cite{lim2002solution}.
Weighted residual methods were introduced by Rigopoulos and Jones \cite{rigopoulos2003finite} for the full PBE, using piecewise linear finite element trial functions. They also analyzed and mitigated potential sources of conservation errors.
Because exact conservation is delicate in the integro-differential form of the PBE, many authors recast the equation in terms of the mass density $m(v, t)=v n(v, t)$. Filbet and Laurençot \cite{filbet2004numerical} showed that $m(v, t)$ in the aggregation-only PBE satisfies a conservative formulation with non-local fluxes and solved it using a low-order finite volume method. This framework has since been widely adopted. 
Qamar and Warnecke \cite{Qamar2007FVM} embedded nucleation and growth into the same finite volume structure. 
Bourgade and Filbet \cite{bourgade2008convergence} analyzed the convergence of the finite volume scheme for the aggregation–breakage PBE in its conservative form. 
Gabriel and Tine \cite{gabriel2010high} improved accuracy by employing a fifth-order WENO reconstruction. 
Liu et al. \cite{liu2019high} designed a high-order, positivity-preserving DG scheme based on the conservative formulation, which Lombart and Laibe later adapted to grain growth in protoplanetary disks \cite{lombart2021grain}.
While the mass-density formulation naturally conserves the first moment, several recent studies revisit the original number-density equation and enforce conservation through specially constructed integration maps. Liu and Rigopoulos \cite{liu2019conservative} introduced an aggregation-map finite volume method that preserves mass to machine precision for the aggregation PBE. O’Sullivan and Rigopoulos \cite{o2022conservative} extended the map to include breakage, nucleation, and growth within a single conservative framework. Sewerin \cite{sewerin2023efficient} derived analytic quadrature rules that reduce the cost of evaluating the non-local integrals for both constant and linear reconstructions.
Building on these conservative finite volume ideas, the present work develops a high-order, positivity-preserving DG method that inherits exact moment conservation and retains sharp resolution on arbitrary meshes.

The discontinuous Galerkin (DG) method is a class of finite element methods that employs piecewise polynomials as basis functions.
Since its introduction by Reed and Hill in the 1970s for solving steady-state transport equations \cite{reed1973triangular}, the DG framework has been extensively studied and successfully extended to a wide range of problems, including hyperbolic equations \cite{shu2009discontinuous}, Hamilton–Jacobi equations \cite{cheng2007discontinuous}, Maxwell’s equations \cite{cockburn2004locally}, elliptic and parabolic equations \cite{riviere2008discontinuous, cockburn1998local}, and higher-order partial differential equations \cite{cheng2008discontinuous, xu2010local}.
Among its variants, the Runge--Kutta discontinuous Galerkin (RKDG) method developed by Cockburn and Shu \cite{RKDG2, RKDG1} has become one of the primary tools for solving conservation laws. 
It can be viewed as a high-order generalization of the finite volume method, where numerical fluxes are carefully designed to transfer information between neighboring cells, and integration of the equations against test functions allows for the accurate computation of multiple moments of the solution.
As a result, RKDG schemes offer advantages such as high-order accuracy, compact stencils, $h$-$p$ adaptivity, local conservation, and high parallel efficiency, making them well suited for complex geometries and large-scale simulations.

In this work, we apply the DG method to solve the population balance equation, with standard treatment for the hyperbolic growth term. 
However, the discretization of the aggregation and breakage terms presents a special challenge: while number conservation in the growth term arises naturally from the conservative form, mass conservation in aggregation--breakage is more delicate. 
To ensure that the first moment (mass) is preserved, we extend and refine the idea of the aggregation map from \cite{liu2019conservative, o2022conservative, sewerin2023efficient} to propose a DG discretization in which the birth and death terms are written in a symmetric double-integral form, evaluated using a common domain refinement based on constrained Delaunay triangulation (CDT) \cite{chew1987constrained} and $D_3$-symmetric quadrature rules \cite{wandzurat2003symmetric}. 
This guarantees exact cancellation of the birth and death contributions to mass.
However, conservation is not the only concern for numerical methods.
For high-order schemes to be robust, it is also essential to preserve the positivity of key physical quantities, such as the density. 
The influential framework of Zhang and Shu \cite{zhang2010maximum} introduced a general methodology for constructing positivity- and bound-preserving DG schemes. The key idea is to first ensure positivity of cell averages through flux choices and CFL constraints, then apply a scaling limiter that modifies the solution within each cell while preserving the cell average and high-order accuracy.
This strategy, however, cannot be directly applied to the aggregation--breakage PBE. 
In this setting, the quantity of interest is the number density $n(v,t)$, which must remain nonnegative. 
However, it is not the \textit{zeroth moment} (number) but the \textit{first moment} (mass) that must be conserved. 
As shown in later sections, while the total particle number may vary due to aggregation and breakage, the total particle mass remains constant. 
This distinction complicates the application of existing positivity-preserving frameworks.
One approach \cite{liu2019high} in the literature is to reformulate the PBE in terms of the mass density $m(v,t) = v n(v,t)$, which satisfies a conservation law with a nonlocal flux. 
Positivity-preserving schemes can then be developed for $m(v,t)$, and limiters can be applied to preserve the average local mass.
However, this formulation alters the standard hyperbolic conservation structure when a growth term is present, as the growth term no longer appears in conservative form.
In this paper, we take a different path: we work directly with the number density $n(v,t)$ and develop a DG scheme that maintains mass conservation through careful design. 
We prove the positivity of the first moment of the numerical solution and introduce a new scaling limiter that preserves an arbitrary moment while enforcing nonnegativity across the domain. 
A corresponding CFL condition is derived to ensure positivity preservation.
To the best of our knowledge, this is the first work to design a positivity-preserving DG algorithm that explicitly targets the conservation of a specified moment. We expect the methodology to be useful in a broader class of applications where physical moments, rather than cell averages, are the primary conserved quantities.

The remainder of the paper is organized as follows.
In Section 2, we present the model problem, including the mathematical formulation of the population balance equation (PBE) and the underlying conservation properties, which motivate the design of our conservative DG scheme.
Section 3 introduces the proposed high-order conservative discontinuous Galerkin method, establishes its conservation properties, and discusses key implementation details.
In Section 4, we develop the positivity-preserving strategy for the aggregation–breakage PBE, consisting of a two-step procedure: first ensuring the positivity of a cell moment, then applying a novel moment-conserving scaling limiter to enforce nonnegativity across the entire domain.
Section 5 presents numerical experiments that verify the method’s high-order accuracy, conservation properties, and positivity preservation.
%Finally, Section 6 offers concluding remarks.

\section{Model Problem}\label{sec:math}
The population balance equation (PBE), accounting for the growth, aggregation, breakage, and nucleation of particles, is an integro-differential equation given by
\begin{equation}\label{eq:PBE}
\begin{split}
\frac{\partial n(v,t)}{\partial t}&+\frac{\partial \big(G(v)n(v,t)\big)}{\partial v} = \frac12\int_{0}^{v}\beta(w,v-w)n(w,t)n(v-w,t)\,dw\\
& - \int_{0}^{\infty}\beta(v,w)n(v,t)n(w,t)\,dw + \int_{v}^{\infty}p(v,w)\gamma(w)n(w,t)\,dw\\
& -\gamma(v)n(v,t) + S(v),
\end{split}
\end{equation}
where the terms are defined as follows:
\begin{itemize}
\item $n(v,t)$: Number density of particles of volume $v$ at time $t$.
\item $G(v)$: Growth rate in volume for particles of size $v$.
\item $\beta(v,w)$: Aggregation kernel --- frequency at which particles of volumes $v$ and $w$ coalesce. It satisfies:
\begin{equation}\label{eq:collision_symmetry}
\beta(v,w)=\beta(w,v) \quad \text{(collision symmetry)}.
\end{equation}
\item $\gamma(v)$: Breakage rate for particles of volume $v$.
\item $p(v,w)$: Daughter distribution function --- the number density of fragments of size $v$ produced when a parent particle of volume $w$ breaks, with $v\leq w$.
It satisfies:
\begin{equation}\label{eq:mass_conservation}
\int_0^w v\,p(v,w)\,dv = w \quad \text{(mass conservation)}.
\end{equation}
\item $S(v)$: Nucleation rate for particles of volume $v$. Often modeled as a delta source: $S(v)=N^{\star}\delta(v-v^{\star})$, where $N^{\star}$ is the nucleation rate at the nucleation volume $v^{\star}$, and $\delta$ is the Dirac delta function.
\end{itemize}
The existence, uniqueness, and nonnegativity of solutions for aggregation have been proved for bounded kernels \cite{melzak1957scalar} and for certain unbounded kernels \cite{mcleod1962infinite}.  
In this paper, we assume that terms in \eqref{eq:PBE} are either bounded or grow mildly, so that all integrals appearing later remain finite.
Below, we demonstrate key properties of the PBE, including number conservation for the growth and mass conservation for the aggregation and breakage.
Although the proofs of these properties are straightforward, we present them in detail because the design of the numerical schemes closely aligns with these proofs.

\begin{prop}
The growth of particles conserves the total particle number; that is, the particle number $N(t) := \int_{0}^{\infty} n(v,t)\,dv$ remains constant in the population balance equation when only the growth term is present:
\begin{equation}\label{eq:PBE_growth}
\frac{\partial n(v,t)}{\partial t} + \frac{\partial \big(G(v)n(v,t)\big)}{\partial v} = 0,
\end{equation}
provided that the number density $n(\cdot, t) \in \mathcal{C}^{\infty}_{0}(\mathbb{R}^{+})$ is compactly supported.
\end{prop}

\begin{proof}
From \eqref{eq:PBE_growth}, it is straightforward to compute that
\begin{equation*}
\begin{split}
\frac{\dd N(t)}{\dd t} & = \int_{0}^{\infty} \frac{\partial n(v,t)}{\partial t}\,dv = -\int_{0}^{\infty} \frac{\partial \big(G(v)n(v,t)\big)}{\partial v} \,dv = -G(v)n(v,t)\Big|_{0}^{\infty} = 0.
\end{split}
\end{equation*}
\end{proof}

\begin{prop}\label{prop:mass_conservation}
The aggregation and breakage of particles conserve the total particle mass; that is, the particle mass $M(t) := \int_{0}^{\infty}vn(v,t)\,dv$ (assuming constant particle density) remains constant in the population balance equation when only the aggregation and/or breakage terms are present:
\begin{equation}\label{eq:PBE_aggregation_breakage}
\begin{split}
\frac{\partial n(v,t)}{\partial t}=& \frac12\int_{0}^{v}\beta(w,v-w)n(w,t)n(v-w,t)\,dw - \int_{0}^{\infty}\beta(v,w)n(v,t)n(w,t)\,dw \\
&+ \int_{v}^{\infty}p(v,w)\gamma(w)n(w,t)\,dw -\gamma(v)n(v,t).
\end{split}
\end{equation}
\end{prop}
\begin{proof}
From \eqref{eq:PBE_aggregation_breakage}, one can compute that
\begin{equation*}
\begin{split}
\frac{\dd M(t)}{\dd t}&=\int_{0}^{\infty}v\frac{\partial n(v,t)}{\partial t}\,dv:=\text{I}+\text{II},
\end{split}
\end{equation*}
where 
\begin{equation*}
\text{I}\, =\, \underbrace{\int_{0}^{\infty} \int_{0}^{v}\frac12 v\beta(w,v-w)n(w,t)n(v-w,t)\,dw \,dv}_{\text{I}_{1}}-\underbrace{\int_{0}^{\infty} \int_{0}^{\infty}v\beta(v,w)n(v,t)n(w,t)\,dw \,dv}_{\text{I}_{2}}
\end{equation*}
and 
\begin{equation*}
\begin{split}
\text{II}\, =\, \underbrace{\int_{0}^{\infty} \int_{v}^{\infty}vp(v,w)\gamma(w)n(w,t)\,dw\,dv}_{\text{II}_1} - \underbrace{\int_{0}^{\infty}v\gamma(v)n(v,t)\,dv}_{\text{II}_2}
\end{split}
\end{equation*}
represent the change of mass due to aggregation and breakage, respectively, where $\text{I}_1$ and $\text{II}_1$ are the birth terms, and $\text{I}_2$ and $\text{II}_2$ are the corresponding death terms.
It suffices to show that $\text{I}_1=\text{I}_{2}$ and $\text{II}_1=\text{II}_2$.

For aggregation, by the change of variables $u = v - w$ 
%(see the illustration in Figure~\ref{fig:CoVsI1})
, we obtain
\begin{equation*}
\begin{split}
\text{I}_1=\int_{0}^{\infty} \int_{0}^{\infty}\frac12 (w+u)\beta(w,u)n(w,t)n(u,t)\,dw \,du=\text{I}_2,
\end{split}
\end{equation*}
where the second equality follows from the collision symmetry~\eqref{eq:collision_symmetry}.

For breakage, by the mass conservation \eqref{eq:mass_conservation}, we obtain
\begin{equation*}
\begin{split}
\text{II}_2 = \int_{0}^{\infty} \Big(\int_{0}^{v}wp(w,v)\,dw\Big)\gamma(v)n(v,t)\,dv =
\int_{0}^{\infty} \int_{w}^{\infty}wp(w,v)\gamma(v)n(v,t)\,dvdw =
\text{II}_1,
\end{split}
\end{equation*}
where the second equality follows from changing the order of integration.

We remark that if a rapidly growing kernel induces gelation  \cite{leyvraz1981singularities, filbet2004numerical} (loss of mass from the finite-size regime due to transfer into an untracked gel phase), the integrals $\text{I}_1$ and $\text{I}_2$ can diverge, and the cancellation becomes the indeterminate form $\infty - \infty$.  
Throughout this paper, we restrict our attention to bounded or mildly growing kernels for which all relevant integrals remain finite, so the proof is valid.

\end{proof}

\section{Discontinuous Galerkin Method}\label{sec:DG}
In this section, we propose a high-order discontinuous Galerkin method for the population balance equations, which respects their conservation properties.

\subsection{Prerequisites}
We denote by $v_{\max}$ the maximum particle size in the problem under consideration and represent the computational domain as $\mathcal{D}=[0,v_{\max}]$.
Consider a partition $0=v_{\frac12}\leq v_{\frac32}<\cdots<v_{L+\frac12}=v_{\max}$ of the domain.
We denote by $I_i=[v_{i-\frac12}, v_{i+\frac12}]$ the $i$-th cell of the grid, with cell center at $v_{i}=\frac12(v_{i-\frac12}+v_{i+\frac12})$ and cell length $\Delta v_{i}=v_{i+\frac12}-v_{i-\frac12}$, for $i=1,2,\ldots, L$.
Moreover, we define $\Delta v=\min_{1\leq i\leq L}\Delta v_{i}$.

The discontinuous Galerkin space $S_{h}^{k}$ (we assume $k\geq 1$ throughout the paper) is defined on the grid as follows:
\begin{equation}\label{eq:DG_space}
S_{h}^{k}=\big\{\varphi\in L^{2}(\mathcal{D}): \varphi\big|_{I_i}\in\mathcal{P}^{k}(I_{i}), \, 1\leq i\leq L \big\},
\end{equation}
where $\mathcal{P}^{k}(I_i)$ denotes the space of polynomials of degree at most $k$ on the interval $I_i$.
Since functions in the DG space are piecewise-defined and may be discontinuous at cell interfaces $v_{i+\frac12}$, we write the restriction of $\varphi\in S_{h}^{k}$ to $I_i$ as $\varphi^i=\varphi\big|_{I_i}$ for convenience, and adopt the notation
\begin{equation*}
\varphi_{i+\frac12}^{\pm}=\lim_{\epsilon\rightarrow 0^{+}}\varphi(v_{i+\frac12}\pm\epsilon),
\end{equation*}
to denote the right and left limits at each interface.

We adopt the basis of orthogonal polynomials on each interval $I_i$ to construct the DG space $S_{h}^{k}$:
\begin{equation}\label{eq:FEMbasis}
\begin{split}
S_{h}^{k}&=\bigoplus_{i=1}^{L} \text{span}\{{\phi}_j^{i}(v), \, j=0,1,\ldots, k\},
\end{split}
\end{equation}
where $\phi_{j}^{i}(v)=\widehat{\phi}_j(\frac{2(v-v_{i})}{\Delta v_{i}})$ is obtained via an affine mapping of the orthonormal basis functions $\widehat{\phi}_{j}\in\mathcal{P}^{j}(I)$ on the reference interval $I=[-1,1]$.

Below, we present the DG methods for the PBE with growth and nucleation, aggregation, and breakage treated separately. The complete DG scheme for \eqref{eq:PBE} is then obtained by summing their respective contributions.

\subsection{Growth and nucleation}

We consider the PBE with growth and nucleation:
\begin{equation}\label{eq:PBE_growth_nucleation}
\frac{\partial n(v,t)}{\partial t}+\frac{\partial \big(G(v)n(v,t)\big)}{\partial v} = S(v).
\end{equation}
The standard semi-discrete DG formulation is given as follows: Find $n_h(v,t) \in S_{h}^{k}$ such that, for $1\leq i\leq L$,
\begin{equation}\label{eq:DG_growth_nucleation}
\begin{split}
\int_{I_i}(n_h)_t\varphi dv = & \int_{I_i}G(v) n_h(v,t)\varphi'(v)\,dv+G_{i-\frac12}^{-}(n_h)_{i-\frac12}^{-}\varphi_{i-\frac12}^{+}-G_{i+\frac12}^{-}(n_h)_{i+\frac12}^{-}\varphi_{i+\frac12}^{-}\\
& + \int_{I_i}S(v)\varphi(v)\,dv\quad \forall\,\varphi\in \mathcal{P}^{k}(I_i),    
\end{split}
\end{equation}
where $(n_h)_{\frac12}^{-} := n(0,t)$.
The upwind flux is adopted for the growth term, as the volume growth is monotonic.

In practice, the integrals on the right-hand side of \eqref{eq:DG_growth_nucleation} are evaluated using numerical quadrature.
For an integral of function $f$ on the interval $I_i$, we employ the Gauss-Lobatto quadrature rule:
\begin{equation}\label{eq:1D_quad}
\mathlarger{\mathcal{Q}}_{I_i}\big[f(u)\big]:=\Delta v_i\sum_{\alpha=1}^{N_G}\varpi_{\alpha}^{G}f(u_{\alpha}^{G}),
\end{equation}
where $\mathcal{Q}_{I_i}$ is the quadrature operator on $I_i$, $N_G$ is the number of quadrature points, $u_{\alpha}^{G}$ are the Gauss-Lobatto points, and $\varpi_{\alpha}^{G}$ are the corresponding weights.
We choose $N_G\geq k+2$ to ensure adequate algebraic accuracy \cite{RKDG4}.
Note that the Gauss–Lobatto points include the endpoints of the interval, and the quadrature weights are strictly positive.

The integrals in \eqref{eq:DG_growth_nucleation} can then be evaluated as:
\begin{equation}\label{eq:growth_quadrature}
\int_{I_i}G(v) n_h(v,t)\varphi'(v)\,dv\tildeminus \mathlarger{\mathcal{Q}}_{I_i}\big[G(v) n_h(v,t)\varphi'(v)\big],
\end{equation}
and
\begin{equation}\label{eq:nucleation_quadrature}
\int_{I_i}S(v)\varphi(v)\,dv \tildeminus  \mathlarger{\mathcal{Q}}_{I_i}\big[S(v)\varphi(v)\big].
\end{equation}
The scheme defined by \eqref{eq:DG_growth_nucleation}, \eqref{eq:growth_quadrature}, and \eqref{eq:nucleation_quadrature} constitutes the DG method for \eqref{eq:PBE_growth_nucleation}.

It is clear that the DG scheme for growth is both locally and globally conservative with respect to the particle number $N_h(t):=\int_{0}^{v_{\max}}n_h(v,t)\,dv$ in the absence of nucleation, which can be obtained by choosing $\varphi$ as piecewise constants.

\subsection{Aggregation}
We consider the PBE with aggregation:
\begin{equation}\label{eq:PBE_aggregation}
\frac{\partial n(v,t)}{\partial t} = \frac12\int_{0}^{v}\beta(w,v-w)n(w,t)n(v-w,t)\,dw - \int_{0}^{\infty}\beta(v,w)n(v,t)n(w,t)\,dw.
\end{equation}
The semi-discrete DG formulation is given as follows: Find $n_h(v,t)\in S_{h}^{k}$ such that, for $1\leq i\leq L$,
\begin{equation}\label{eq:DG_aggregation}
\begin{split}
\int_{I_i}(n_h)_t\varphi \,dv=& \frac12\iint_{A_i}\varphi(w+u)\beta(w,u)n_h(w,t)n_h(u,t)\,dwdu \\
&- \iint_{B_i}\varphi(u)\beta(u,w)n_h(u,t)n_h(w,t)\,dwdu\qquad \forall\,\varphi\in \mathcal{P}^{k}(I_i),
\end{split}
\end{equation}
where 
\begin{equation}\label{eq:partition_Ai}
A_i=\{(u,w)\in\mathbb{R}^{+}\times\mathbb{R}^{+}: v_{i-\frac12}\leq u+w\leq v_{i+\frac12}\},\quad \text{for}\: i=1,2,\ldots, L,
\end{equation}
and
\begin{equation}\label{eq:partition_Bi}
B_{i}=\{(u,w)\in\mathbb{R}^{+}\times\mathbb{R}^{+}:v_{i-\frac12}\leq u\leq v_{i+\frac12}, u+w\leq v_{\max}\},\quad \text{for}\: i=1,2,\ldots, L,
\end{equation}
are the integration regions for the birth and death terms of aggregation, respectively (see an illustration in Figure~\ref{fig:Ai_Bi_regions}).
Here, the integration region $A_{i}$ is obtained via the change of variables $u=v-w$ in the double integral $\frac12\int_{v_{i-\frac12}}^{v_{i+\frac12}}\int_{0}^{v}\beta(w,v-w)n_h(w,t)n_h(v-w,t)\,dw dv$, following the same steps as in the proof of Proposition \ref{prop:mass_conservation}.
\begin{figure}[htbp]
  \centering

  % --------- Sub-figure 1 :  A_i  -------------------------------
  \begin{subfigure}[t]{0.45\textwidth}
    \centering
    \begin{tikzpicture}[scale=1.1]
      % axes -----------------------------------------------------
      \draw[->] (0,0) -- (3.6,0) node[right] {$u$};
      \draw[->] (0,0) -- (0,3.6) node[above] {$w$};

      % light-gray fill for A_i ---------------------------------
      \path[fill=gray!25]
        (0,\viplus) -- (\viplus,0) --
        (\vimin,0)  -- (0,\vimin)  -- cycle;

      % bounding lines u+w = v_{i±1/2} --------------------------
      \draw (0,\vimin)  -- (\vimin,0)
            node[below right=-1pt] {};
      \draw (0,\viplus) -- (\viplus,0)
            node[below right=-1pt] {};

      % dashed line u+w = v_max ---------------------------------
      \draw[dashed] (0,\vmax) -- (\vmax,0);

      % numeric labels ------------------------------------------
      \node[below] at (\vimin,0)  {$v_{i-\frac12}$};
      \node[below] at (\viplus,0) {$v_{i+\frac12}$};
      \node[below] at (\vmax,0) {$v_{\max}$};
      \node[left]  at (0,\vimin)  {$v_{i-\frac12}$};
      \node[left]  at (0,\viplus) {$v_{i+\frac12}$};
      \node[left]  at (0,\vmax) {$v_{\max}$};
      
    \end{tikzpicture}
    \caption{Region $A_i$ for the birth term}
  \end{subfigure}
  \hfill
  % --------- Sub-figure 2 :  B_i  -------------------------------
  \begin{subfigure}[t]{0.45\textwidth}
    \centering
    \begin{tikzpicture}[scale=1.1]
      % axes -----------------------------------------------------
      \draw[->] (0,0) -- (3.6,0) node[right] {$u$};
      \draw[->] (0,0) -- (0,3.6) node[above] {$w$};

      % light-gray fill for B_i ---------------------------------
      \path[fill=gray!25]
        (\vimin,0) --
        (\viplus,0) --
        (\viplus,{\vmax-\viplus}) --
        (\vimin,{\vmax-\vimin}) -- cycle;

      % vertical bounds u = v_{i±1/2} ---------------------------
      \draw (\vimin,0) -- (\vimin,2.0);
      \draw (\viplus,0) -- (\viplus,1.0);

      % dashed line u+w = v_max ---------------------------------
      \draw[dashed] (0,\vmax) -- (\vmax,0);

      % axis ticks (u-axis) -------------------------------------
      \node[below] at (\vimin,0)  {$v_{i-\frac12}$};
      \node[below] at (\viplus,0) {$v_{i+\frac12}$};
      \node[below] at (\vmax,0)  {$v_{\max}$};
      \node[left] at (0,\vmax) {$v_{\max}$};
      
    \end{tikzpicture}
    \caption{Region $B_i$ for the death term}
  \end{subfigure}

  \caption{Integration regions on $\Omega^a$ used in the DG formulation \eqref{eq:DG_aggregation} for aggregation.}
  \label{fig:Ai_Bi_regions}
\end{figure}
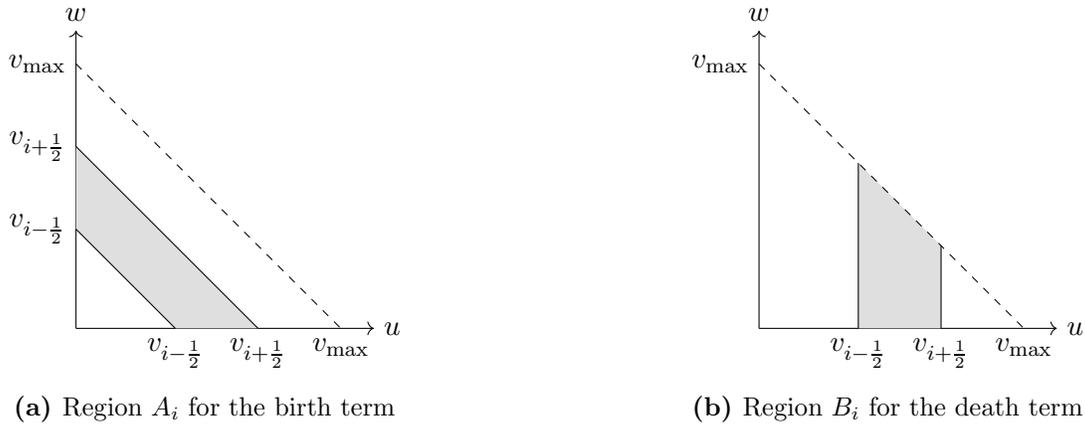

Since the aggregation kernel $\beta(w,u)$ is generally non-polynomial (see \eqref{eq:kernel_free}–\eqref{eq:kernel_gravity}), the integrals on the right-hand side of \eqref{eq:DG_aggregation} are evaluated using numerical quadrature in practice, which introduces truncation errors.
This gives rise to an issue of mass conservation, as the birth and death terms may not cancel exactly when summed over the entire aggregation region 
\begin{equation}\label{eq:region_aggr}
\Omega^a = \{(u,w)\in\mathbb{R}^{+}\times\mathbb{R}^{+} : u+w\leq v_{\max}\}    
\end{equation}
In this work, we extend and refine the idea of the aggregation map from \cite{liu2019conservative, o2022conservative, sewerin2023efficient} to discontinuous Galerkin methods to address this problem.

We define the regions 
\begin{equation}\label{eq:partition_Bi'}
B'_{i}=\{(u,w)\in\mathbb{R}^{+}\times\mathbb{R}^{+}:v_{i-\frac12}\leq w\leq v_{i+\frac12}, u+w\leq v_{\max}\},\quad \text{for}\: i=1,2,\ldots, L,
\end{equation}
according to the grid of $\mathcal{D}$ along the $w$-axis.
Given the partitions $\Omega^a=\cup_{i=1}^{L}A_{i}=\cup_{i=1}^{L}B_{i}=\cup_{i=1}^{L}B'_{i}$, we say that $\mathcal{T}$ is a \emph{common refinement} of partitions $\{A_{i}\}_{i=1}^{L}$, $\{B_{i}\}_{i=1}^{L}$, and $\{B'_{i}\}_{i=1}^{L}$ if $\mathcal{T}$ is a partition of $\Omega^a$ and, for every $T\in\mathcal{T}$, there exist indices $i, j$, and $k$ such that
\begin{equation*}
T\subseteq A_{i}\cap B_{j}\cap B'_{k}.
\end{equation*}
For instance, the coarsest common refinement of partitions $\{A_{i}\}_{i=1}^{L}$, $\{B_{i}\}_{i=1}^{L}$, and $\{B'_{i}\}_{i=1}^{L}$ is given by 
\begin{equation}
\mathcal{T}^0=\{C_{i,j,k}=A_{i}\cap B_{j}\cap B'_{k}\}_{i,j,k=1}^{L}.
\end{equation}

The elements of $\mathcal{T}^{0}$ may consist of triangles, quadrilaterals, pentagons, and even hexagons.
An example of $\mathcal{T}^{0}$ for a nonuniform grid of $\mathcal{D}$ is shown in Figure \ref{fig:Partition_aggr} (a).
To enable uniform treatment of elements in $\mathcal{T}$, our algorithm applies constrained Delaunay triangulation (CDT) to $\mathcal{T}^{0}$, enforcing \emph{$u$--$w$ symmetry}, to produce a triangular partition $\mathcal{T}$ that respects both the original edges in $\mathcal{T}^{0}$ and the diagonal $u=w$; see Figure \ref{fig:Partition_aggr} (b) for an illustration.
\begin{figure}[!htbp]
 \centering
 \begin{subfigure}[b]{0.40\textwidth}
  \includegraphics[width=\textwidth]{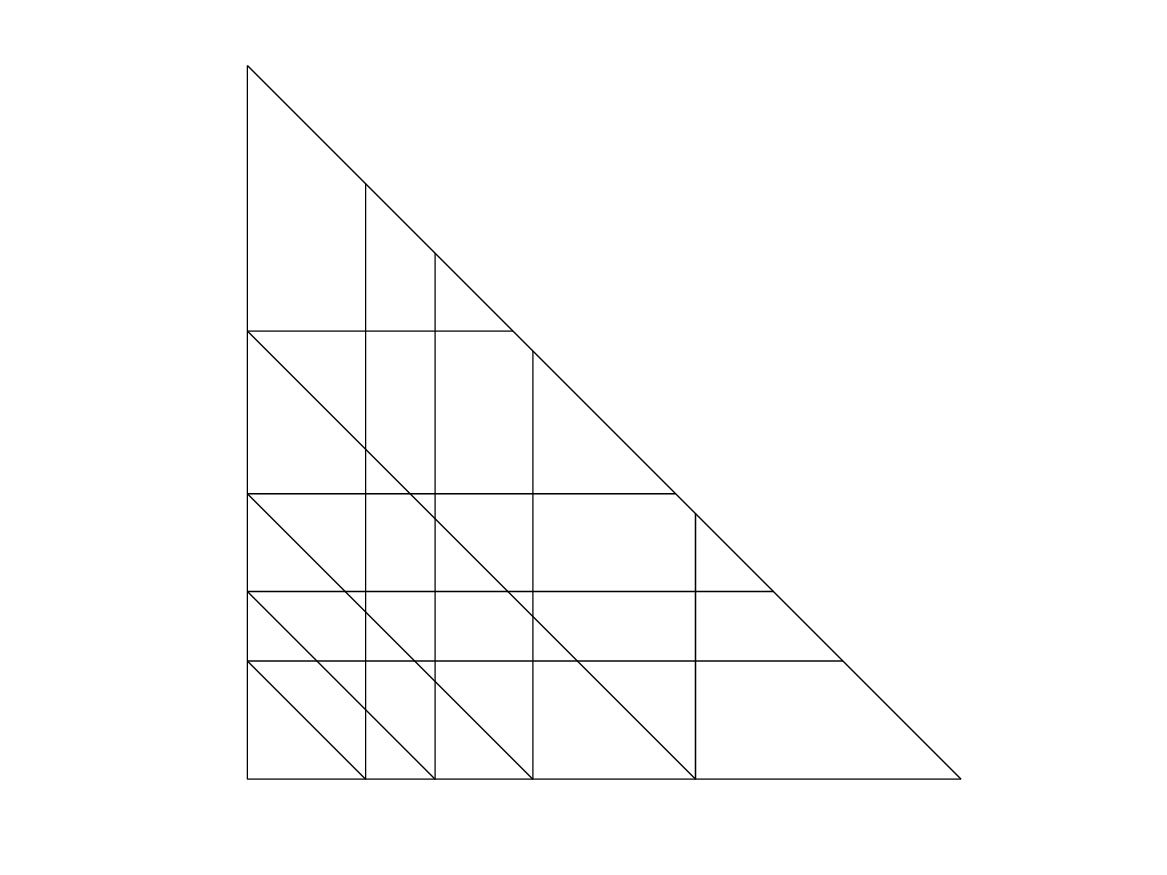}
  \caption{Coarsest common refinement $\mathcal{T}^0$}
 \end{subfigure}
 \begin{subfigure}[b]{0.40\textwidth}
  \includegraphics[width=\textwidth]{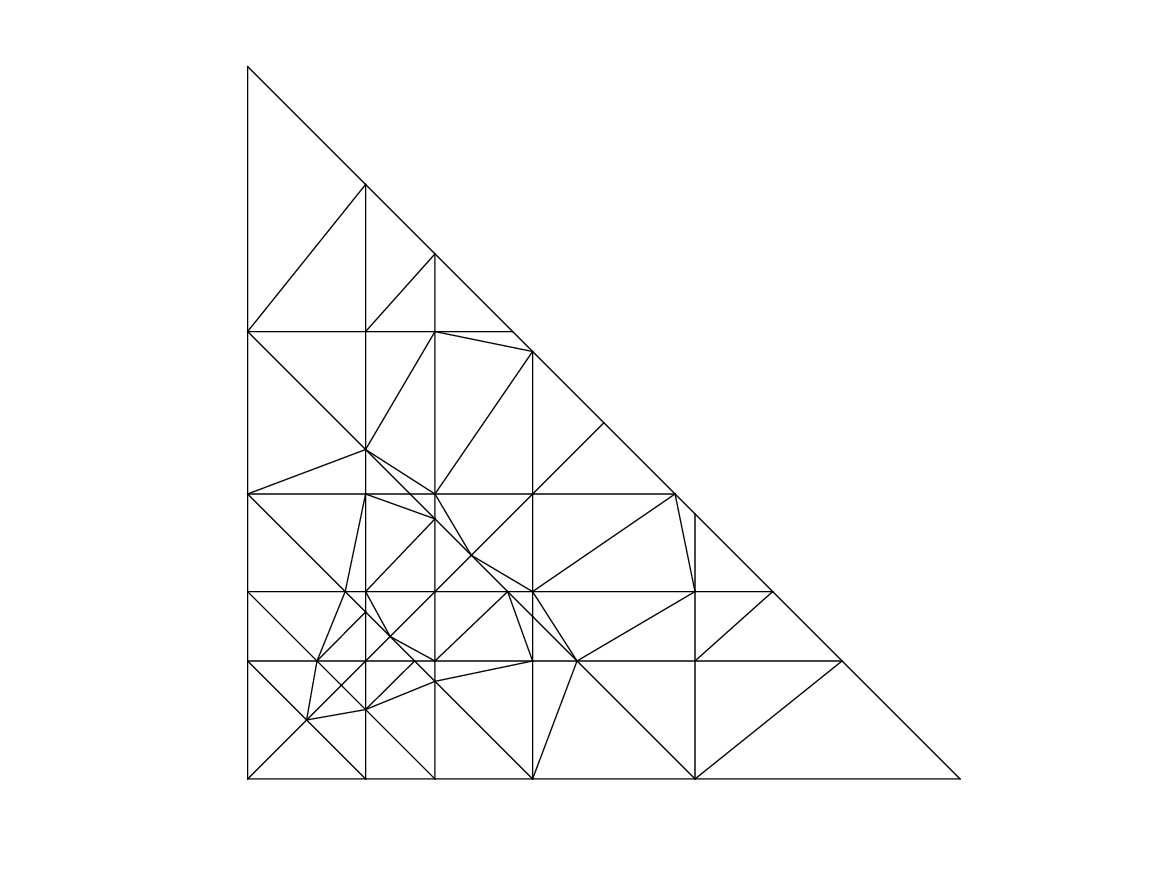}
  \caption{Triangular common refinement $\mathcal{T}$}
 \end{subfigure} 
 \caption{Example of common refinements for the partitions \eqref{eq:partition_Ai}, \eqref{eq:partition_Bi}, and \eqref{eq:partition_Bi'} of the aggregation region. $\mathcal{T}^0$ is the coarsest common refinement, and $\mathcal{T}$ is a refinement obtained using constrained Delaunay triangulation (CDT) applied to $\mathcal{T}^0$ with $u$--$w$ symmetry, which is quadrature-friendly.}
 \label{fig:Partition_aggr}
\end{figure}
On each triangular element $T\in\mathcal{T}$, we adopt the quadrature rules developed in \cite{wandzurat2003symmetric}, which exhibit $D_3$-symmetry, have positive weights, and include no points outside the triangle, to compute double integrals $\iint_{T}f(u,w)\,dudw$:
\begin{equation}\label{eq:tri_quad}
{\mathlarger{\mathcal{Q}}}_{T}\big[f(u,w)\big]:=|T|\sum_{\gamma=1}^{N_{Q}}\varpi_{\gamma} f(u_{\gamma}, w_{\gamma}),
\end{equation}
where $\mathcal{Q}_T$ is the quadrature operator on $T$, $f$ is a function on $T$, $|T|$ denotes the area of $T$, $N_Q$ is the number of quadrature points, $(u_{\gamma},w_{\gamma})$ are the quadrature points, and $\varpi_\gamma$ are the corresponding weights.
Here, the $D_3$-symmetry refers to the property that a quadrature rule on an equilateral triangle is invariant under reflections across the three medians and rotations about the center by $\frac{2\pi}{3}$ and $\frac{4\pi}{3}$.

The double integrals in \eqref{eq:DG_aggregation} can then be evaluated as:
\begin{equation}\label{eq:aggr_birth_quadrature}
\iint_{A_i}\varphi(w+u)\beta(w,u)n_h(w,t)n_h(u,t)\,dwdu \tildeminus\sum_{\substack{T \in \mathcal{T} \\ T \subseteq A_i}}\mathlarger{\mathcal{Q}}_T\big[\varphi(w+u)\beta(w,u)n_h(w,t)n_h(u,t) \big],
\end{equation}
and 
\begin{equation}\label{eq:aggr_death_quadrature}
\iint_{B_i}\varphi(u)\beta(u,w)n_h(u,t)n_h(w,t)\,dwdu\tildeminus\sum_{\substack{T \in \mathcal{T} \\ T \subseteq B_i}}\mathlarger{\mathcal{Q}}_T\big[\varphi(u)\beta(u,w)n_h(u,t)n_h(w,t)\big].
\end{equation}
The scheme defined by \eqref{eq:DG_aggregation}, \eqref{eq:aggr_birth_quadrature}, and \eqref{eq:aggr_death_quadrature} constitutes the DG method for \eqref{eq:PBE_aggregation}.
\begin{prop}\label{prop:aggregation_conservation}
    The total particle mass is conserved in the DG scheme \eqref{eq:DG_aggregation}, \eqref{eq:aggr_birth_quadrature}, and \eqref{eq:aggr_death_quadrature} for aggregation; that is, the particle mass $M_h(t) := \int_{0}^{v_{\max}}vn_h(v,t)\,dv$ remains constant over time.
\end{prop}
\begin{proof}
We take the test function $\varphi(v) = v$ in the DG formulation \eqref{eq:DG_aggregation} and sum the equations over all $i = 1, 2, \ldots, L$ to obtain
\begin{equation*}
\begin{split}
&\frac{\dd M_{h}(t)}{\dd t}=\sum_{i=1}^{L}\int_{I_i}(n_h)_t v\,dv\\
&=\frac12\sum_{T\in \mathcal{T}}\mathlarger{\mathcal{Q}}_T\big[(w+u)\beta(w,u)n_h(w,t)n_h(u,t) \big]-\sum_{T\in \mathcal{T}}\mathlarger{\mathcal{Q}}_T\big[u\beta(u,w)n_h(u,t)n_h(w,t)\big]\\
&=\sum_{T\in \mathcal{T}}\mathlarger{\mathcal{Q}}_T\big[u\beta(w,u)n_h(w,t)n_h(u,t) \big]-\sum_{T\in \mathcal{T}}\mathlarger{\mathcal{Q}}_T\big[u\beta(u,w)n_h(u,t)n_h(w,t)\big]\\
&=0,
\end{split}
\end{equation*}
where the third equality follows from the $u$--$w$ symmetry of the integrand, the partition $\mathcal{T}$, and the quadrature rule.
\end{proof}

\subsection{Breakage}
We consider the PBE with breakage:
\begin{equation}\label{eq:PBE_breakage}
\frac{\partial n(v,t)}{\partial t} =  \int_{v}^{\infty}p(v,w)\gamma(w)n(w,t)\,dw -\gamma(v)n(v,t).
\end{equation}
The semi-discrete DG formulation is given as follows: Find $n_h(v,t)\in S_{h}^{k}$ such that, for $1\leq i\leq L$,
\begin{equation}\label{eq:DG_breakage}
\begin{split}
\int_{I_i}(n_h)_t\varphi \,dv=&\iint_{C_{i}}\varphi(u)p(u,w)\gamma(w)n_h(w,t)\,dwdu - \iint_{D_{i}} u p(u,w)\gamma(w)n_h(w,t)\,dwdu\\
&-\int_{I_i}\Big(\varphi(v)-v\Big)\gamma(v)n_h(v,t)dv\qquad \forall\,\varphi\in \mathcal{P}^{k}(I_i),
\end{split}
\end{equation}
where 
\begin{equation}\label{eq:partition_Ci}
C_{i}=\{(u,w)\in\mathbb{R}^{+}\times\mathbb{R}^{+}: v_{i-\frac12}\leq u \leq v_{i+\frac12}, u\leq w\leq v_{\max}\}
\end{equation}
and 
\begin{equation}\label{eq:partition_Di}
D_{i}=\{(u,w)\in\mathbb{R}^{+}\times\mathbb{R}^{+}: v_{i-\frac12}\leq w\leq v_{i+\frac12}, 0\leq u\leq w\}
\end{equation}
are the integration regions for the birth and (partial) death terms of breakage, respectively (see an illustration in Figure \ref{fig:Ci_Di_regions}).
Here, the integration region $D_{i}$ is derived using the mass conservation \eqref{eq:mass_conservation} in $\int_{I_i}w\gamma(w)n_h(w,t)\,dw$, which leads to the double integral $\int_{v_{i-\frac12}}^{v_{i+\frac12}}\big(\int_{0}^{w}up(u,w)\,du\big)\gamma(w)n_h(w,t)\,dw$, following the same steps as in the proof of Proposition \ref{prop:mass_conservation}.

\begin{figure}[htbp]
  \centering

  % ================================================================
  %  Sub-figure 1 :   C_i  (birth term in breakage)
  % ================================================================
  \begin{subfigure}[t]{0.45\textwidth}
    \centering
    \begin{tikzpicture}[scale=1.1]
      % ------------------------------------------------------------
      % Axes
      % ------------------------------------------------------------
      \draw[->] (0,0) -- (3.6,0) node[right] {$u$};
      \draw[->] (0,0) -- (0,3.6) node[above] {$w$};

      % ------------------------------------------------------------
      % Light-gray fill for C_i :  v_{i-1/2} ≤ u ≤ v_{i+1/2},
      %                            u ≤ w ≤ v_max
      % ------------------------------------------------------------
      \path[fill=gray!25]
        (\vimin,\vimin) --
        (\vimin,\vmax) --
        (\viplus,\vmax) --
        (\viplus,\viplus) -- cycle;

      % ------------------------------------------------------------
      % Bounding lines (dashed)
      %   w = u      (diagonal)
      %   w = v_max  (horizontal)
      % ------------------------------------------------------------
      \draw[dashed] (0,0) -- (\vmax,\vmax);   % w = u
      \draw[dashed] (0,\vmax) -- (\vmax, \vmax); % w = v_max

      % ------------------------------------------------------------
      % Vertical bounds u = v_{i±1/2}  (thin solid)
      % ------------------------------------------------------------
      \draw (\vimin,\vimin) -- (\vimin,\vmax);
      \draw (\viplus,\viplus) -- (\viplus,\vmax);
      \draw (\vimin,-0.1) -- (\vimin,0.1);
      \draw (\viplus,-0.1) -- (\viplus,0.1);
      % ------------------------------------------------------------
      % Axis labels / ticks
      % ------------------------------------------------------------
      \node[below] at (\vimin,0)  {$v_{i-\frac12}$};
      \node[below] at (\viplus,0) {$v_{i+\frac12}$};
      \node[left]  at (0,\vmax)   {$v_{\max}$};
      % Axis labels / ticks
      % ------------------------------------------------------------
      \node[left] at (0,\vimin)  {$v_{i-\frac12}$};
      \node[left] at (0,\viplus) {$v_{i+\frac12}$};
      \draw (-0.1,\vimin) -- (0.1,\vimin);
      \draw (-0.1,\viplus) -- (0.1,\viplus);      
    \end{tikzpicture}
    \caption{Region $C_i$ for the birth term}
  \end{subfigure}
  \hfill
  % ================================================================
  %  Sub-figure 2 :   D_i  (death term in breakage)
  % ================================================================
  \begin{subfigure}[t]{0.45\textwidth}
    \centering
    \begin{tikzpicture}[scale=1.1]
      % ------------------------------------------------------------
      % Axes
      % ------------------------------------------------------------
      \draw[->] (0,0) -- (3.6,0) node[right] {$u$};
      \draw[->] (0,0) -- (0,3.6) node[above] {$w$};

      % ------------------------------------------------------------
      % Light-gray fill for D_i :  v_{i-1/2} ≤ w ≤ v_{i+1/2},
      %                            0 ≤ u ≤ w
      % ------------------------------------------------------------
      \path[fill=gray!25]
        (0,\vimin) --
        (0,\viplus) --
        (\viplus,\viplus) --
        (\vimin,\vimin) -- cycle;

      % ------------------------------------------------------------
      % Bounding lines (dashed)
      %   w = u      (diagonal)
      %   w = v_max  (horizontal)  — drawn for reference
      % ------------------------------------------------------------
      \draw[dashed] (0,0) -- (\vmax,\vmax);     % w = u
      \draw[dashed] (0,\vmax) -- (\vmax,\vmax); % w = v_max

      % ------------------------------------------------------------
      % Horizontal bounds w = v_{i±1/2}  (thin solid)
      % ------------------------------------------------------------
      \draw (0,\vimin) -- (\vimin,\vimin);
      \draw (0,\viplus) -- (\viplus,\viplus);

      % ------------------------------------------------------------
      % Axis labels / ticks
      % ------------------------------------------------------------
      \node[left] at (0,\vimin)  {$v_{i-\frac12}$};
      \node[left] at (0,\viplus) {$v_{i+\frac12}$};
      \node[left] at (0,\vmax)   {$v_{\max}$};
      % Vertical bounds u = v_{i±1/2}  (thin solid)
      % ------------------------------------------------------------
      \draw (\vimin,-0.1) -- (\vimin,0.1);
      \draw (\viplus,-0.1) -- (\viplus,0.1);
      % ------------------------------------------------------------
      % Axis labels / ticks
      % ------------------------------------------------------------
      \node[below] at (\vimin,0)  {$v_{i-\frac12}$};
      \node[below] at (\viplus,0) {$v_{i+\frac12}$};
      \node[left]  at (0,\vmax)   {$v_{\max}$};      
    \end{tikzpicture}
    \caption{Region $D_{i}$ for the (partial) death term}
  \end{subfigure}

  \caption{Integration regions on $\Omega^{b}$ used in the DG formulation \eqref{eq:DG_breakage} for breakage.}           
  \label{fig:Ci_Di_regions}
\end{figure}
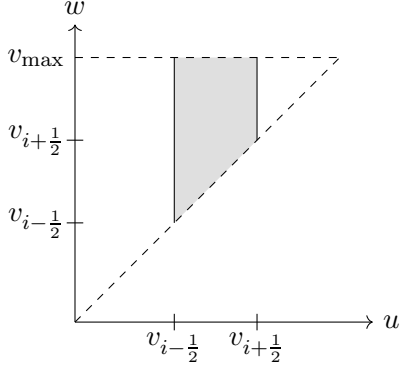
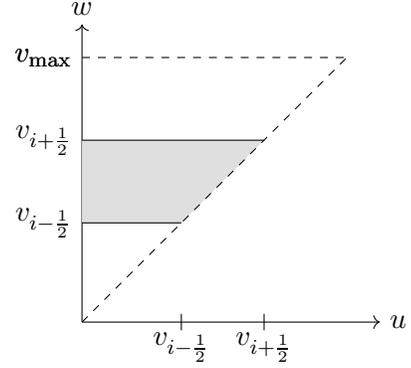

The integrals on the right-hand side of \eqref{eq:DG_breakage} are evaluated using numerical quadratures.
Similar to the aggregation term, special quadratures are needed to ensure that the birth and death terms cancel exactly when summed over the entire breakage region
\begin{equation}\label{eq:region_break}
\Omega^{b}=\{(u,w)\in\mathbb{R}^{+}\times\mathbb{R}^{+}: 0\leq u\leq w\leq v_{\max}\},
\end{equation}
in order to preserve mass conservation.

\begin{figure}[!htbp]
 \centering
  \includegraphics[width=0.4\textwidth]{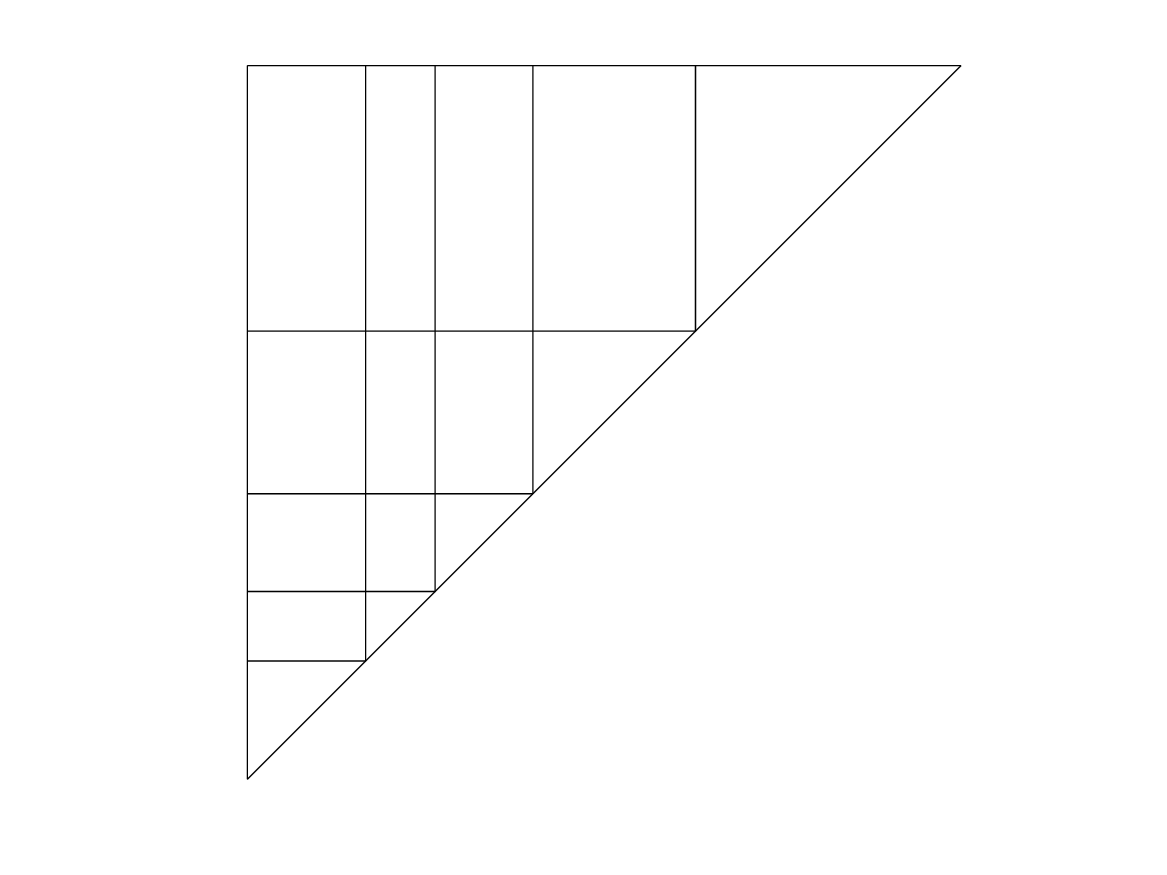}
 \caption{The coarsest common refinement $\mathcal{E}^0$ for the partitions \eqref{eq:partition_Ci} and \eqref{eq:partition_Di} of the breakge region $\Omega^{b}$.}
 \label{fig:Partition_break}
\end{figure}

We denote by $\mathcal{E}^0=\{E_{i,j}=C_{i}\cap D_{j}\}_{i,j=1}^{L}$ the coarsest common refinement of the partitions \eqref{eq:partition_Ci} and \eqref{eq:partition_Di} for the breakage region $\Omega^{b}$; see Figure \ref{fig:Partition_break} for an illustration.
There are two types of elements in $\mathcal{E}^{0}$: triangles and rectangles.
On triangular elements $E\in\mathcal{E}^0$, we adopt the quadrature rule \eqref{eq:tri_quad}.
If $E\in\mathcal{E}^0$ is a rectangle, the tensor product of Gauss-Lobatto quadrature rules is employed:
\begin{equation}\label{eq:rect_quad}
\mathlarger{\mathcal{Q}}_{E}\big[f(u,w)\big]:=|E|\sum_{\alpha=1}^{N_{G}}\sum_{\beta=1}^{N_{G}}\varpi_{\alpha}^{G}\varpi_{\beta}^{G}f(u_{\alpha}^{G},w_{\beta}^{G}),
\end{equation}
where $\mathcal{Q}_{E}$ is the quadrature operator on $E$, $f$ is a function on $E$, $|E|$ denotes the area of $E$, $N_{G}$ is the number of Gauss-Lobatto quadrature points in each dimension, $u_{\alpha}^{G}$ and $w_{\beta}^{G}$ are the one-dimensional Gauss-Lobatto quadrature points, and $\varpi_{\alpha}^{G}$ and $\varpi_{\beta}^{G}$ are the corresponding weights.

The integrals in \eqref{eq:DG_breakage} can then be evaluated as:
\begin{equation}\label{eq:break_birth_quadrature}
\iint_{C_{i}}\varphi(u)p(u,w)\gamma(w)n_h(w,t)\,dwdu \tildeminus \sum_{\substack{E \in \mathcal{E}^0 \\ E \subseteq C_i}} \mathlarger{\mathcal{Q}}_{E}\big[ \varphi(u)p(u,w)\gamma(w)n_h(w,t) \big],
\end{equation}
\begin{equation}\label{eq:break_death_quadrature1}
\iint_{D_{i}} u p(u,w)\gamma(w)n_h(w,t)\,dwdu \tildeminus \sum_{\substack{E \in \mathcal{E}^0 \\ E \subseteq D_i}} \mathlarger{\mathcal{Q}}_{E}\big[u p(u,w)\gamma(w)n_h(w,t)\big],
\end{equation}
and
\begin{equation}\label{eq:break_death_quadrature2}
\int_{I_i}\Big(\varphi(v)-v\Big)\gamma(v)n_h(v,t)dv \tildeminus \mathlarger{\mathcal{Q}}_{I_i}\big[ \big(\varphi(v)-v\big)\gamma(v)n_h(v,t) \big]
\end{equation}
The scheme defined by \eqref{eq:DG_breakage} and \eqref{eq:break_birth_quadrature} -- \eqref{eq:break_death_quadrature2} constitutes the DG method for \eqref{eq:PBE_breakage}.
\begin{prop}\label{prop:breakge_conservation}
    The total particle mass is conserved in the DG scheme \eqref{eq:DG_breakage} and \eqref{eq:break_birth_quadrature} -- \eqref{eq:break_death_quadrature2} for breakage; that is, the particle mass $M_h(t) := \int_{0}^{v_{\max}}vn_h(v,t)\,dv$ remains constant over time.
\end{prop}
\begin{proof}
We take the test function $\varphi(v)=v$ in the DG formulation \eqref{eq:DG_breakage} and sum the equations over all $i=1,2,\ldots,L$ to obtain
\begin{equation*}
\begin{split}
&\frac{\dd M_{h}(t)}{\dd t}=\sum_{i=1}^{L}\int_{I_i}(n_h)_t v\,dv\\
&=\sum_{E\in \mathcal{E}^0} \mathlarger{\mathcal{Q}}_{E}\big[ up(u,w)\gamma(w)n_h(w,t) \big] - \sum_{E\in \mathcal{E}^0} \mathlarger{\mathcal{Q}}_{E}\big[u p(u,w)\gamma(w)n_h(w,t)\big] - \sum_{i=1}^{L}\mathlarger{\mathcal{Q}}_{I_i}\big[0\big]\\
&=0.
\end{split}
\end{equation*}
\end{proof}

\subsection{Implementation}
In this subsection, we discuss the implementation details and practical considerations of the semi-discrete DG method, which combines the contributions from \eqref{eq:DG_growth_nucleation}, \eqref{eq:DG_aggregation}, and \eqref{eq:DG_breakage}, for the full form of the PBE \eqref{eq:PBE}:
Find $n_h(v,t) \in S_{h}^{k}$ such that, for $1\leq i\leq L$,
\begin{equation}\label{eq:DG_complete}
\begin{split}
\int_{I_i}(n_h)_t\varphi dv = & \int_{I_i}G(v) n_h(v,t)\varphi'(v)\,dv+G_{i-\frac12}^{-}(n_h)_{i-\frac12}^{-}\varphi_{i-\frac12}^{+}-G_{i+\frac12}^{-}(n_h)_{i+\frac12}^{-}\varphi_{i+\frac12}^{-} + \int_{I_i}S(v)\varphi(v)\,dv\\
&+\frac12\iint_{A_i}\varphi(w+u)\beta(w,u)n_h(w,t)n_h(u,t)\,dwdu - \iint_{B_i}\varphi(u)\beta(u,w)n_h(u,t)n_h(w,t)\,dwdu\\
& +\iint_{C_{i}}\varphi(u)p(u,w)\gamma(w)n_h(w,t)\,dwdu - \iint_{D_{i}} u p(u,w)\gamma(w)n_h(w,t)\,dwdu\\
&-\int_{I_i}\Big(\varphi(v)-v\Big)\gamma(v)n_h(v,t)dv \qquad \forall\,\varphi\in \mathcal{P}^{k}(I_i),
\end{split}
\end{equation}
where the integrals on the right-hand side are evaluated using the numerical quadratures \eqref{eq:growth_quadrature}, \eqref{eq:nucleation_quadrature}, \eqref{eq:aggr_birth_quadrature}, \eqref{eq:aggr_death_quadrature}, \eqref{eq:break_birth_quadrature}, \eqref{eq:break_death_quadrature1} and \eqref{eq:break_death_quadrature2} based on the common refinements $\mathcal{T}$ and $\mathcal{E}^{0}$.
Using the orthogonal basis \eqref{eq:FEMbasis}, the numerical solution of \eqref{eq:DG_complete} can be represented as
\begin{equation}
n_h(v,t)=\bigoplus_{i=1}^{L}\Big(\sum_{j=0}^{k}c_{j}^{i}(t)\phi_{j}^{i}(v)\Big),
\end{equation}
where $c_{j}^{i}$ are the coefficients to be determined.
The test function $\varphi$ in \eqref{eq:DG_complete} can be taken as any basis function $\phi_{j}^{i}$ for $j=0,1,\ldots,k$.

We define the mappings $p, q : \mathcal{T} \rightarrow \mathbb{N}^{+}$ such that $T \subseteq I_{p(T)} \times I_{q(T)}$ for every $T \in \mathcal{T}$. By abuse of notation, we also use $p, q : \mathcal{E}^0 \rightarrow \mathbb{N}^{+}$ so that $E \subseteq I_{p(E)} \times I_{q(E)}$ for every $E \in \mathcal{E}^0$.

We now define the following tensors, matrices, and vectors:
\begin{equation}\label{eq:DG_data}
\begin{split}
&\mathcal{G}^{i}_{j,m}= \mathlarger{\mathcal{Q}}_{I_i}\big[G(v)(\phi_{j}^{i})'(v)\phi_{m}^{i}(v)\big] - G_{i+\frac12}^{-}\widehat{\phi}_{j}(1)\widehat{\phi}_{m}(1),\quad j,m=0,1,\ldots,k,\, 1\leq i\leq L, \\
&\mathcal{G}^{-,i}_{j,m}= G_{i-\frac12}^{-}\widehat{\phi}_{j}(-1)\widehat{\phi}_{m}(1),\quad j,m=0,1,\ldots,k,\,1\leq i\leq L,\\
&\mathcal{N}_{m}^{i}=\mathlarger{\mathcal{Q}}_{I_i}\big[S(v)\phi_{m}^{i}(v)\big],\quad m=0,1,\ldots,k,1\leq i\leq L\\
&\mathcal{A}_{j,m,\ell}^{Birth}(T) = \frac12\mathlarger{\mathcal{Q}}_{T}\big[\phi_{j}^{i}(w+u)\beta(w,u)\phi_{m}^{p(T)}(u)\phi_{\ell}^{q(T)}(w)\big],\quad j,m,\ell=0,1,\ldots,k, \,T\in\mathcal{T},\\
&\mathcal{A}_{j,m,\ell}^{Death}(T)=\mathlarger{\mathcal{Q}}_{T}\big[\phi_{j}^{i}(u)\beta(u,w)\phi_{m}^{p(T)}(u)\phi_{\ell}^{q(T)}(w)\big],\quad j,m,\ell=0,1,\ldots,k,\, T\in\mathcal{T},\\
&\mathcal{B}_{j,m}^{Birth}(E) = \mathlarger{\mathcal{Q}}_{E}\big[\phi_{j}^{i}(u)p(u,w)\gamma (w)\phi_{m}^{q(E)}(w)\big],\quad j,m=0,1,\ldots,k,\, E\in\mathcal{\mathcal{E}}^{0},\\
&\mathcal{B}_{m}^{Death}(E)=\mathlarger{\mathcal{Q}}_{E}\big[up(u,w)\gamma(w)\phi_{m}^{q(E)}(w)\big],\quad m=0,1,\ldots,k,\, E\in\mathcal{E}^0\\
&\mathcal{B}_{j,m}^{i} = \mathlarger{\mathcal{Q}}_{I_i}\big[\Big(\phi_{j}^{i}(v)-v\Big)\gamma(v)\phi_{m}^{i}(v)\big],\quad j,m=0,1,\ldots, k,\, 1\leq i\leq L.
\end{split}
\end{equation}
These quantities remain constant throughout the simulation. Therefore, they should be precomputed once at the beginning of the program and reused in all subsequent time steps to reduce computational cost.
With these definitions, the DG formulation \eqref{eq:DG_complete} can be written in component form as:
\begin{equation}\label{eq:DG_component}
\begin{split}
\frac{\Delta v_i}{2}(c_{j}^{i})_{t} = & \sum_{m=0}^{k}\mathcal{G}_{j,m}^{i}c_{m}^{i} + \sum_{m=0}^{k}\mathcal{G}_{j,m}^{-,i}c_{m}^{i-1} + \mathcal{N}_{j}^{i}\\
&+\sum_{\substack{T\in\mathcal{T}\\ T\subseteq A_{i}}}\sum_{m,\ell=0}^{k}\mathcal{A}_{j,m,\ell}^{Birth}(T)c_{m}^{p(T)}c_{\ell}^{q(T)} - \sum_{\substack{T\in\mathcal{T}\\ T\subseteq B_{i}}}\sum_{m,\ell=0}^{k}\mathcal{A}_{j,m,\ell}^{Death}(T)c_{m}^{p(T)}c_{\ell}^{q(T)}\\
& +\sum_{\substack{E\in\mathcal{E}^{0}\\ E\subseteq C_{i}}}\sum_{m=0}^{k}\mathcal{B}_{j,m}^{Birth}(E)c_{m}^{q(E)} - \sum_{\substack{E\in\mathcal{E}^{0}\\ E\subseteq D_{i}}}\sum_{m=0}^{k}\mathcal{B}_{m}^{Death}(E)c_{m}^{q(E)}-\sum_{m=0}^{k}\mathcal{B}_{j,m}^{i} c_{m}^{i},
\end{split}
\end{equation}
for $1\leq i\leq L$ and $j=0,1, \ldots,k$.

We end this section with a few tips that may help improve computational efficiency.  
For example, in the DG scheme \eqref{eq:DG_aggregation}, we can save computational cost by using the symmetry of the integrand in $u$ and $w$: it is enough to compute the integrals in the region $u \leq w$, and the rest can be obtained by symmetry.  
In addition, when elements $T \in \mathcal{T}$ belong to the same element in $\mathcal{T}^0$, their contributions in \eqref{eq:DG_component} can be merged before use.  
Since our focus is on algorithmic design, we do not elaborate further on implementation details.

\subsection{Time integration}
The previous discussion focuses on the spatial discretization of PBE. 
The resulting system of ordinary differential equations can be written as
\begin{equation}
\mathbf{c}_t=L(\mathbf{c}),
\end{equation}
where $\mathbf{c}=(\cdots,c_{k}^{i-1},c_{0}^{i},\cdots,c_{k}^{i}, c^{i+1}_{0},\cdots)^T$ denotes the vector of degrees of freedom of $n_h(t)$ associated with the chosen set of basis functions \eqref{eq:FEMbasis}, and $L(\mathbf{c})$ is a quadratic function.
In this paper, we adopt the strong-stability-preserving Runge-Kutta (SSP-RK) methods \cite{shu1988efficient, shu1988total, gottlieb2009high, gottlieb2001strong}, which are convex combinations of the forward Euler method, to advance the equation in time.
The third-order SSP-RK scheme \cite{shu1988efficient} is given by
\begin{equation}\label{eq:SSP-RK}
\begin{split}
\mathbf{c}^{(1)}&=\mathbf{c}^{m}+\Delta t L(\mathbf{c}^{m}),\\
\mathbf{c}^{(2)}&=\frac34\mathbf{c}^{m}+\frac14\big(\mathbf{c}^{(1)}+\Delta t L(\mathbf{c}^{(1)})\big),\\
\mathbf{c}^{m+1}&=\frac13\mathbf{c}^{m}+\frac23\big(\mathbf{c}^{(2)}+\Delta t L(\mathbf{c}^{(2)})\big),
\end{split}
\end{equation}
where $\Delta t$ denotes the time step size, and $\mathbf{c}^{m}$ approximates $\mathbf{c}(t^{m})$, with $t^{m} = m\Delta t$ for $m \geq 1$.
If the PBE is stiff due to multi-scale dynamics among different processes, one may also consider using implicit \cite{ketcheson2009optimal}, implicit-explicit \cite{ascher1997implicit}, or exponential time-differencing \cite{xu2025stability} Runge-Kutta methods.

\section{Positivity-Preservation}\label{sec:PPDG}
In the previous section, we have established a DG scheme for the population balance equation with the conservation property, a traditional advantage of DG-type methods that is important for ensuring the convergence to the physical solution.
However, conservation is not the only concern for numerical schemes. It is also important to preserve the positivity of the solution, as a negative number density is not only physically unacceptable but may also lead to numerical instability.

Unlike the work in \cite{liu2019high}, which develops a positivity-preserving DG method for the aggregation-breakage PBE based on an equivalent conservation law formulation for the mass density $m(v,t):=vn(v,t)$, here we establish our positivity-preserving method directly from the standard formulation \eqref{eq:PBE} for the number density $n(v,t)$.
This brings an extra difficulty. The mass (i.e., the first moment of the numerical solution $n_h$) is the physical quantity we aim to preserve during the limiting procedure, so the standard Zhang–Shu limiter for cell averages (zeroth moment) cannot be directly applied. Therefore, we propose a positivity-preserving limiter that conserves an arbitrary moment without destroying accuracy. 
%The limiter is carefully designed, and its properties are proven.

Before we begin, we note that all discussions in this section are based on the forward Euler discretization, which is sufficient since the SSP-RK method \eqref{eq:SSP-RK} is a convex combination of forward Euler stages.

\subsection{Positivity of moments}

We present the following positivity results for moments of numerical solutions of the DG schemes: \eqref{eq:DG_aggregation} for the aggregation PBE, \eqref{eq:DG_breakage} for the breakage PBE, and \eqref{eq:DG_growth_nucleation} for the growth and nucleation PBE.

\begin{thm}\label{thm:pos_moments}
Suppose the numerical solution at time level $t^m$ is nonnegative, i.e.,  $n_{h}^{m}(v) \geq 0$. Then, after one forward Euler step:

\begin{itemize}
    \item The solutions of the aggregation scheme \eqref{eq:DG_aggregation} and the breakage scheme \eqref{eq:DG_breakage} have a nonnegative first moment on each cell $I_i$; that is,
    \[
    \int_{I_i} v n_h^{m+1}(v)\,dv \geq 0,
    \]
    provided that the time-step satisfies
    \[
    \Delta t \leq \frac{1}{v_{\max} \|\beta(u,w)n_{h}(w)\|_{L^{\infty}(\Omega^a)}}
    \quad \text{and} \quad
    \Delta t \leq \frac{1}{\|up(u,w)\gamma(w)\|_{L^\infty(\Omega^b)}},
    \]
    respectively.

    \item The solution of the growth and nucleation scheme \eqref{eq:DG_growth_nucleation} has a nonnegative $s$-th moment on each cell $I_i$ for all $s = 0, 1, \ldots, k$; that is,
    \[
    \int_{I_i} v^{s} n_h^{m+1}(v)\,dv \geq 0,
    \]
    under the time-step condition
    \[
    \Delta t \leq \frac{\varpi_{N_G}^{G}}{\|G\|_{L^{\infty}(\mathcal{D})}} \Delta v.
    \]
\end{itemize}
\end{thm}

\begin{proof}
We drop the superscript $m$ for brevity in the derivations below.

First, we take the test function $\varphi(v) = v$ in the aggregation scheme \eqref{eq:DG_aggregation} with forward Euler discretization to obtain
\begin{equation*}
\begin{split}
\int_{I_i} v n_h^{m+1}(v)\,dv 
=&\; \mathcal{Q}_{I_i}\big[v n_h(v)\big] 
+ \Delta t \sum_{\substack{T\in\mathcal{T}\\ T \subseteq A_i}} \mathcal{Q}_{T}\big[u \beta(w,u) n_h(w) n_h(u)\big] \\
&\; - \Delta t \sum_{\substack{T\in\mathcal{T}\\ T \subseteq B_i}} \mathcal{Q}_{T}\big[u \beta(u,w) n_h(u) n_h(w)\big] \\
\geq&\; \mathcal{Q}_{I_i}\big[v n_h(v)\big] 
- \Delta t \sum_{\substack{T\in\mathcal{T}\\ T \subseteq B_i \setminus A_i}} \mathcal{Q}_{T}\big[u \beta(u,w) n_h(u) n_h(w)\big] \\
=&\; \sum_{\substack{T\in\mathcal{T}\\ T \subseteq B_i \setminus A_i}} \mathcal{Q}_{T}\left[ u n_h(u) \left( \frac{1}{v_{\max} - v_{i+\frac{1}{2}}} - \Delta t \beta(u,w) n_h(w) \right) \right] \\
\geq&\; 0,
\end{split}
\end{equation*}
under the time-step condition $\Delta t\leq \frac{1}{v_{\max}||\beta(u,w)n_{h}(w)||_{L^{\infty}(\Omega^a)}}$.
Here, we have used the fact that the quadrature rules $\mathcal{Q}_{I_i}$ and $\mathcal{Q}_T$ are exact for the polynomial $v n_h(v)$.

Next, we take $\varphi(v) = v$ in the breakage scheme \eqref{eq:DG_breakage} with forward Euler discretization to obtain
\begin{equation*}
\begin{split}
\int_{I_i} v n_h^{m+1}(v)\,dv 
=&\; \mathcal{Q}_{I_i}\big[v n_h(v)\big] 
+ \Delta t \sum_{\substack{E\in\mathcal{E}^{0}\\ E \subseteq C_i}} \mathcal{Q}_{E}\big[u p(u,w) \gamma(w) n_h(w)\big] \\
&\; - \Delta t \sum_{\substack{E\in\mathcal{E}^{0}\\ E \subseteq D_i}} \mathcal{Q}_{E}\big[u p(u,w) \gamma(w) n_h(w)\big] \\
\geq&\; \mathcal{Q}_{I_i}\big[v n_h(v)\big] 
- \Delta t \sum_{\substack{E\in\mathcal{E}^{0}\\ E \subseteq D_i \setminus C_i}} \mathcal{Q}_{E}\big[u p(u,w) \gamma(w) n_h(w)\big] \\
=&\; \sum_{\substack{E\in\mathcal{E}^{0}\\ E \subseteq D_i \setminus C_i}} \mathcal{Q}_{E}\left[ \left( \frac{w}{v_{i-\frac{1}{2}}} - \Delta t u p(u,w) \gamma(w) \right) n_h(w) \right] \\
\geq&\; \sum_{\substack{E\in\mathcal{E}^{0}\\ E \subseteq D_i \setminus C_i}} \mathcal{Q}_{E}\left[ \left( 1 - \Delta t u p(u,w) \gamma(w) \right) n_h(w) \right] \\
\geq&\; 0,
\end{split}
\end{equation*}
under the time-step condition $\Delta t\leq \frac{1}{||up(u,w)\gamma(w)||_{L^\infty(\Omega^b)}}$.
Here again, the quadrature rules  $\mathcal{Q}_{I_i}$ and $\mathcal{Q}_E$ are exact for the polynomial $v n_h(v)$.

Finally, we take the test function $\varphi(v) = v^s$ in the growth and nucleation scheme \eqref{eq:DG_growth_nucleation}, yielding
\begin{equation*}
\begin{split}
\int_{I_i} v^s n_h^{m+1}(v)\,dv 
\geq&\; \mathcal{Q}_{I_i}\big[v^s n_h(v)\big] 
+ \Delta t \left( \mathcal{Q}_{I_i}\big[ s v^{s-1} G(v) n_h(v) \big] 
+ G_{i-\frac{1}{2}}^{-} (n_h)_{i-\frac{1}{2}}^{-} v_{i-\frac{1}{2}}^s 
- G_{i+\frac{1}{2}}^{-} (n_h)_{i+\frac{1}{2}}^{-} v_{i+\frac{1}{2}}^s \right) \\
=&\; \Delta v_i \sum_{\alpha=1}^{N_G - 1} \varpi_{\alpha}^{G} (v_{\alpha}^{G})^s n_h(v_{\alpha}^{G}) 
+ \Delta t \mathcal{Q}_{I_i}\big[ s v^{s-1} G(v) n_h(v) \big] \\
&\; + \Delta t G_{i-\frac{1}{2}}^{-} (n_h)_{i-\frac{1}{2}}^{-} v_{i-\frac{1}{2}}^s 
+ \left( \Delta v_i \varpi_{N_G}^{G} - \Delta t G_{i+\frac{1}{2}}^{-} \right) (n_h)_{i+\frac{1}{2}}^{-} v_{i+\frac{1}{2}}^s \\
\geq&\; 0,
\end{split}
\end{equation*}
under the time-step condition $\Delta t\leq \frac{\varpi_{N_G}^{G}}{||G||_{L^{\infty}(\mathcal{D})}}\Delta v$.
Here, we have used the exactness of the quadrature rules $\mathcal{Q}_{I_i}$ and $\mathcal{Q}_{E}$ for the polynomial $v^s n_h(v)$, and the fact that 
$v^{s} S(v) \geq 0$.
\end{proof}

Theorem \ref{thm:pos_moments} establishes the positivity of the first moment (local particle mass) in the aggregation and breakage processes, as well as the positivity of arbitrary moments (e.g., local particle number, local particle mass) in the growth and nucleation processes, under certain CFL conditions, provided that the numerical solution for the number density remains nonnegative over the entire domain at the current time stage.
If the PBE involves multiple processes, such as the aggregation-breakage PBE studied in \cite{liu2019high}, the positivity of the first moment on each cell $I_i$ still holds, possibly under more stringent CFL conditions. This can be proved by decomposing the moment.

\begin{rem}
We note that the CFL conditions established here serve only as a theoretical guarantee.
In practice, we use standard CFL conditions and apply a fallback strategy: whenever negative moments are detected at a time step, the computation is rewound and retried using a halved time step size.
The theorem established here ensures that the number of halvings is finite.
\end{rem}

\subsection{Scaling limiter}
To close the cycle of positivity-preserving algorithm, we need to apply appropriate scaling limiters to enforce the nonnegativity of the solution across the entire domain at the next time level.

For PBEs involving only growth and nucleation, we adopt the classical Zhang–Shu limiter to preserve the cell averages of the solution, thereby ensuring particle number conservation. This technique, introduced in the pioneering work on positivity-preserving DG methods for hyperbolic conservation laws \cite{zhang2010maximum}, remains standard. For extensions to Lax–Wendroff time integration and implicit time marching, we refer the reader to \cite{xu2022third} and \cite{xu2023conservation}, respectively. 
Since this case does not require any special treatment beyond these methods, we omit further discussion.

In contrast to most existing works based on the Zhang–Shu framework \cite{zhang2010maximum}, our focus for the aggregation–breakage PBE is on preserving the \textit{first moment} rather than the cell average, to ensure mass conservation. 
For this reason, the standard Zhang–Shu limiter cannot be directly applied to aggregation–breakage problems. One existing strategy is to reformulate the equation in terms of the mass density  $m(v, t) = v n(v, t)$, resulting in a conservation law with a nonlocal flux. 
A positivity-preserving scheme can then be constructed for $m(v,t)$, and the standard scaling limiter applied to maintain the cell average of the local mass.
However, this formulation alters the standard hyperbolic conservation structure when a growth term is present, as the growth term no longer appears in conservative form and the conventional treatment must also be modified accordingly. 
In our approach, however, we approximate the number density $n(v, t)$ directly using the standard form of the PBE. 
Accordingly, we develop a new scaling limiter that preserves an arbitrary moment, enabling us to enforce positivity while maintaining the desired conservation property.

The situation becomes more subtle when all four processes, growth, nucleation, aggregation, and breakage, are present. 
In this case, neither the total particle number nor total mass is constant. 
Two strategies are possible. 
One is to apply a limiter that preserves local mass on each cell, based on the belief that mass conservation remains the more fundamental mechanism. 
The other is to use Strang splitting in time \cite{liu2023positivity}, solving the growth–nucleation and aggregation–breakage subproblems sequentially, with limiters designed to preserve the cell average and first moment, respectively. 
The downside of this splitting approach is that it limits the temporal accuracy to at most second order.
To stay focused on the main theme of this paper, we consider only the \textit{mass-conservative limiter} for the aggregation–breakage PBE, which is also the focus of \cite{liu2019high}.
The moment-conservative limiter is defined as follows.

Let $n_{h}\in\mathcal{P}^{k}(I_i)$ be such that its $s$-th moment is nonnegative:
\begin{equation}
\int_{I_i}v^{s}n_{h}(v)dv\geq0,
\end{equation}
where $s$ is an integer between $0$ and $k$.
If $n_h<0$ at any point in the cell $I_i$, the modified solution $\widetilde{n}_h$ is defined by
\begin{equation}\label{eq:limiter_mass}
\widetilde{n}_h=\theta(n_{h}+m),
\end{equation}
where 
\begin{equation}
m:=-\min_{v\in I_i}n_{h}(v)>0,\quad \theta:=\frac{\int_{I_i}v^{s}n_{h}(v)\,dv}{\int_{I_i}v^{s}n_{h}(v)\,dv+m\int_{I_i}v^{s}\,dv}\in [0,1].
\end{equation}
Otherwise, the solution remains unchanged

When $s = 0$, the limiter \eqref{eq:limiter_mass} reduces to the standard positivity-preserving limiter from \cite{liu1996nonoscillatory, zhang2010maximum}, which preserves the cell averages of the numerical solution for the growth–nucleation PBE.  
Here, we adopt the limiter with $s = 1$ to conserve mass in the aggregation–breakage PBE.
If the internal coordinate of the number density is the particle length $l$ rather than the volume $v$ in certain forms of the PBE, then the moments  $\int_{0}^{\infty} l^s n(l,t)\,dl$ represent the total particle number, total particle area, and total particle mass for $s = 0$, $2$, and $3$, respectively~\cite{marchisio2003quadrature}.  
In this case, the limiter with $s = 3$ should be applied to ensure mass conservation.

\begin{thm}\label{thm:limiter_mass}
The modified polynomial $\widetilde{n}_h$ in \eqref{eq:limiter_mass} satisfies the following properties:
\begin{enumerate}
    \item \emph{Positivity:} $\widetilde{n}_{h}(v)\geq0$ for all $v\in I_i$.
    \item \emph{Conservation of the $s$-th moment:}  $\int_{I_i}v^{s}\widetilde{n}_h(v)\,dv=\int_{I_i}v^{s}n_{h}(v)\,dv$.
    \item \emph{Accuracy:} $||n_{h}-\widetilde{n}_{h}||_{L^{\infty}(I_i)}\leq C(k,s)m$, where the constant $C(k,s)$ depends only on the polynomial degree $k$ and moment $s$. 
    \item \emph{No overshoot:} $\max_{v\in I_i}\widetilde{n}_{h}(v)\leq \max_{v\in I_i}n_{h}(v)$.
\end{enumerate}    
\end{thm}
\begin{proof}
Property ($1$) follows directly from the definition, and Property ($2$) can be verified by a straightforward calculation.

To prove Property ($3$), we define the intermediate function $\widehat{n}_{h}:=n_{h}+m$, so that $\widetilde{n}_{h}=\theta\widehat{n}_h$.
Then,
\begin{equation}\label{eq:err_total}
||n_{h}-\widetilde{n}_{h}||_{L^\infty(I_i)}\leq ||n_{h}-\widehat{n}_{h}||_{L^\infty(I_i)}+||\widehat{n}_{h}-\widetilde{n}_{h}||_{L^\infty(I_i)}=:e_{1}+e_{2}.
\end{equation}
Clearly, $e_1=m$. 

To estimate $e_2$, we use the $(k+1)$-point Gauss-Legendre quadrature rule on the reference interval $(-1,1)$ with nodes $r_{\alpha}$ and weights $\varpi_{\alpha}$, for $\alpha=0,1,\ldots, k$. 
We map the reference nodes to the physical cell $I_i$ by
\begin{equation*}
u_{\alpha}:=v_{i}+\frac12 r_{\alpha}\Delta v_{i},\quad \alpha=0,1,\ldots, k.
\end{equation*}
Recall that this quadrature rule is exact for polynomials in $\mathcal{P}^{2k+1}$.
Moreover, the nodes lie strictly within the open interval, and all weights are strictly positive.

At each Gauss-Legendre node $u_{\alpha}$, we estimate:
\begin{equation*}
\begin{split}
\big|\widehat{n}_h(u_{\alpha})-\widetilde{n}_h(u_{\alpha})\big|=&(1-\theta)\widehat{n}_{h}(u_{\alpha})\\
=&\frac{m\int_{I_i}v^{s}\,dv}{\int_{I_i}v^{s}n_{h}(v)\,dv+m\int_{I_i}v^{s}\,dv}\cdot\widehat{n}_{h}(u_\alpha)\\
=&\frac{\widehat{n}_{h}(u_\alpha)}{\int_{I_i}v^{s}\widehat{n}_{h}(v)\,dv}\cdot\frac{m}{s+1}\cdot(v_{i+\frac12}^{s+1}-v_{i-\frac12}^{s+1})\\
=&\frac{\widehat{n}_{h}(u_{\alpha})}{\sum_{\beta=0}^{k}\varpi_{\beta}u_{\beta}^{s}\widehat{n}_{h}(u_{\beta})}\cdot\frac{m}{s+1}\cdot\frac{v_{i+\frac12}^{s+1}-v_{i-\frac12}^{s+1}}{v_{i+\frac12}-v_{i-\frac12}}\\
\leq& \frac{\widehat{n}_{h}(u_\alpha)}{\varpi_{\alpha}u_{\alpha}^{s}\widehat{n}_{h}(u_{\alpha})}\cdot\frac{m}{s+1}\cdot\sum_{p=0}^{s}v_{i-\frac12}^{s-p}v_{i+\frac12}^{p} \\
\leq & \varpi_{\alpha}^{-1}\left(\frac{v_{i+\frac12}}{u_{\alpha}}\right)^{s}m\\
=&\varpi_{\alpha}^{-1}\left(\frac{v_{i-\frac12}+\Delta v_{i}}{v_{i-\frac12}+\frac12(1+r_{\alpha})\Delta v_{i}}\right)^{s}m\\
\leq& \varpi_{\alpha}^{-1}\left(\frac{1}{\frac12(1+r_{\alpha})}\right)^{s}m\\
\leq& {C}_1(k,s)m,
\end{split}
\end{equation*}
where ${C}_1(k,s)$ is a constant determined by the quadrature rule and the moment $s$.

Next, we write the difference using the Lagrange basis functions $\ell_{\alpha}(u)$:
\begin{equation*}
\widehat{n}_h(u)-\widetilde{n}_{h}(u)=\sum_{\alpha=0}^{k}\big(\widehat{n}_h(u_{\alpha})-\widetilde{n}_{h}(u_{\alpha})\big)\ell_{\alpha}(u),
\end{equation*}
and thus,
\begin{equation*}
\begin{split}
e_{2}\leq&\sum_{\alpha=0}^{k}\big|\widehat{n}_h(u_{\alpha})-\widetilde{n}_{h}(u_{\alpha})\big|\cdot||\ell_{\alpha}(u)||_{L^\infty(I_i)}\\
\leq&\sum_{\alpha=0}^{k}||\ell_{\alpha}(u)||_{L^\infty(I_i)}\cdot C_{1}(k,s) m\\
=:&C_{2}(k,s)m.
\end{split}
\end{equation*}

Combining with \eqref{eq:err_total}, we obtain:
\begin{equation*}\label{prf:e1e2}
||n_{h}-\widetilde{n}_{h}||_{L^\infty(I_i)}\leq e_1+e_2=(1+C_{2}(k,s))m=:C(k,s)m.
\end{equation*}

Finally, for Property (4), using the conservation of the $s$-th moment:
\begin{equation*}
\int_{I_i}v^{s}\left(n_{h}(v)-\widetilde{n}_{h}(v)\right)\,dv=0,
\end{equation*}
we deduce
\begin{equation*}
\begin{split}
0\leq\max_{v\in I_i}\left( n_{h}(v)-\widetilde{n}_{h}(v) \right)=&\max_{v\in I_i}\left( \widehat{n}_{h}(v)-m-\theta\widehat{n}_{h}(v)\right)\\
=&\max_{v\in I_i} \widehat{n}_{h}(v)-m - \max_{v\in I_i} \theta\widehat{n}_{h}(v)\\
=&\max_{v\in I_i} n_h(v) - \max_{v\in I_i}\widetilde{n}_{h}(v),
\end{split}
\end{equation*}
where the second-to-last equality follows from $\widehat{n}_h \geq 0$ and $\theta \in [0,1]$.

This concludes the proof.
\end{proof}

In practice, the value of $m$ in the limiter can be computed by evaluating  $n_h$ over quadrature points.
The positivity-preserving DG scheme, together with the scaling limiter, constitutes a complete forward Euler stage, which serves as the building block for high-order SSP-RK time integration methods.
We will validate the theoretical results through numerical tests in Section~\ref{sec:tests}.

\section{Numerical Experiments}\label{sec:tests}
In this section, we present a series of numerical experiments to verify the accuracy, conservation, and positivity-preserving properties of the proposed DG method.
The section is divided into three parts.
In the first part, we test simple benchmark problems with analytical solutions to validate the scheme.
In the second part, we consider more complex and practical cases to demonstrate the scheme’s performance.
In the final part, we focus on challenging problems involving positivity-preserving issues to illustrate the enhanced stability offered by our approach.
Because the aggregation process presents the greatest numerical challenges \cite{hasseine2015adomian}, we focus on this case more extensively in our tests.
To save space, we use the $\mathcal{P}^2$-DG scheme and SSP-RK3 time integration method in most tests.
However, the DG method also performs well with other polynomial spaces and time integration methods.

\subsection{Benchmarks with analytical solutions}
\begin{exmp}\label{ex:Single_process}
\textbf{Single process}
\end{exmp}
In this example, we test our algorithm on PBEs involving single processes, including pure aggregation \eqref{eq:PBE_aggregation} with either constant or additive kernels \cite{scott1965analytic}, and pure breakage \eqref{eq:PBE_breakage} with a uniform daughter distribution and a linear breakage rate (the UL model) \cite{ziff1985kinetics}.

\begin{itemize}
    \item {Case I}: Aggregation with a constant kernel
    \begin{equation}
        \beta(u,w)=1,\quad n_1(v,t)=\frac{4}{(t+2)^2 v_0}\exp\left(-\frac{2v}{(t+2)v_0}\right).
    \end{equation}
    \item {Case II}: Aggregation with an additive kernel
    \begin{equation}
        \beta(u,w)=u+w,\quad n_2(v,t)=\frac{1-T}{v\sqrt{T}}\exp\left(-(1+T)\frac{v}{v_0}\right)I_1\left(2\sqrt{T}\frac{v}{v_0}\right),
    \end{equation}
    where $T=1-e^{-v_0t}$ and $I_1$ is the modified Bessel function of the first kind (order $1$).
    % \item Case III (multiplicative kernel):
    %     \begin{equation}
    %     \beta_3(u,w)=uw, \quad n_3(v,t)=\frac{1}{2\sqrt{\pi}v_0^{2}}\theta^{-5/2}\eta^{-3/2}\exp\big(-\frac{\eta}{4\theta^2}\big),
    % \end{equation}
    % where $\theta=1-v_0t$ and $\eta=v/v_0$. 
    % The solution has finite second and higher moments only in the pre-gel regime $0\leq t<t_g:=1/v_0$, where $t_g$ denotes the gelation (critical) time.
    \item {Case III}: Breakage with a uniform daughter distribution and linear rate
    \begin{equation}\label{eq:breakage_setting}
        p(u,w)=\frac{2}{w},\quad \gamma(w)=w,\quad n_3(v,t)=\frac{(1+v_0 t)^2}{v_0}\exp\left(-(1+v_0 t)\frac{v}{v_0}\right).
    \end{equation}    
\end{itemize}
All three cases share the same initial condition:
\begin{equation}\label{eq:exp_init_cond}
n_i(v,0)=\frac{1}{v_0}\exp\!\left(-\frac{v}{v_0}\right), \qquad i=1,2,3.
\end{equation}
Throughout this study we set $v_0 = 0.2$, which serves as the characteristic particle volume.

The computational setup is as follows.  
We truncate the computational domain to $\mathcal{D} = [0, v_{\max}]$ with $v_{\max} = 10$ to ensure that the tails of the distribution beyond the boundary are negligible.  
Since the number densities are concentrated near the origin and decay rapidly, we employ a nonuniform mesh with $v_{i+\frac{1}{2}} = v_{\max} \left(\frac{i}{L}\right)^3, \quad i = 0, 1, \ldots, L,$ with $L = 15$ cells. 
The CFL condition is taken as $\Delta t = 10 \Delta v$.
Simulations are performed for all three cases up to final time $t = 1$.  
In Figure \ref{fig:Single_process}, we display the numerical number densities at the cell centers, along with the corresponding exact solutions for comparison.  
We also track the deviation of the total mass over time, which is plotted in the same figure.  
The numerical solutions show excellent agreement with the exact ones, and the total mass deviation remains at the round-off level, on the order of $10^{-16}$.
To assess the convergence behavior of our method, we report the numerical errors and observed orders of accuracy over a sequence of refined meshes. 
The mesh at level~0 corresponds to the one used in Figure~\ref{fig:Single_process}.
Each finer mesh at level $i$ is obtained by equally splitting each cell of the level $i-1$ mesh into two, for $i = 1, 2, 3$.
The results are presented in Table \ref{tab:Single_process}, where we clearly observe the expected third-order convergence of the scheme.

\begin{figure}[!htbp]
 \centering
 \begin{subfigure}[b]{0.32\textwidth}
  \includegraphics[width=\textwidth]{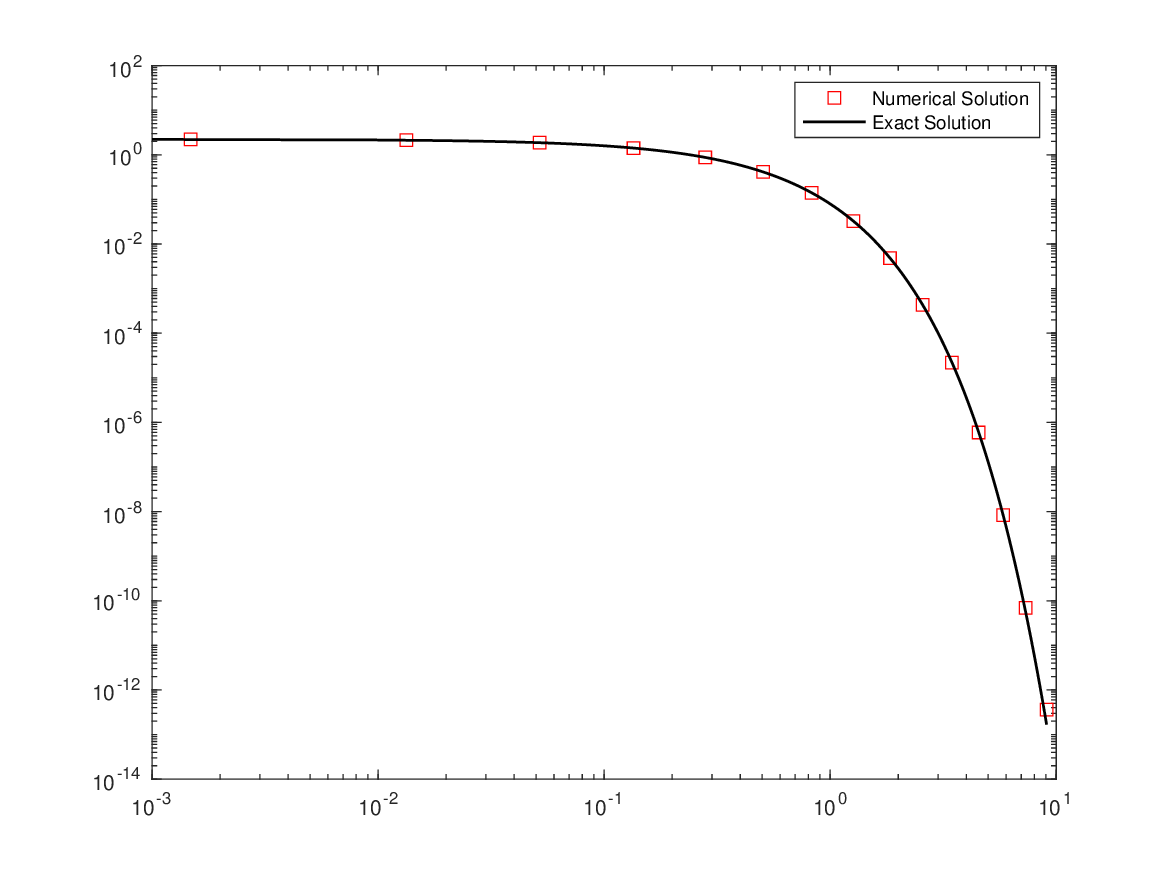}
  \caption{Case I: $n_h(v,t)$ at $t=1$}
 \end{subfigure}
 \begin{subfigure}[b]{0.32\textwidth}
  \includegraphics[width=\textwidth]{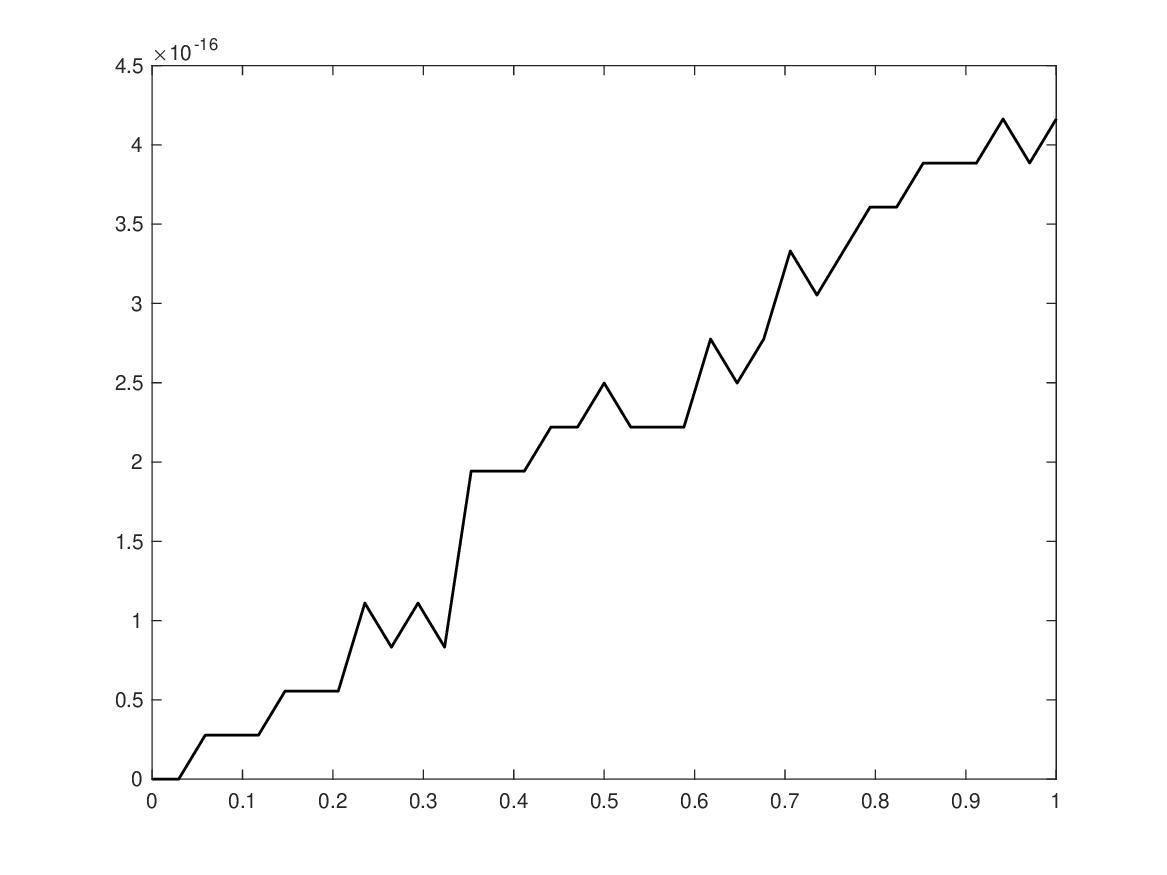}
  \caption{Case I: Total mass deviation}
 \end{subfigure}

 \begin{subfigure}[b]{0.32\textwidth}
  \includegraphics[width=\textwidth]{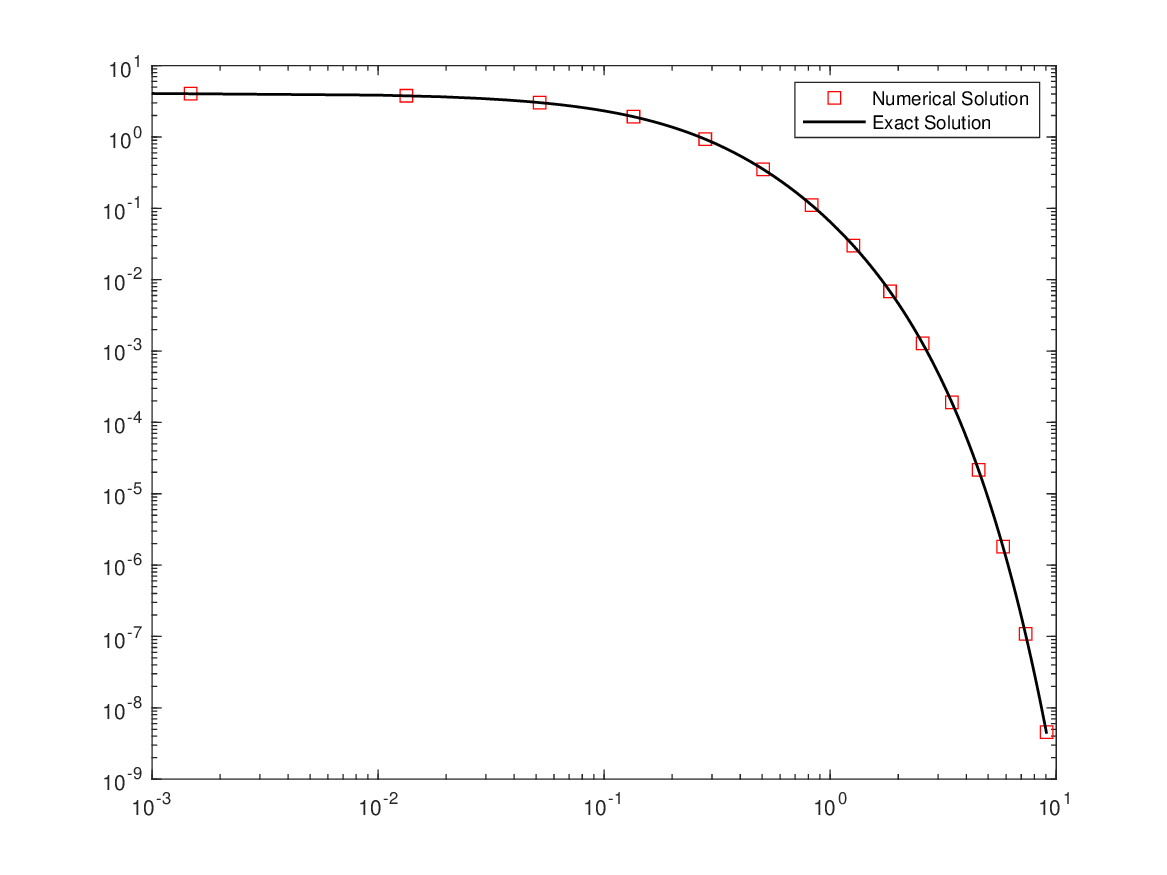}
  \caption{Case II: $n_h(v,t)$ at $t=1$}
 \end{subfigure}
 \begin{subfigure}[b]{0.32\textwidth}
  \includegraphics[width=\textwidth]{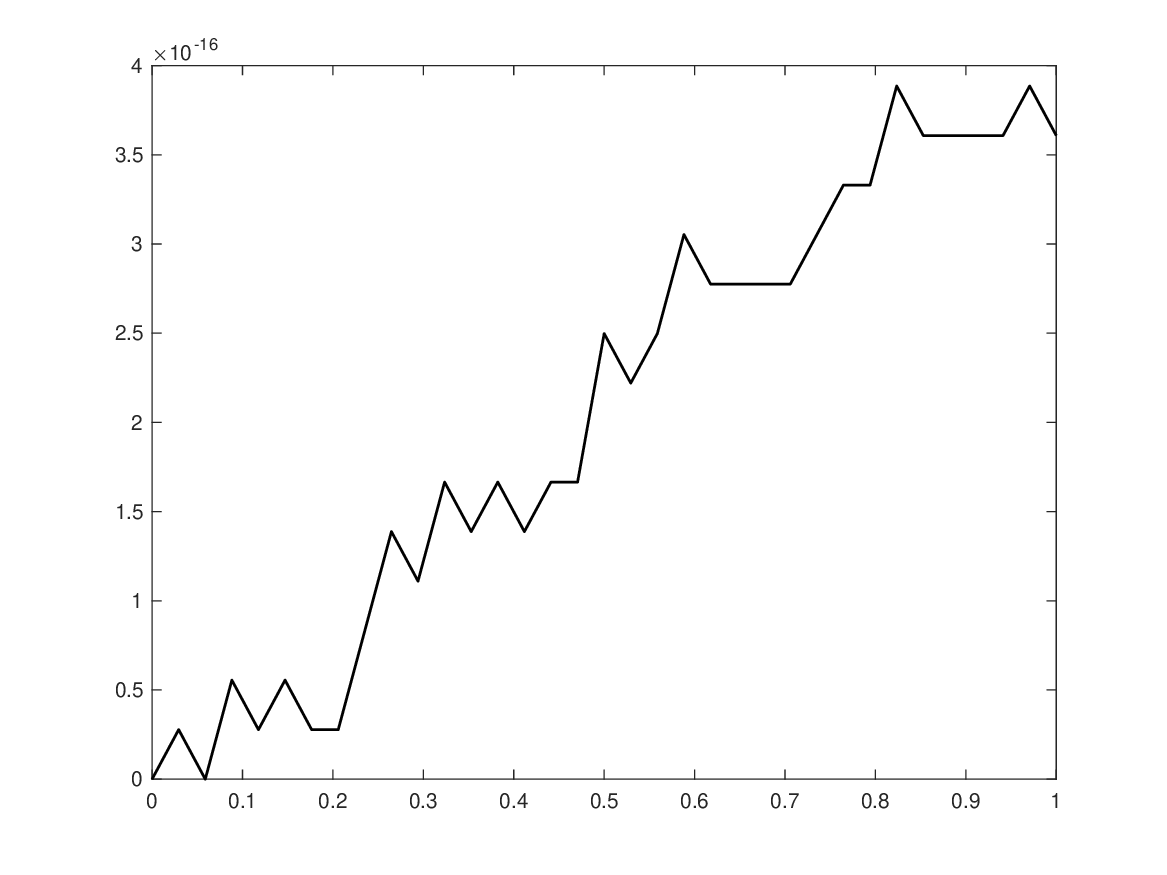}
  \caption{Case II: Total mass deviation}
 \end{subfigure}

 \begin{subfigure}[b]{0.32\textwidth}
  \includegraphics[width=\textwidth]{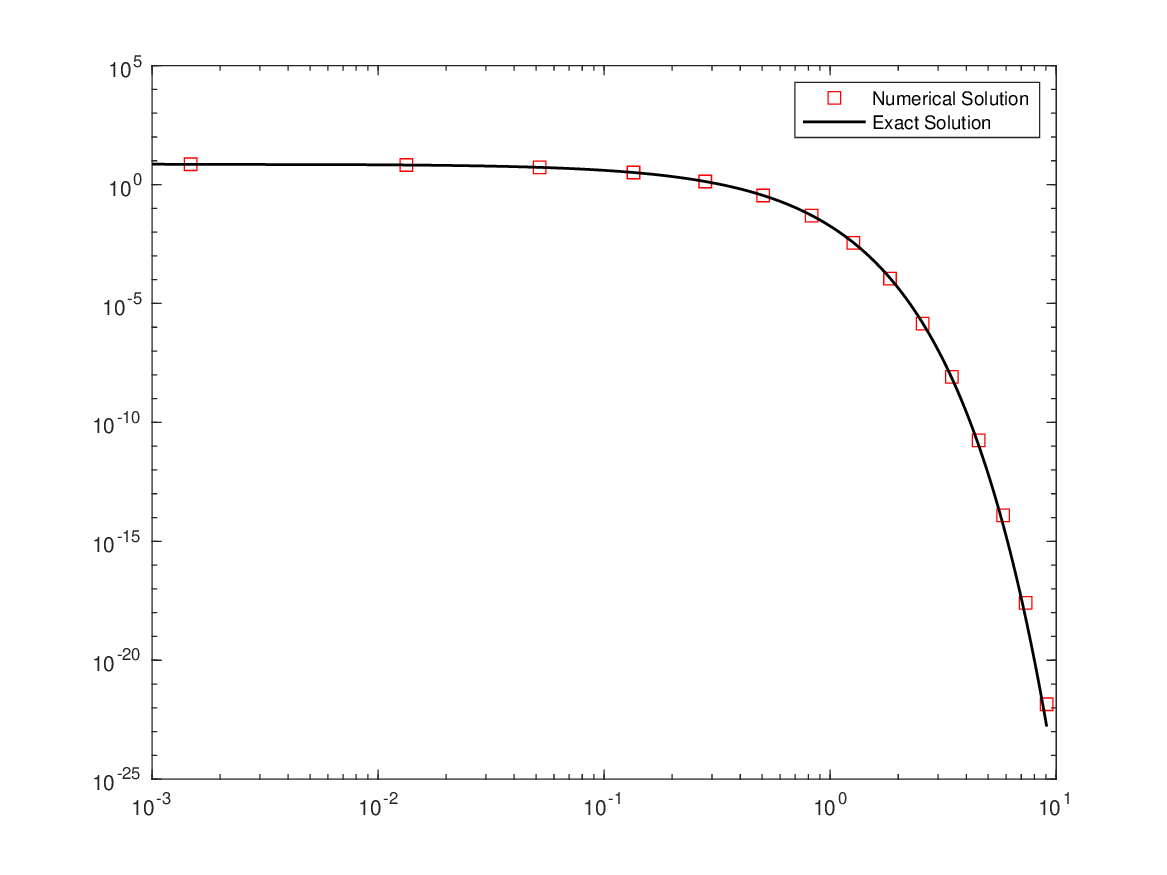}
  \caption{Case III: $n_h(v,t)$ at $t=1$}
 \end{subfigure}
 \begin{subfigure}[b]{0.32\textwidth}
  \includegraphics[width=\textwidth]{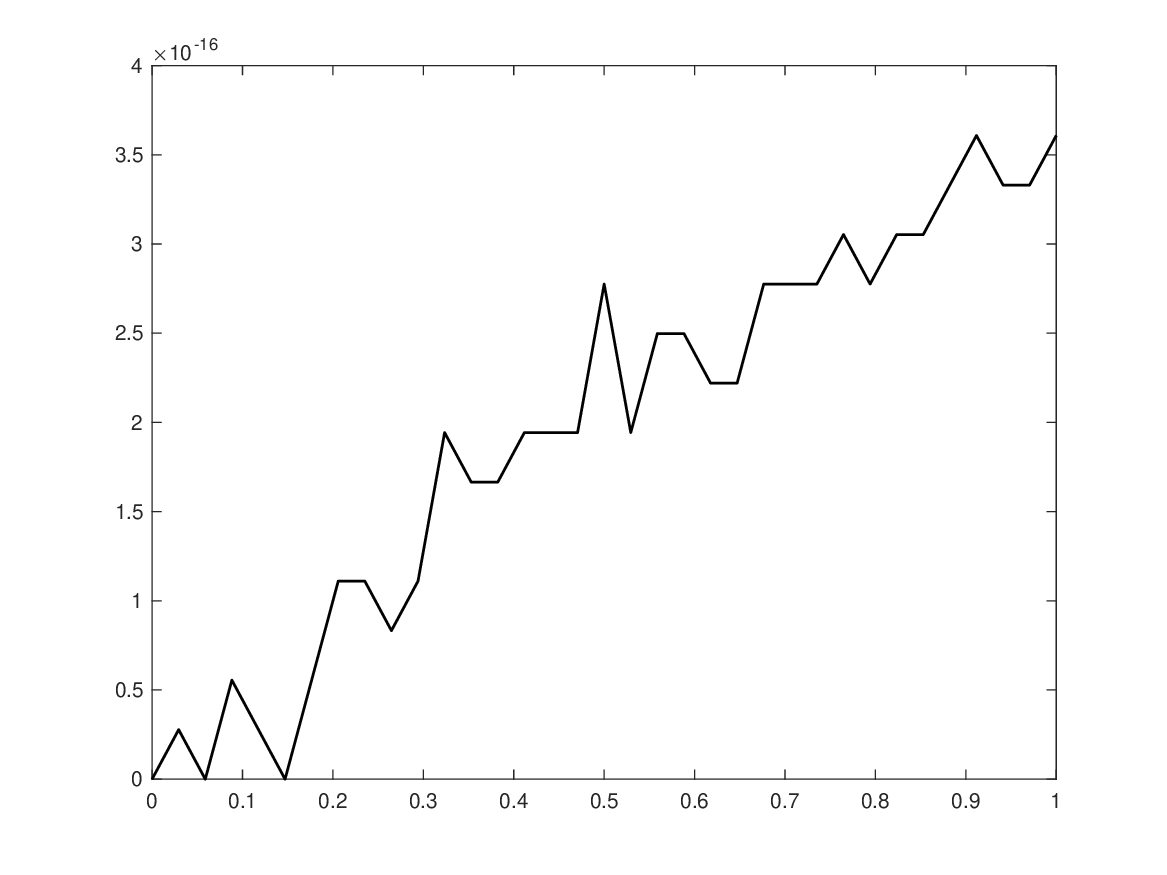}
  \caption{Case III: Total mass deviation}
 \end{subfigure} 
\caption{\textbf{Example~\ref{ex:Single_process}: Single process.} Left: numerical and exact number densities; right: deviation of the total mass over time. 
A nonuniform mesh with $L = 15$ cells is used.}
\label{fig:Single_process}
\end{figure}

\begin{table}[htbp]
  \centering
  \renewcommand{\arraystretch}{1.2}   % vertical spacing
  \setlength{\tabcolsep}{5pt}         % horizontal padding
  \begin{tabular}{c*{6}{c}}
    \toprule
    %\midrule[1.0pt]
    \multicolumn{7}{c}{{Case I}}\\
    \midrule
    Level & $||n-n_h||_{L^1}$ & Order & $||n-n_h||_{L^2}$ & Order & $||n-n_h||_{L^{\infty}}$ & Order\\
    \midrule
      0 & $9.53\times10^{-4}$ & --   & $8.10\times10^{-4}$ & --   & $1.35\times10^{-3}$ & --  \\
      1 & $1.24\times10^{-4}$ & $2.95$ & $1.08\times10^{-4}$ & $2.90$ & $2.03\times10^{-4}$ & $2.73$\\
      2 & $1.57\times10^{-5}$ & $2.98$ & $1.39\times10^{-5}$ & $2.96$ & $2.86\times10^{-5}$ & $2.83$\\
      3 & $1.96\times10^{-6}$ & $3.00$ & $1.74\times10^{-6}$ & $3.00$ & $3.77\times10^{-6}$ & $2.92$\\
    \midrule
    \multicolumn{7}{c}{{Case II}}\\
    \midrule
    Level & $||n-n_h||_{L^{1}}$ & Order & $||n-n_h||_{L^{2}}$ & Order & $||n-n_h||_{L^{\infty}}$ & Order\\
    \midrule
      0 & $1.75\times10^{-3}$ & -- & $1.80\times10^{-3}$ & -- & $3.95\times10^{-3}$ & -- \\
      1 & $2.31\times10^{-4}$ & $2.92$ & $2.50\times10^{-4}$ & $2.85$ & $6.35\times10^{-4}$ & $2.64$ \\
      2 & $2.94\times10^{-5}$ & $2.98$ & $3.22\times10^{-5}$ & $2.96$ & $9.06\times10^{-5}$ & $2.81$ \\
      3 & $3.68\times10^{-6}$ & $3.00$ & $4.06\times10^{-6}$ & $2.99$ & $1.21\times10^{-5}$ & $2.90$ \\
    \midrule
    \multicolumn{7}{c}{{Case III}}\\
    \midrule
    Level & $||n-n_h||_{L^1}$ & Order & $||n-n_h||_{L^2}$ & Order & $||n-n_h||_{L^{\infty}}$ & Order\\
    \midrule
      0 & $3.08\times10^{-3}$ & -- & $3.51\times10^{-3}$ & -- & $7.20\times10^{-3}$ & --  \\
      1 & $4.08\times10^{-4}$ & $2.92$ & $4.86\times10^{-4}$ & $2.85$ & $1.14\times10^{-3}$ & $2.66$ \\
      2 & $5.17\times10^{-5}$ & $2.98$ & $6.24\times10^{-5}$ & $2.96$ & $1.69\times10^{-4}$ & $2.75$ \\
      3 & $6.50\times10^{-6}$ & $2.99$ & $7.87\times10^{-6}$ & $2.99$ & $2.32\times10^{-5}$ & $2.87$\\      
    \bottomrule
  \end{tabular}
\caption{\textbf{Example~\ref{ex:Single_process}: Single process.} Error table for the test on PBEs involving single processes. Level~$0$ represents a baseline mesh with $L = 15$ cells, and level~$\ell$ is obtained by equally splitting the cells of mesh level~$\ell-1$, for $\ell = 1, 2, 3$.}\label{tab:Single_process}
\end{table}

%============================================================
% Example 2 — Aggregation–breakage benchmark
%============================================================

\begin{exmp}\label{ex:AggBreak}
\textbf{Aggregation--breakage process}
\end{exmp}
In this example we test the scheme on the aggregation--breakage PBE \eqref{eq:PBE_aggregation_breakage} with
\begin{equation}\label{eq:aggr_break_para}
\beta(u,w)=1,\quad p(u,w)=\frac{2}{w},\quad \gamma(w)=w,
\end{equation}
whose analytical solution reads \cite{patil1998analytical, lage2002comments}
\begin{equation}\label{eq:aggr_break_sol0}
n(v,t)=2\left(\frac{\lambda(t)}{\lambda_{\infty}}\right)^2\exp\left(-\lambda(t)v\right),
\end{equation}
where
\begin{equation}\label{eq:aggr_break_sol1}
\lambda(t)=\lambda_{\infty}\frac{\lambda_0+\lambda_{\infty}\tanh(t/\lambda_{\infty })}{\lambda_{\infty}+\lambda_0\tanh(t/\lambda_{\infty})},
\end{equation}
and $\lambda_0, \lambda_{\infty}>0$ are positive constants.
The initial condition of the solution is
\begin{equation}\label{eq:aggr_break_init}
n(v,0)=2\left(\frac{\lambda_0}{\lambda_{\infty}}\right)^2\exp\left(-\lambda_0v\right),
\end{equation}
and we take $\lambda_0=3$, $\lambda_{\infty}=4$ in our test.

The computational domain, mesh generation and CFL choice are kept identical to Example \ref{ex:Single_process}.
We evolve the solution to $t=1$.
Figure \ref{fig:AggBreak} compares the numerical number density at $t=1$ with the analytical profile (left panel) and reports the deviation of the total mass over time (right panel).  
The overlap of the two curves and the machine-level mass preservation confirm both accuracy and conservation.  
Table \ref{tab:AggBreak} lists the errors on a sequence of successively refined meshes (the same meshes used in Example \ref{ex:Single_process}).  
A consistent third-order convergence is observed in all norms.

\begin{figure}[!htbp]
 \centering
 \begin{subfigure}[b]{0.32\textwidth}
  \includegraphics[width=\textwidth]{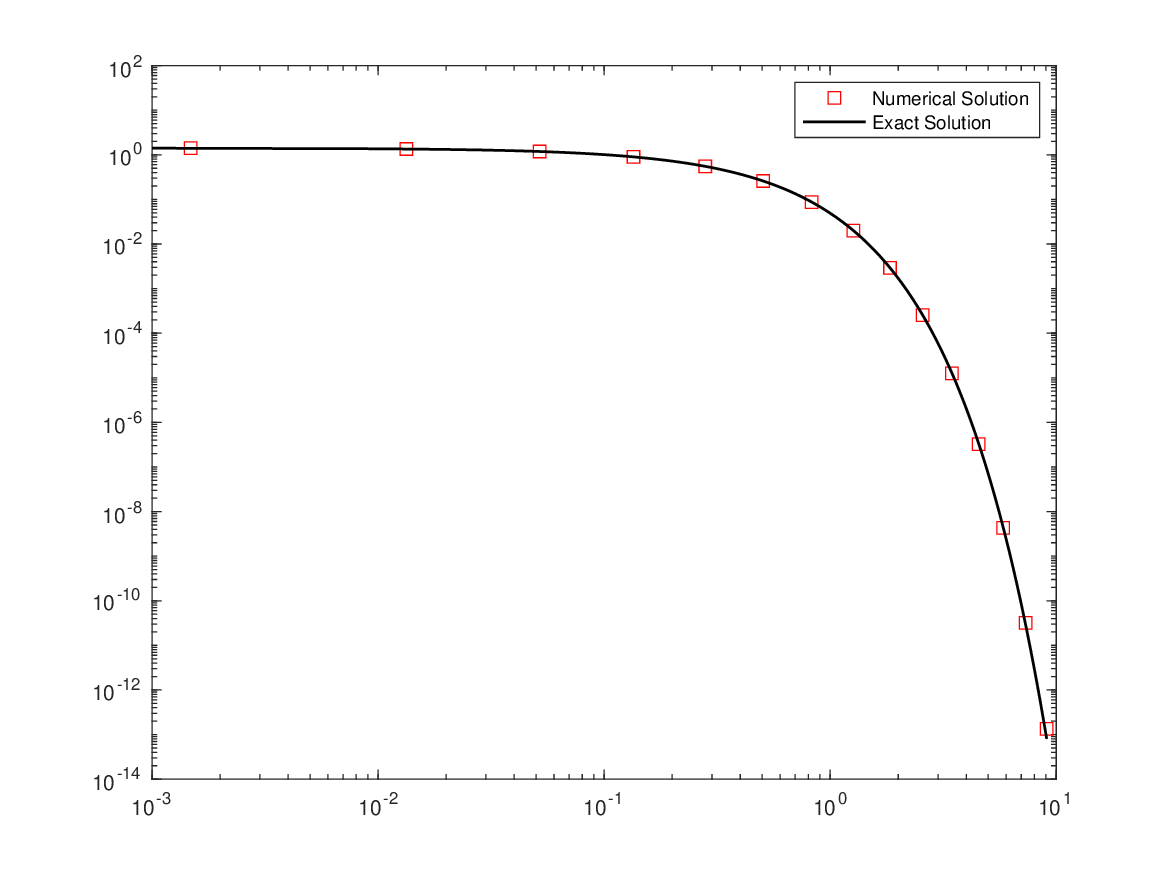}
  \caption{$n_h(v,t)$ at $t=1$}
 \end{subfigure}
 \begin{subfigure}[b]{0.32\textwidth}
  \includegraphics[width=\textwidth]{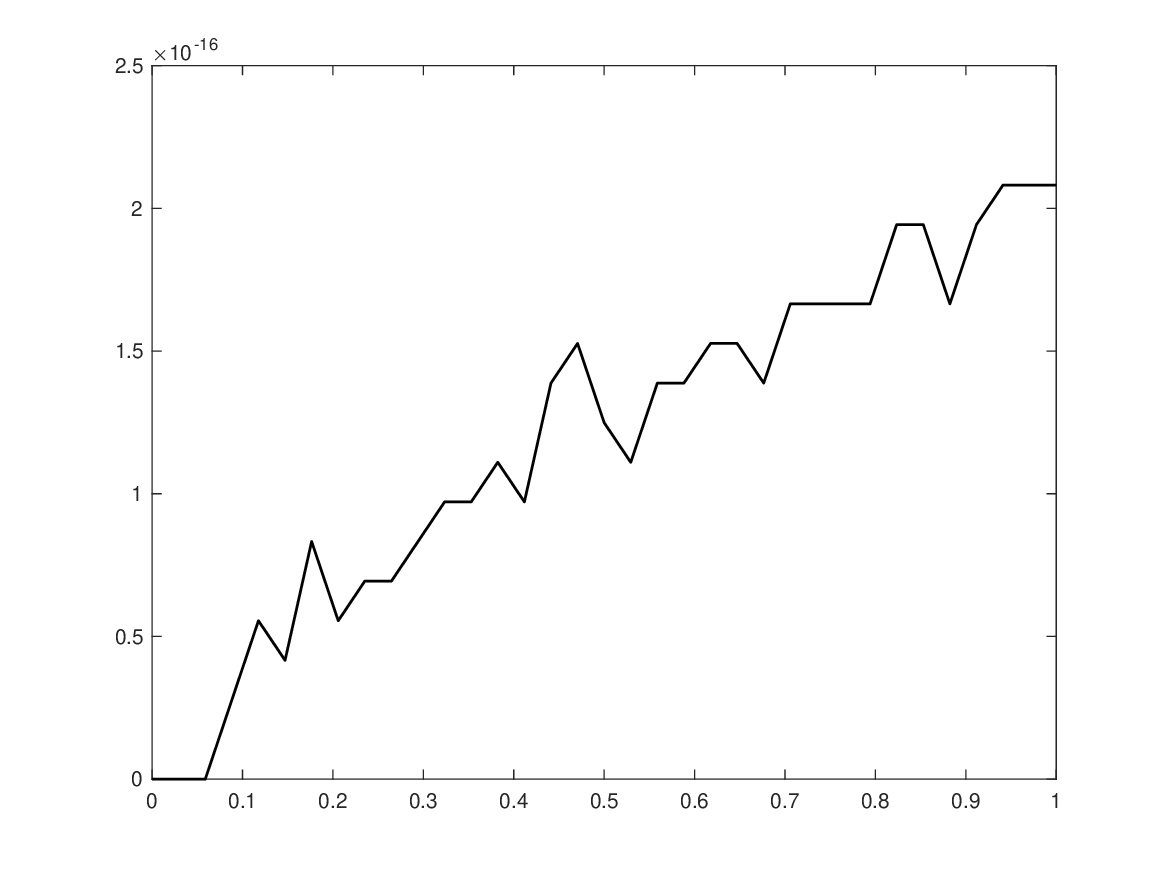}
  \caption{Total mass deviation}
 \end{subfigure}
\caption{\textbf{Example~\ref{ex:AggBreak}: Aggregation--breakage.}
Left: numerical and exact number densities; right: deviation of the total mass over time.  A nonuniform mesh with $L = 15$ cells is used.}
\label{fig:AggBreak}
\end{figure}

\begin{table}[htbp]
  \centering
  \renewcommand{\arraystretch}{1.2}   % vertical spacing
  \setlength{\tabcolsep}{5pt}         % horizontal padding
  \begin{tabular}{c*{6}{c}}
    \toprule
    %\midrule[1.0pt]
    % \multicolumn{7}{c}{{Case I}}\\
    % \midrule
    Level & $||n-n_h||_{L^{1}}$ & Order & $||n-n_h||_{L^{2}}$ & Order & $||n-n_h||_{L^{\infty}}$ & Order\\
    \midrule
      0 & $6.08\times10^{-4}$ & -- & $5.19\times10^{-4}$ & -- & $8.18\times10^{-4}$ & -- \\
      1 & $7.91\times10^{-5}$ & $2.94$ & $6.97\times10^{-5}$ & $2.90$ & $1.26\times10^{-4}$ & $2.70$ \\
      2 & $9.98\times10^{-6}$ & $2.99$ & $8.87\times10^{-6}$ & $2.97$ & $1.79\times10^{-5}$ & $2.81$ \\
      3 & $1.25\times10^{-6}$ & $3.00$ & $1.12\times10^{-6}$ & $2.99$ & $2.39\times10^{-6}$ & $2.90$ \\
    \bottomrule
  \end{tabular}
  \caption{\textbf{Example~\ref{ex:AggBreak}: Aggregation--breakage.} Error table for the test on the PBE involving aggregation-breakage process. Level~$0$ represents a baseline mesh with $L = 15$ cells, and level~$\ell$ is obtained by equally splitting the cells of mesh level~$\ell-1$, for $\ell = 1, 2, 3$.}
  \label{tab:AggBreak}  
\end{table}

%============================================================
% Example 3 — Aggregation–growth benchmark
%============================================================

\begin{exmp}\label{ex:AggGrowth}
\textbf{Aggregation--growth process}
\end{exmp}
We next test three aggregation--growth benchmarks taken from \cite{ramabhadran1976dynamics}.  
All cases include the linear growth term $G(v)=v$ and admit closed-form solutions.  

\begin{itemize}
    \item Case I: Constant aggregation kernel $\beta(u,w)=1$
    
The analytical solution is given by
    \begin{equation}
        n(v,t)=\frac{M_0^2}{M_1}\exp\left( -\frac{M_0}{M_1}v \right),
    \end{equation}
where
\begin{equation}
    M_0=\frac{2}{2+t},\quad M_1=v_0e^{t}.
\end{equation}
The corresponding initial and boundary conditions are given by
\begin{equation}
n(v,0)=\frac{1}{v_0}e^{-v/v_0},\quad n(0,t)=\frac{4}{v_0e^{t}(2+t)^2}.
\end{equation}

\item Case II: Constant aggregation kernel $\beta(u,w)=1$

The analytical solution is given by
\begin{equation}
n(v,t)=\frac{2M_0^2}{M_1}\frac{1}{\sqrt{1-M_0}}\exp\left( -\frac{2v}{M_1} \right)\sinh\left(\sqrt{1-M_0}\frac{2v}{M_1}\right)
\end{equation}
where
\begin{equation}
    M_0=\frac{2}{2+t},\quad M_1=2v_0e^{t}.
\end{equation}
The corresponding initial and boundary conditions differ from those in Case~I:
\begin{equation}
n(v,0)=\frac{v}{v_0^2}e^{-v/v_0},\quad n(0,t)=0    
\end{equation}

\item Case III: Additive aggregation kernel $\beta(u,w)=u+w$

The analytical solution is given by
\begin{equation}
n(v,t)=\frac{\frac{M_0^2}{M_1}\exp\left(-\frac{M_0}{M_1}\left(\frac{2}{M_0}-1\right)v\right)I_1\left(2\sqrt{1-M_0}\frac{v}{M_1}\right)}{\frac{M_0v}{M_1}\sqrt{1-M_0}},
\end{equation}
where 
\begin{equation}
M_0=\exp\left(v_0(1-e^{t})\right),\quad M_1=v_0e^{t}.
\end{equation}
The corresponding initial and boundary conditions are given by
\begin{equation}
n(v,0)=\frac{1}{v_0}e^{-v/v_0},\quad n(0,t)=\frac{\exp\left(v_0(1-e^{t})\right)}{v_0e^{t}}
\end{equation}
\end{itemize}
Here, $M_0:=\int_{0}^{\infty}n(v,t)\,dv$ and $M_1:=\int_{0}^{\infty}vn(v,t)\,dv$ are the zeroth and first moments of the solution representing the total number and mass of particles over time, respectively.

We reuse the domain, mesh family and time step described in Example \ref{ex:Single_process}.  
Each simulation is run to $t=1$.
Due to growth, the total mass of particles is not conserved over time, nor is the total number.
Figure \ref{fig:AggGrowth} displays, for every case, (i) the numerical and analytical number density at the final time, (ii) the evolution of the total number $M_0(t)$ and (iii) that of total mass $M_1(t)$; all show excellent agreement with the exact solutions.
Table \ref{tab:AggGrowth} presents the corresponding error and convergence data.
Third-order accuracy is again observed in all norms.

\begin{figure}[!htbp]
 \centering
 \begin{subfigure}[b]{0.32\textwidth}
  \includegraphics[width=\textwidth]{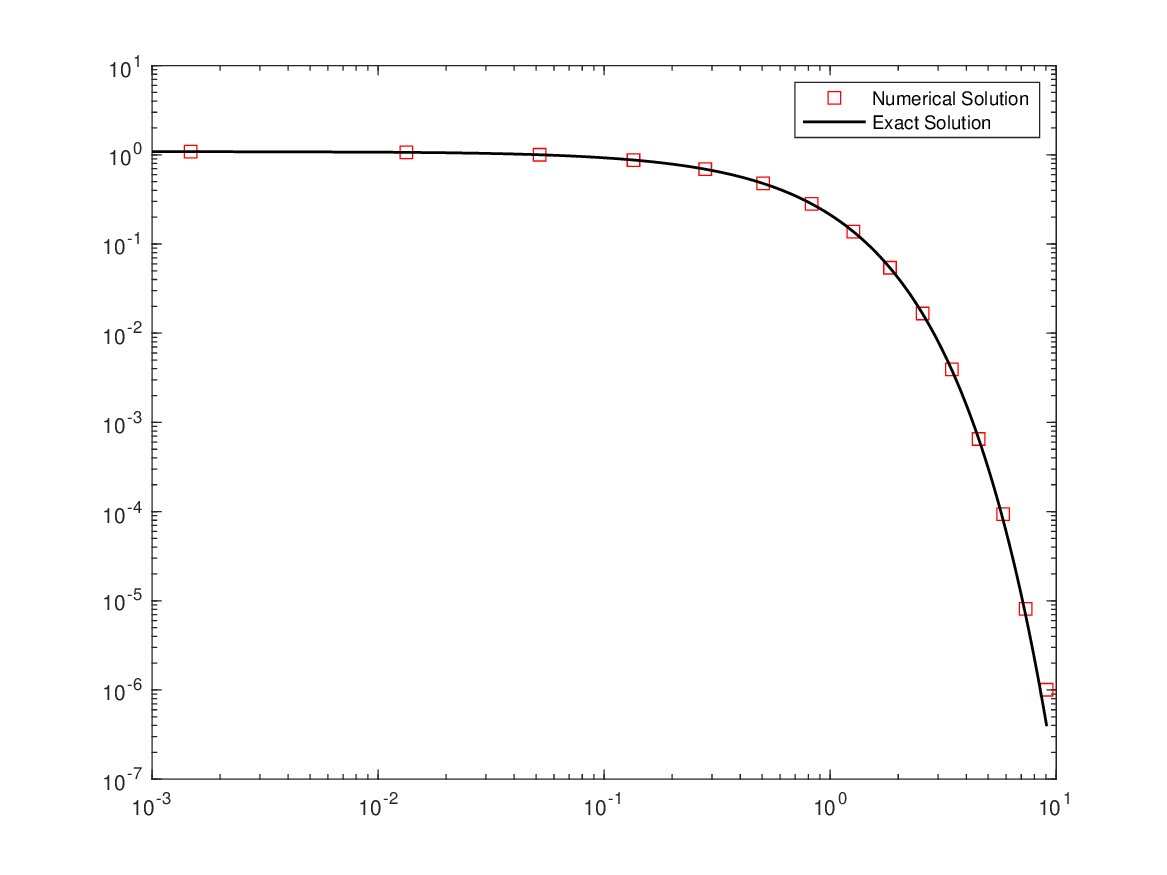}
  \caption{Case I: $n_h(v,t)$ at $t=1$}
 \end{subfigure}
 \begin{subfigure}[b]{0.32\textwidth}
  \includegraphics[width=\textwidth]{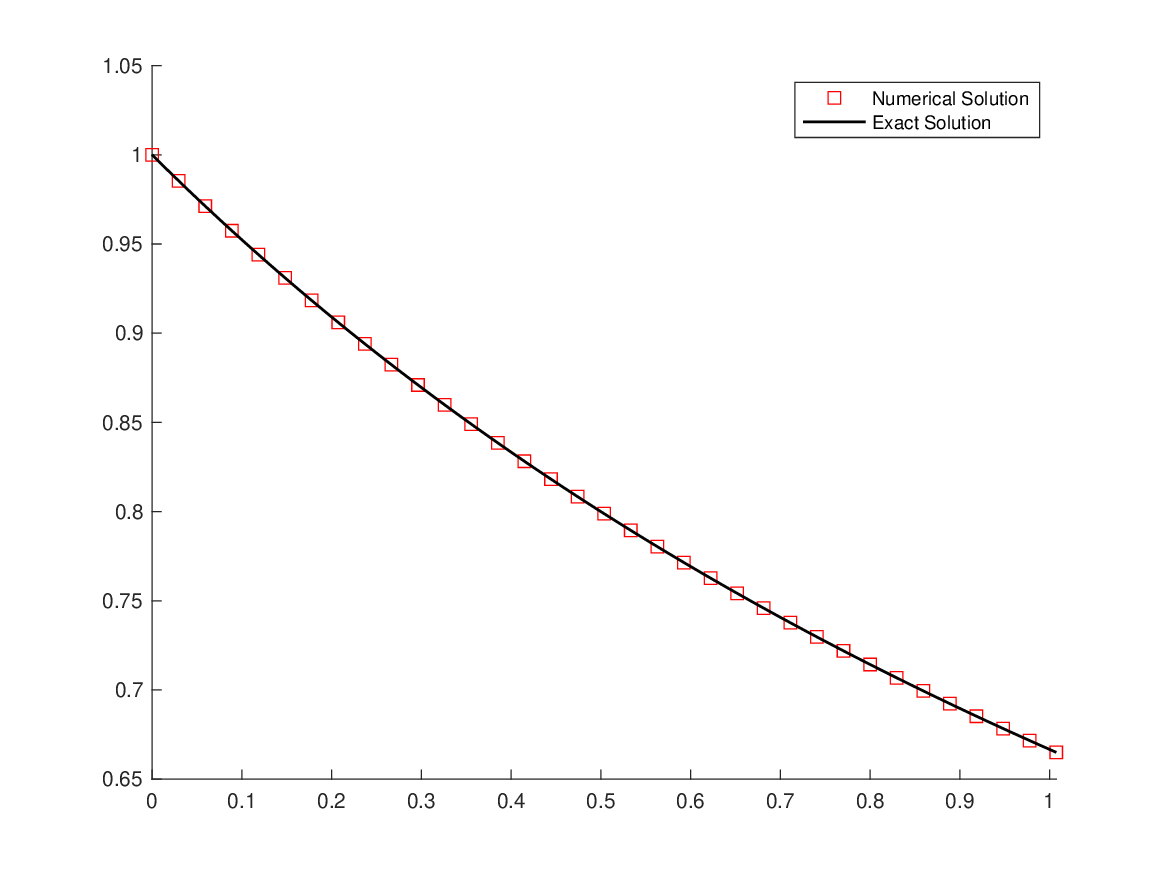}
  \caption{Case I: $M_0$ over time}
 \end{subfigure}
 \begin{subfigure}[b]{0.32\textwidth}
  \includegraphics[width=\textwidth]{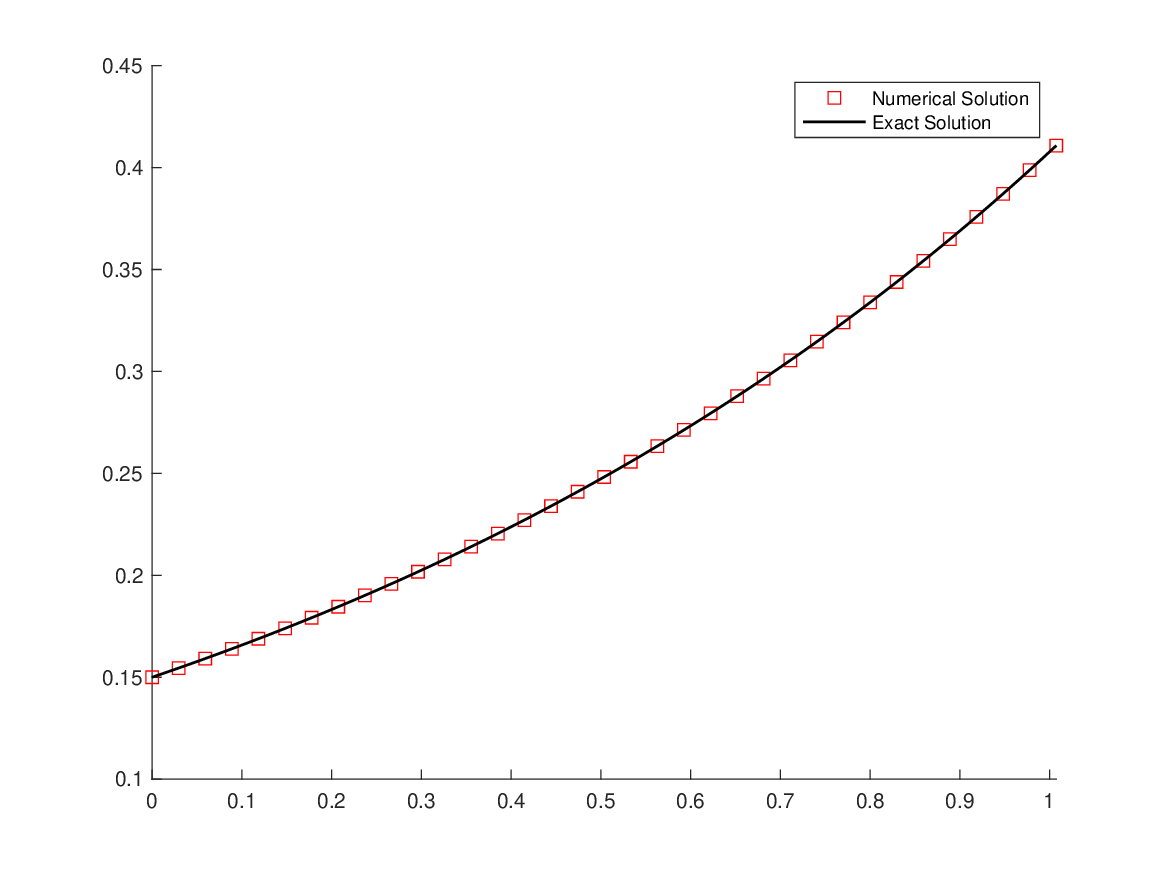}
  \caption{Case I: $M_1$ over time}
 \end{subfigure}

 \begin{subfigure}[b]{0.32\textwidth}
  \includegraphics[width=\textwidth]{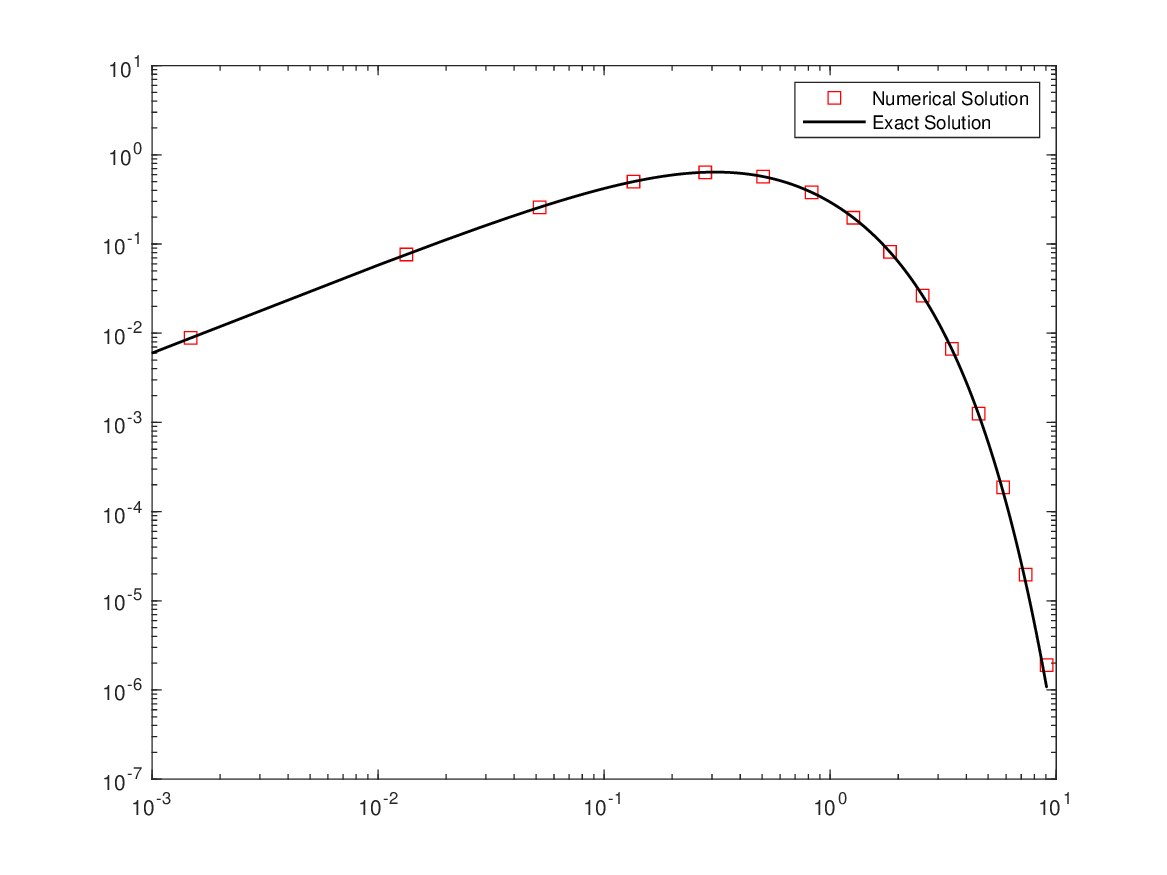}
  \caption{Case II: $n_h(v,t)$ at $t=1$}
 \end{subfigure}
 \begin{subfigure}[b]{0.32\textwidth}
  \includegraphics[width=\textwidth]{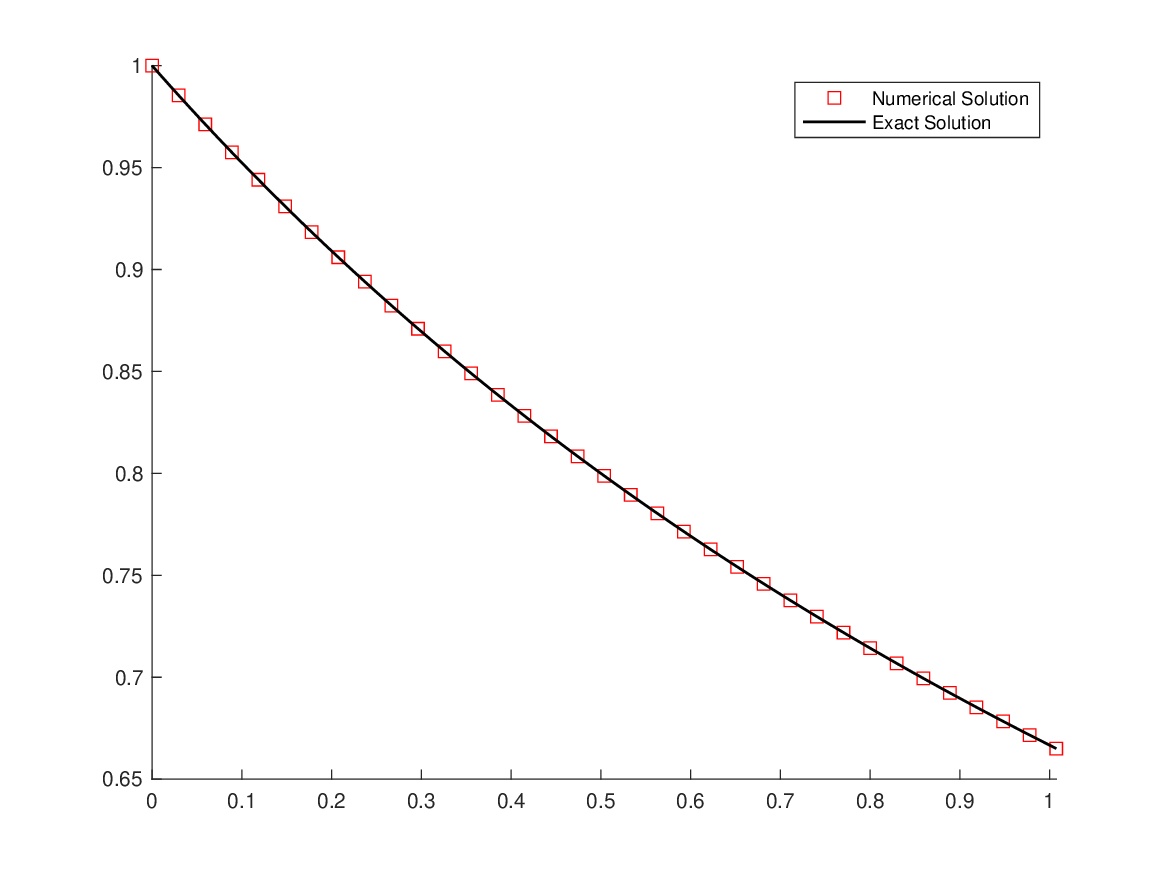}
  \caption{Case II: $M_0$ over time}
 \end{subfigure}
 \begin{subfigure}[b]{0.32\textwidth}
  \includegraphics[width=\textwidth]{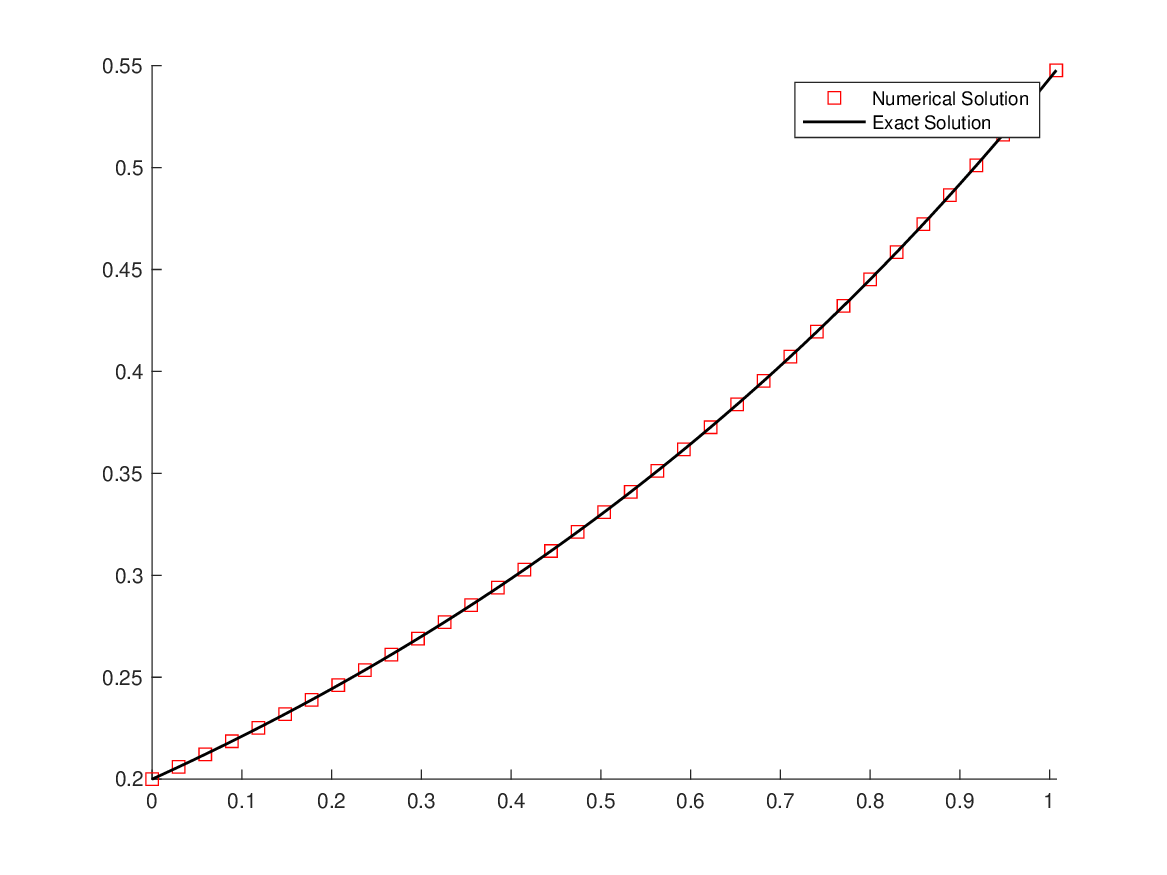}
  \caption{Case II: $M_1$ over time}
 \end{subfigure}

 \begin{subfigure}[b]{0.32\textwidth}
  \includegraphics[width=\textwidth]{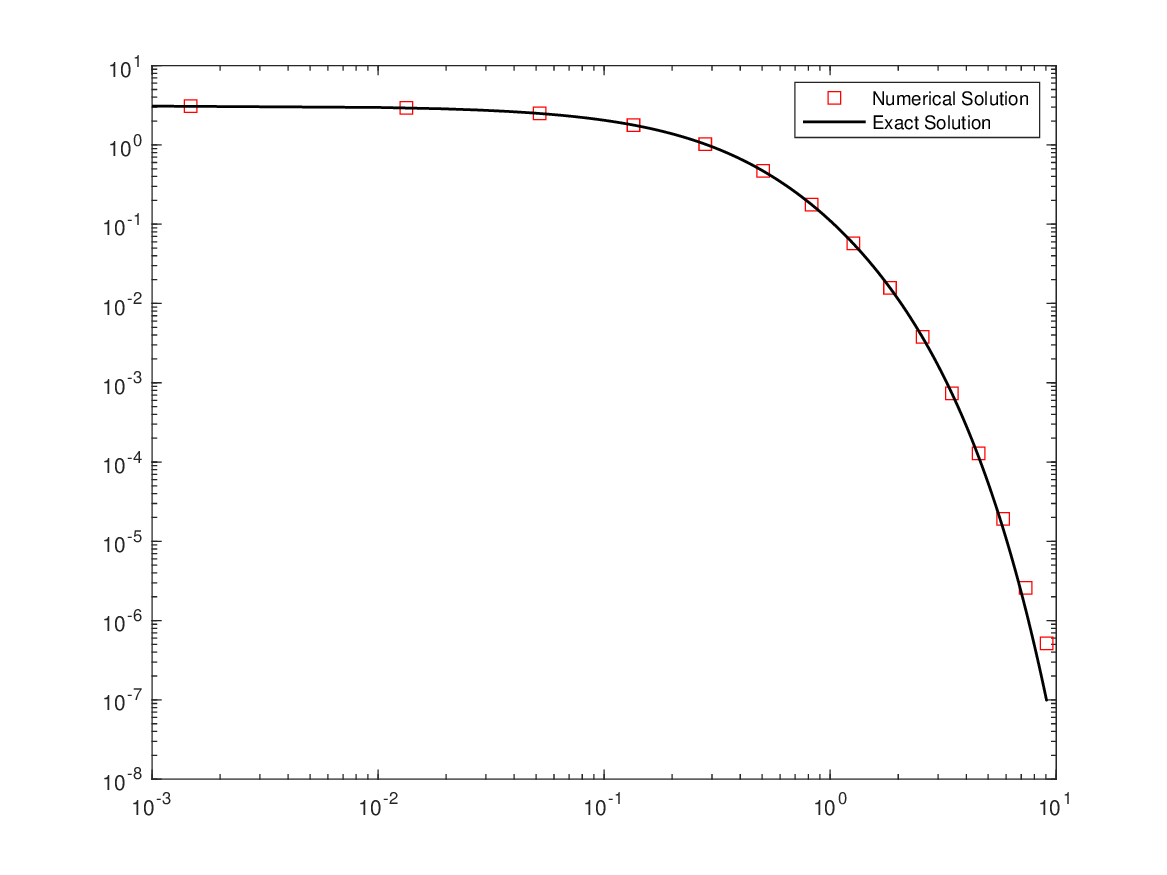}
  \caption{Case III: $n_h(v,t)$ at $t=1$}
 \end{subfigure}
 \begin{subfigure}[b]{0.32\textwidth}
  \includegraphics[width=\textwidth]{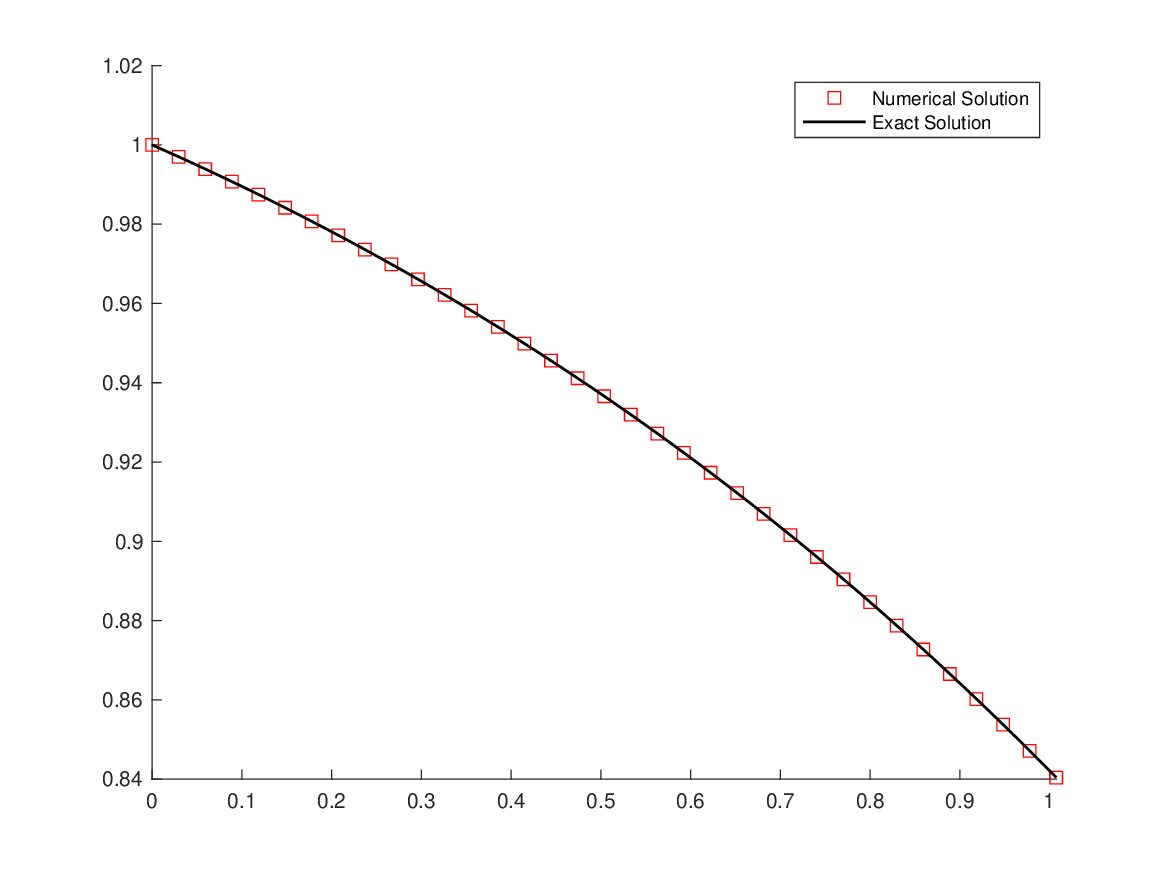}
  \caption{Case III: $M_0$ over time}
 \end{subfigure}
 \begin{subfigure}[b]{0.32\textwidth}
  \includegraphics[width=\textwidth]{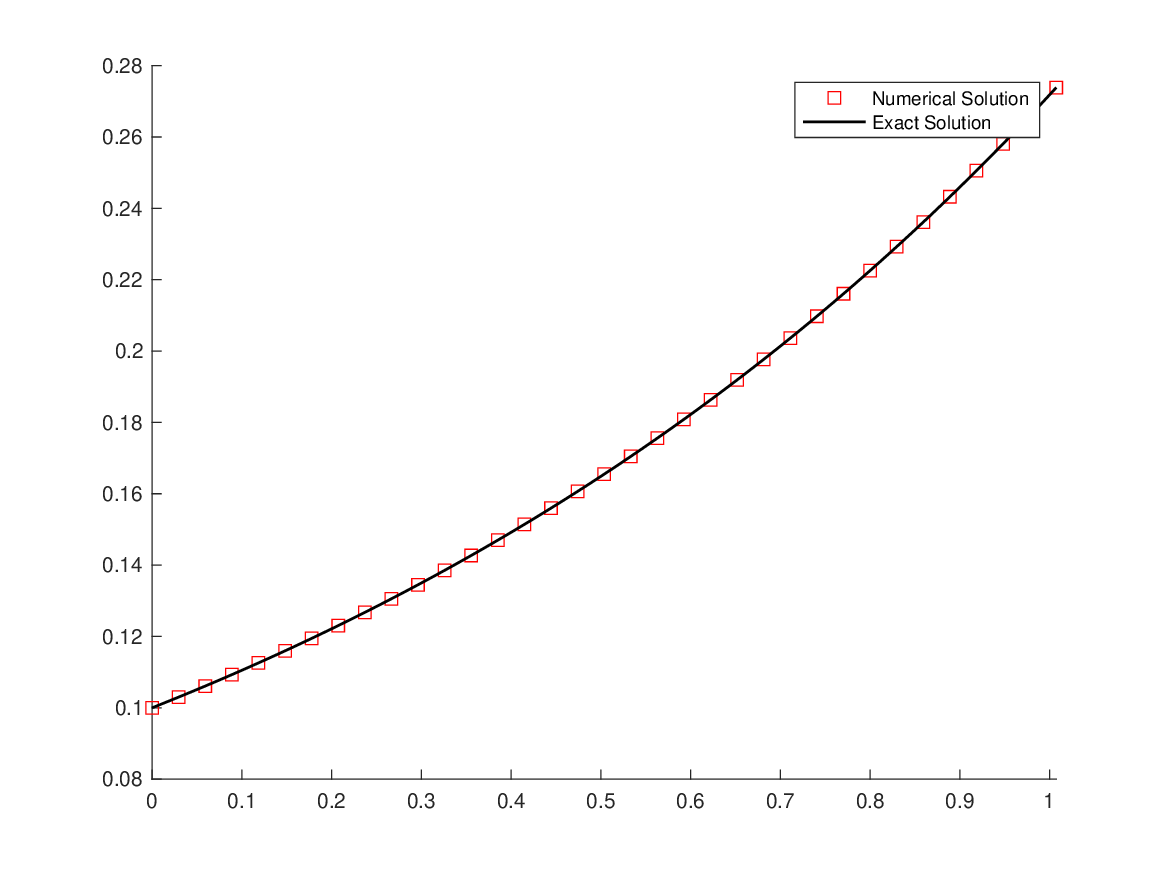}
  \caption{Case III: $M_1$ over time}
 \end{subfigure}
\caption{\textbf{Example~\ref{ex:AggGrowth}: Aggregation--growth.}
Left: numerical and exact number densities; middle: evolution of the total number over time; right: evolution of the total mass over time.
A nonuniform mesh with $L = 15$ cells is used.}
\label{fig:AggGrowth}
\end{figure}

\begin{table}[htbp]
  \centering
  \renewcommand{\arraystretch}{1.2}   % vertical spacing
  \setlength{\tabcolsep}{5pt}         % horizontal padding
  \begin{tabular}{c*{6}{c}}
    \toprule
    \multicolumn{7}{c}{\textbf{Case I}}\\
    \midrule
    Level & $||n-n_h||_{L^1}$ & Order & $||n-n_h||_{L^2}$ & Order & $||n-n_h||_{L^{\infty}}$ & Order\\
    \midrule
      0 & $7.17\times10^{-4}$ & --   & $4.93\times10^{-4}$ & --   & $8.38\times10^{-4}$ & --  \\
      1 & $9.38\times10^{-5}$ & 2.93 & $6.56\times10^{-5}$ & 2.91 & $1.57\times10^{-4}$ & 2.41\\
      2 & $1.11\times10^{-5}$ & 3.09 & $7.78\times10^{-6}$ & 3.08 & $1.65\times10^{-5}$ & 3.26\\
      3 & $1.33\times10^{-6}$ & 3.05 & $9.34\times10^{-7}$ & 3.06 & $2.18\times10^{-6}$ & 2.92\\
    \midrule
    \multicolumn{7}{c}{\textbf{Case II}}\\
    \midrule
    Level & $||n-n_h||_{L^1}$ & Order & $||n-n_h||_{L^2}$ & Order & $||n-n_h||_{L^{\infty}}$ & Order\\
    \midrule
      0 & $1.48\times10^{-3}$ & --   & $1.25\times10^{-3}$ & --   & $3.46\times10^{-3}$ & -- \\
      1 & $1.86\times10^{-4}$ & 2.99 & $1.66\times10^{-4}$ & 2.91 & $4.68\times10^{-4}$ & 2.88 \\
      2 & $2.19\times10^{-5}$ & 3.09 & $1.99\times10^{-5}$ & 3.06 & $8.29\times10^{-5}$ & 2.50 \\
      3 & $2.62\times10^{-6}$ & 3.06 & $2.32\times10^{-6}$ & 3.10 & $1.02\times10^{-5}$ & 3.02 \\
    \midrule
    \multicolumn{7}{c}{\textbf{Case III}}\\
    \midrule
    Level & $||n-n_h||_{L^1}$ & Order & $||n-n_h||_{L^2}$ & Order & $||n-n_h||_{L^{\infty}}$ & Order\\
    \midrule
      0 & $1.90\times10^{-3}$ & --   & $2.03\times10^{-3}$ & --   & $6.63\times10^{-3}$ & --  \\
      1 & $2.50\times10^{-4}$ & 2.93 & $2.74\times10^{-4}$ & 2.89 & $8.98\times10^{-4}$ & 2.88\\
      2 & $3.11\times10^{-5}$ & 3.01 & $3.39\times10^{-5}$ & 3.02 & $1.36\times10^{-4}$ & 2.72\\
      3 & $3.82\times10^{-6}$ & 3.02 & $4.12\times10^{-6}$ & 3.04 & $1.69\times10^{-5}$ & 3.01\\      
    \bottomrule
  \end{tabular}
    \caption{\textbf{Example~\ref{ex:AggGrowth}: Aggregation--growth.} Error table for the test on the PBE involving aggregation-growth process. Level~$0$ represents a baseline mesh with $L = 15$ cells, and level~$\ell$ is obtained by equally splitting the cells of mesh level~$\ell-1$, for $\ell = 1, 2, 3$.}
\label{tab:AggGrowth}
\end{table}

\subsection{More complex cases}
Expressions for the aggregation kernels encountered in practice are usually highly nonlinear, making analytical solutions unavailable.
To obtain benchmark curves, we therefore solve the \emph{discrete population balance equation} (DPBE) \cite{smoluchowski1918}:
\begin{equation}
\left\{
\begin{aligned}
\frac{\dd n_1}{\dd t}=&-n_1\sum_{j=1}^{\infty}\beta_{1j}n_j,\\
\frac{\dd n_i}{\dd t}=&\frac12\sum_{j=1}^{i-1}\beta_{j,i-j}n_jn_{i-j}-n_i\sum_{j=1}^\infty\beta_{ij}n_j,\qquad i\ge2,    
\end{aligned}\right.
\end{equation}
on a fine uniform grid $v_i=i\Delta v$, where $\beta_{ij}=\beta(v_i,v_j)$ and $n_i(t)$ denotes the particle count in class $i$.
It can be shown \cite{drake1972} that the DPBE converges to the continuous equation as $\Delta v\rightarrow0$.
However, since this convergence is only first-order in $\Delta v$, computing high-accuracy reference solutions using the DPBE is computationally prohibitive.
We therefore do not perform a formal convergence study in this subsection, as extensive accuracy tests have already been presented in the previous subsection.
Instead, the DPBE solutions are used only to verify the qualitative correctness of the DG results.
%in cases where analytical solutions are not available.

\begin{exmp}\label{ex:complex}
\textbf{Aggregation with physical kernels}
\end{exmp}
We consider aggregation processes with the following physically derived kernels \cite{Friedlander2000}:
\begin{itemize}
\item Case I: Free-molecule kernel
\begin{equation}\label{eq:kernel_free}
\beta(u,w)=\Big(\frac{1}{u}+\frac{1}{w} \Big)^{\frac12}\Big(u^{1/3}+w^{1/3}\Big)^{2},
\end{equation}
valid in the free-molecular regime, where the gas mean free path greatly exceeds the particle diameter.
\item Case II: Brownian kernel
\begin{equation}\label{eq:kernel_brown}
\beta(u,w)=\Big(\frac{1}{v^{1/3}}+\frac{1}{w^{1/3}}\Big)\Big( v^{1/3} + w^{1/3} \Big),
\end{equation}
appropriate for the continuum (Brownian) regime, where the mean free path is small compared with the particle diameter.   
\item Case III: Gravitational kernel
\begin{equation}\label{eq:kernel_gravity}
\beta(u,w)=\Big(v^{1/3}+w^{1/3}\Big)^2\Big|v^{2/3}-w^{2/3}\Big|,
\end{equation}
representing differential gravitational settling, in which larger particles overtake and collect smaller ones as they fall.
\end{itemize}
All cases use the same exponential initial condition:
\begin{equation}
n(v,0)=\frac{1}{v_0}\exp\left(-\frac{v}{v_0}\right).
\end{equation}
In all three tests, we take $v_0=0.15$.

We reuse the domain, mesh and time step described in Example \ref{ex:Single_process}.  
Each simulation is run to $t=1$.
Figure \ref{fig:complex} compares the numerical number densities at $t=1$ with the reference ones obtained from DPBE (left panel) and reports the deviation of the total mass over time (right panel).  
The overlap of the curves and the machine-level mass preservation confirm both accuracy and conservation.

\begin{figure}[!htbp]
 \centering
 \begin{subfigure}[b]{0.32\textwidth}
  \includegraphics[width=\textwidth]{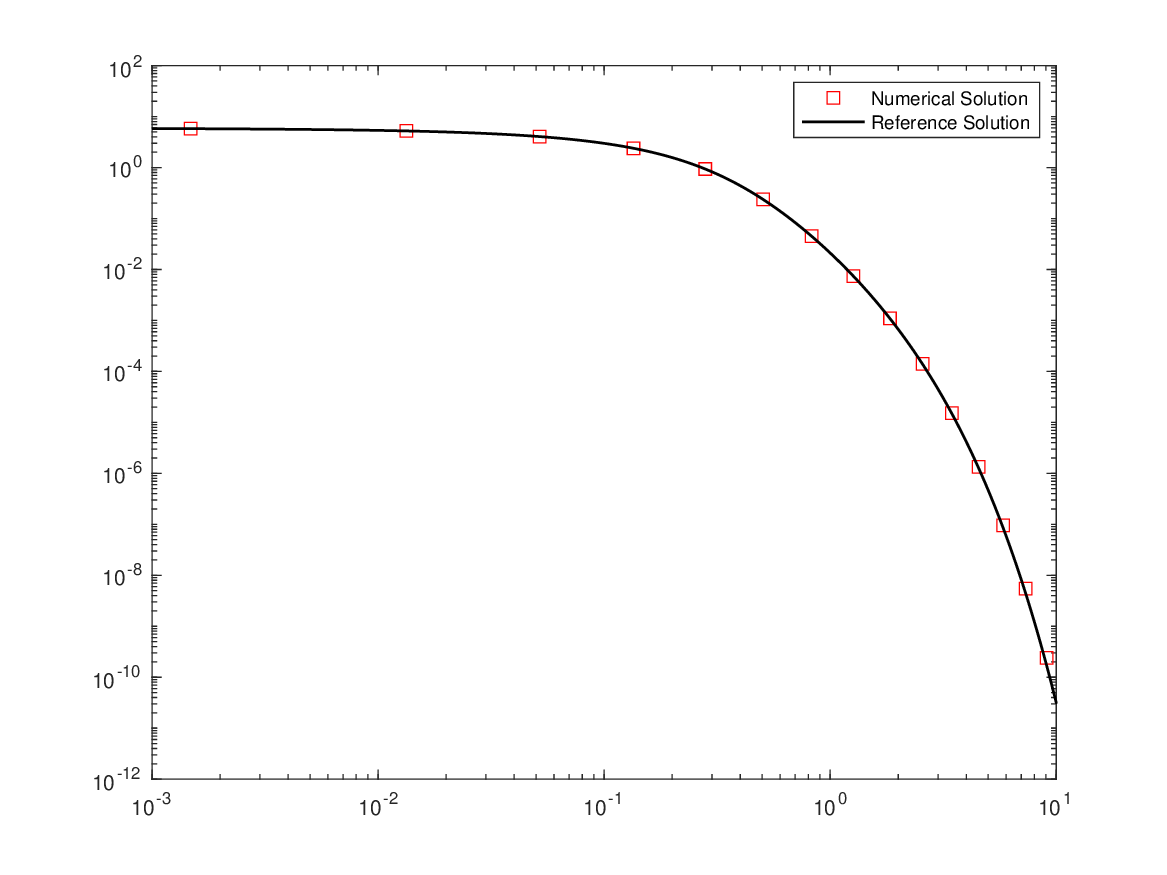}
  \caption{Case I: $n_h(v,t)$ at $t=1$}
 \end{subfigure}
 \begin{subfigure}[b]{0.32\textwidth}
  \includegraphics[width=\textwidth]{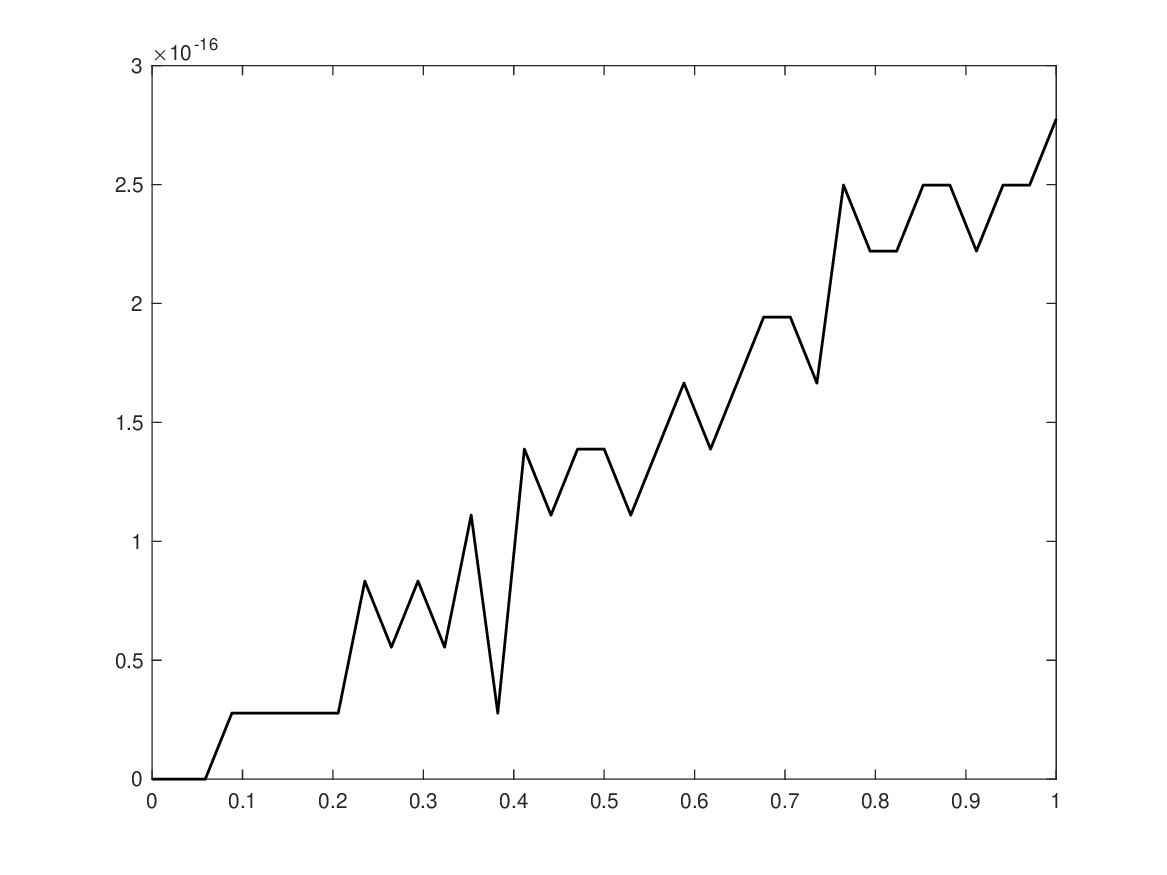}
  \caption{Case I: Total mass deviation}
 \end{subfigure}

 \begin{subfigure}[b]{0.32\textwidth}
  \includegraphics[width=\textwidth]{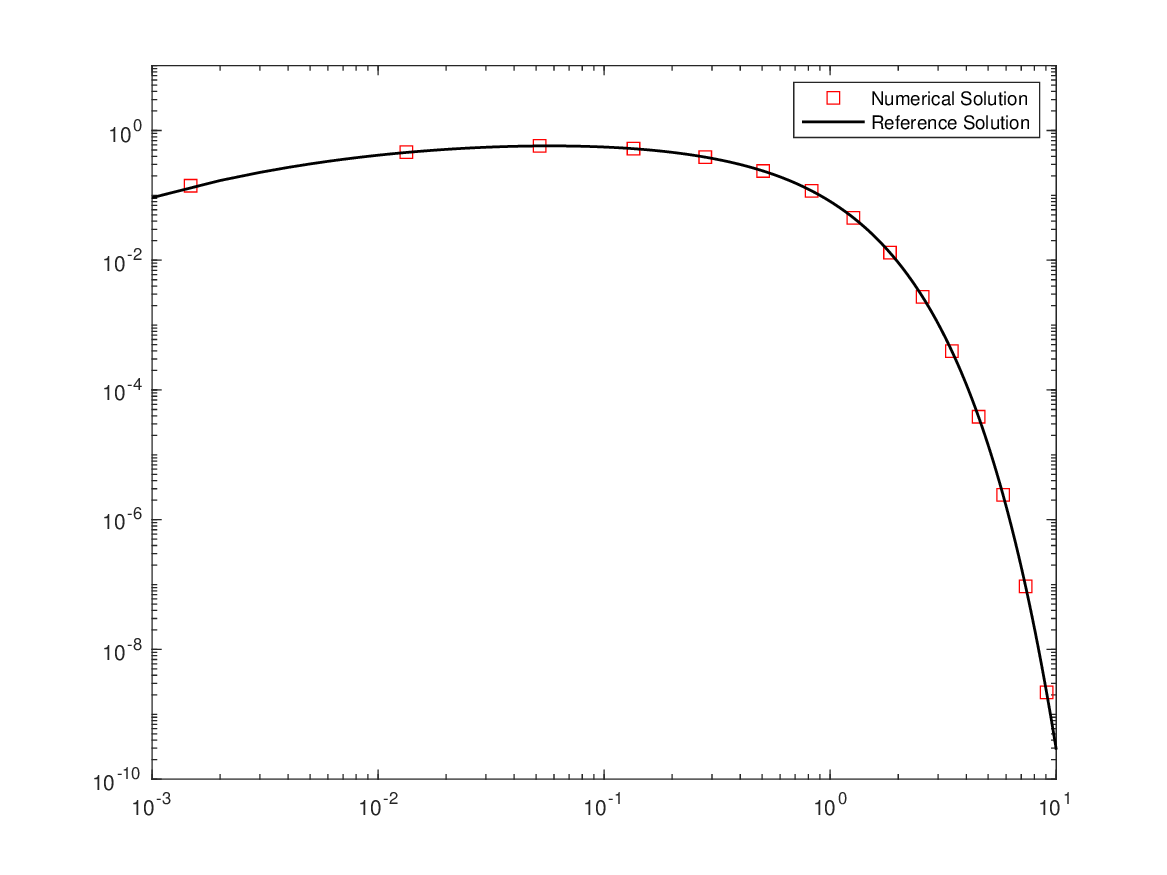}
  \caption{Case II: $n_h(v,t)$ at $t=1$}
 \end{subfigure}
 \begin{subfigure}[b]{0.32\textwidth}
  \includegraphics[width=\textwidth]{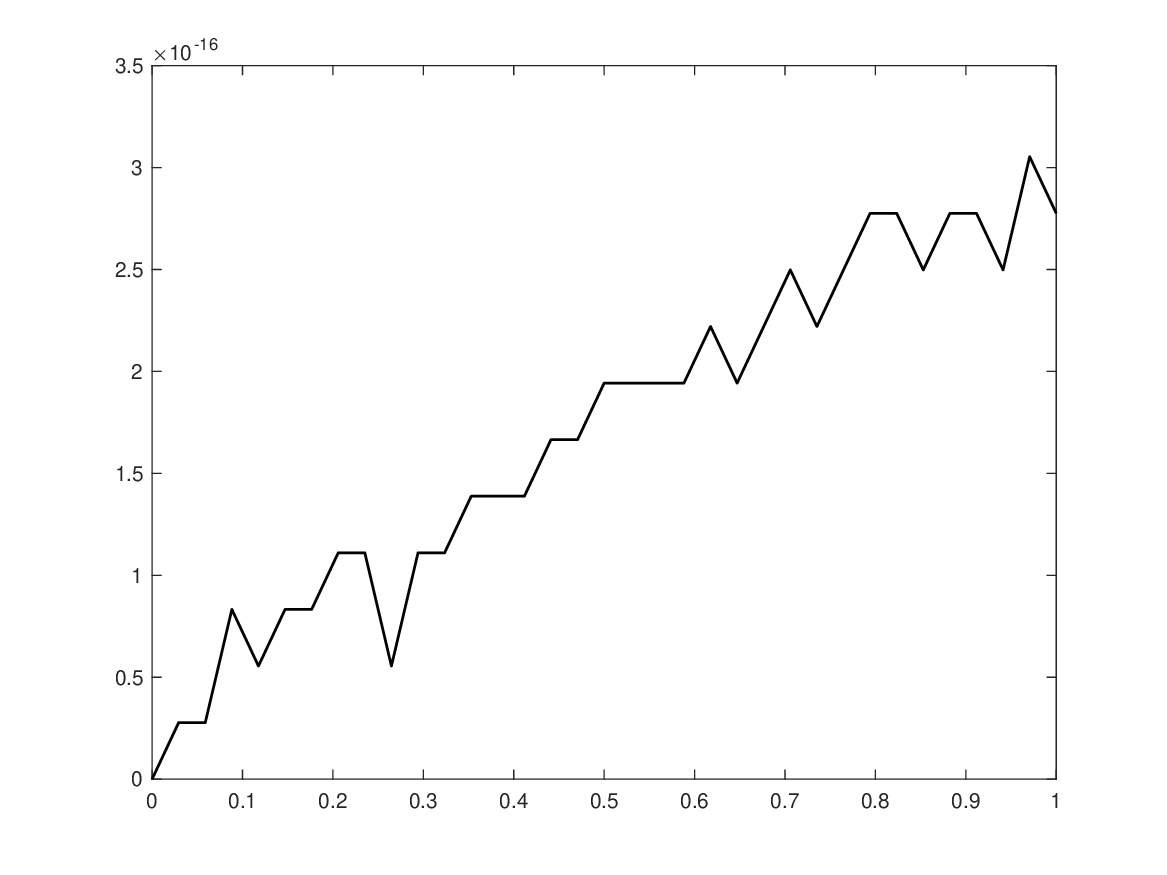}
  \caption{Case II: Total mass deviation}
 \end{subfigure}

 \begin{subfigure}[b]{0.32\textwidth}
  \includegraphics[width=\textwidth]{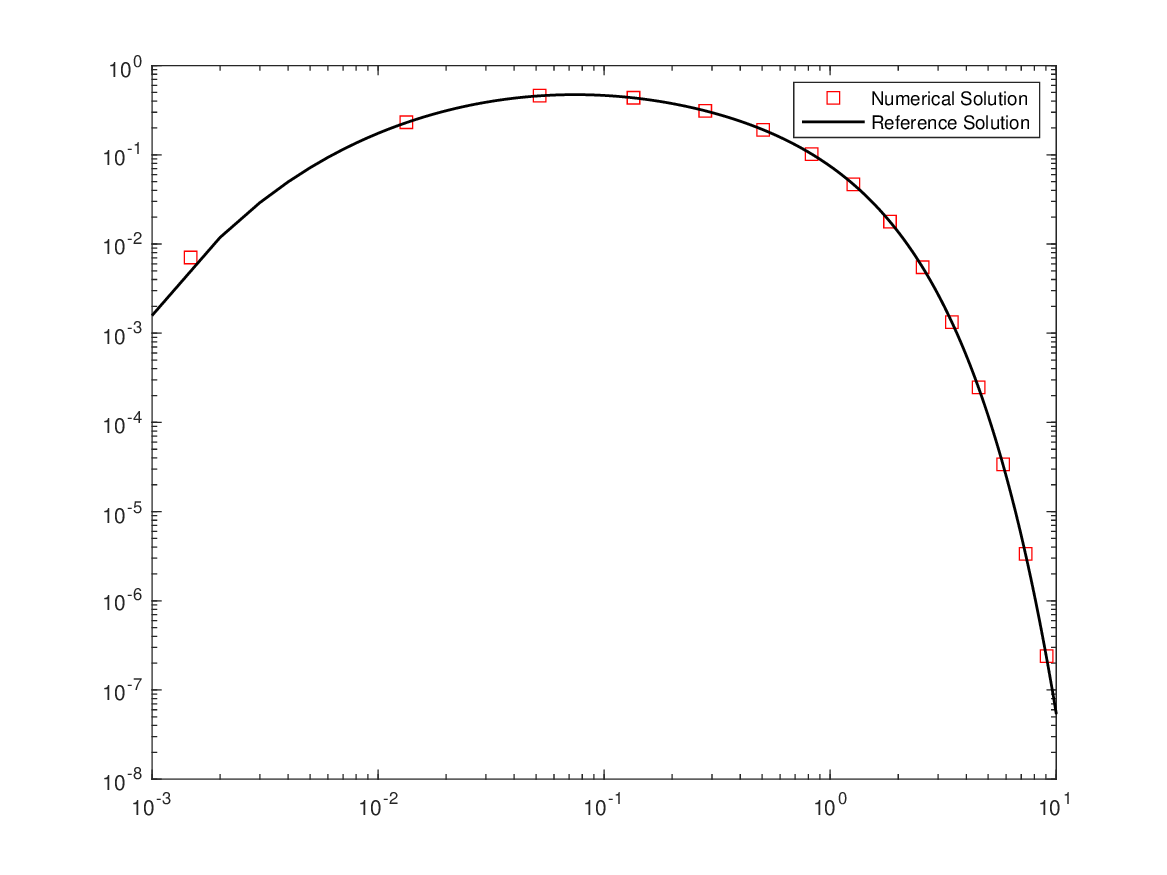}
  \caption{Case III: $n_h(v,t)$ at $t=1$}
 \end{subfigure}
 \begin{subfigure}[b]{0.32\textwidth}
  \includegraphics[width=\textwidth]{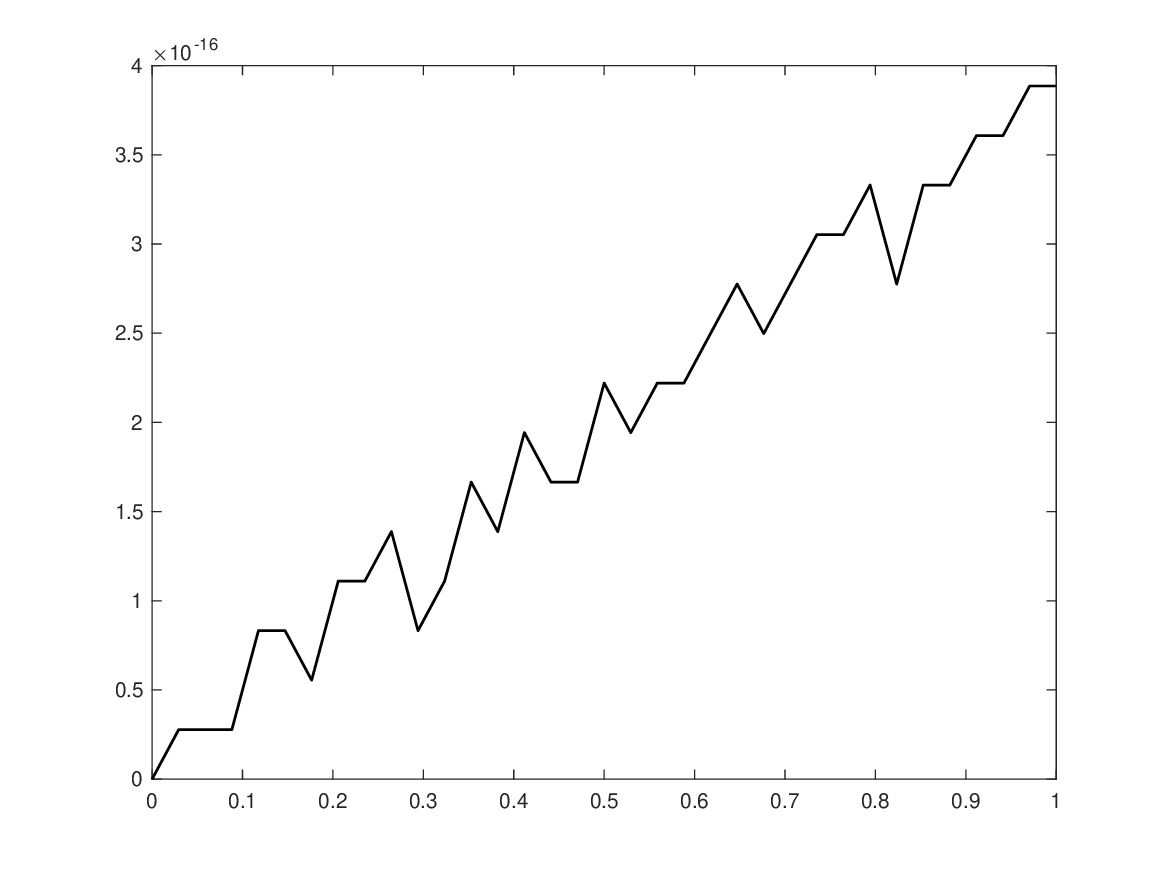}
  \caption{Case III: Total mass deviation}
 \end{subfigure} 
\caption{\textbf{Example~\ref{ex:complex}: Physical kernels.} Left: numerical and reference number densities; right: deviation of the total mass over time. A nonuniform mesh with $L = 15$ cells is used.}
\label{fig:complex}
\end{figure}

\subsection{Positivity-preserving tests}
In this subsection, we highlight the importance of the positivity-preserving technique in simulating the population balance equation.
In all previous tests, the positivity-preserving limiter was employed.
The resulting numerical solutions were nonnegative and exhibited exact mass conservation.
In those tests, we observed that several examples yielded negative number densities when the limiter was turned off, which is physically unacceptable.
More importantly, we examine a test in which the use of the positivity-preserving technique is essential: without it, the numerical simulation fails entirely.

\begin{exmp}\label{ex:pptest}
\textbf{Effectiveness of positivity preservation}
\end{exmp}
We revisit the aggregation--breakage PBE considered in Example~\ref{ex:AggBreak}.  
The parameters are set as 
\begin{itemize}
    \item Case I: 
    \begin{equation}
        \lambda_0=1,\quad \lambda_{\infty}=\frac12 \quad(\text{aggregation-dominated regime})
    \end{equation}
    \item Case II:
    \begin{equation}
        \lambda_0=1,\quad \lambda_{\infty}=1 \quad(\text{aggregation-breakage balanced regime})
    \end{equation}
    \item Case III:
    \begin{equation}
        \lambda_0=1,\quad \lambda_{\infty}=2 \quad(\text{breakage-dominated regime})        
    \end{equation}
\end{itemize}
The computational domain is $\mathcal{D} = [0, 10^{3}]$.  
We employ the $\mathcal{P}^{4}$-DG scheme ($k = 4$) with $L = 11$ cells, using a partition with grid points $v_{\frac12} = 0$ and $v_{i+\frac12}$, $i = 1, 2, \ldots, L$ uniformly distributed in log scale between $10^{-3}$ and $10^{3}$.
We perform a long-term simulation up to $t = 10$.  

In these tests, we observe that the DG method equipped with the positivity-preserving limiter remains stable with a time step of $\Delta t = 0.02$ across all cases and yields results consistent with the analytical solutions.
In contrast, the DG method without the limiter becomes unstable and blows up rapidly under the same CFL condition.  
Further experiments show that only when the time step is reduced to $\Delta t = 0.002$ the results of the non-limited DG method are comparable to those of the positivity-preserving DG scheme with $\Delta t = 0.02$.
For Case I, Figure \ref{fig:pptest1} displays the number density and mass density distributions over particle sizes, as well as the $L^1$-error over time, for the limited DG scheme ($\Delta t = 0.02$) at the final time $t = 10$, and for the non-limited DG scheme ( $\Delta t = 0.003$) at $t = 5.04$ right before the blow-up time.
Cases~II and III exhibit similar behavior; to save space, we only present the comparison of $L^1$-error evolution over time in Figure \ref{fig:pptest2} for these two cases.  
These numerical results clearly demonstrate the enhanced robustness provided by the positivity-preserving limiter.

\begin{figure}[!htbp]
 \centering
 \begin{subfigure}[b]{0.32\textwidth}
  \includegraphics[width=\textwidth]{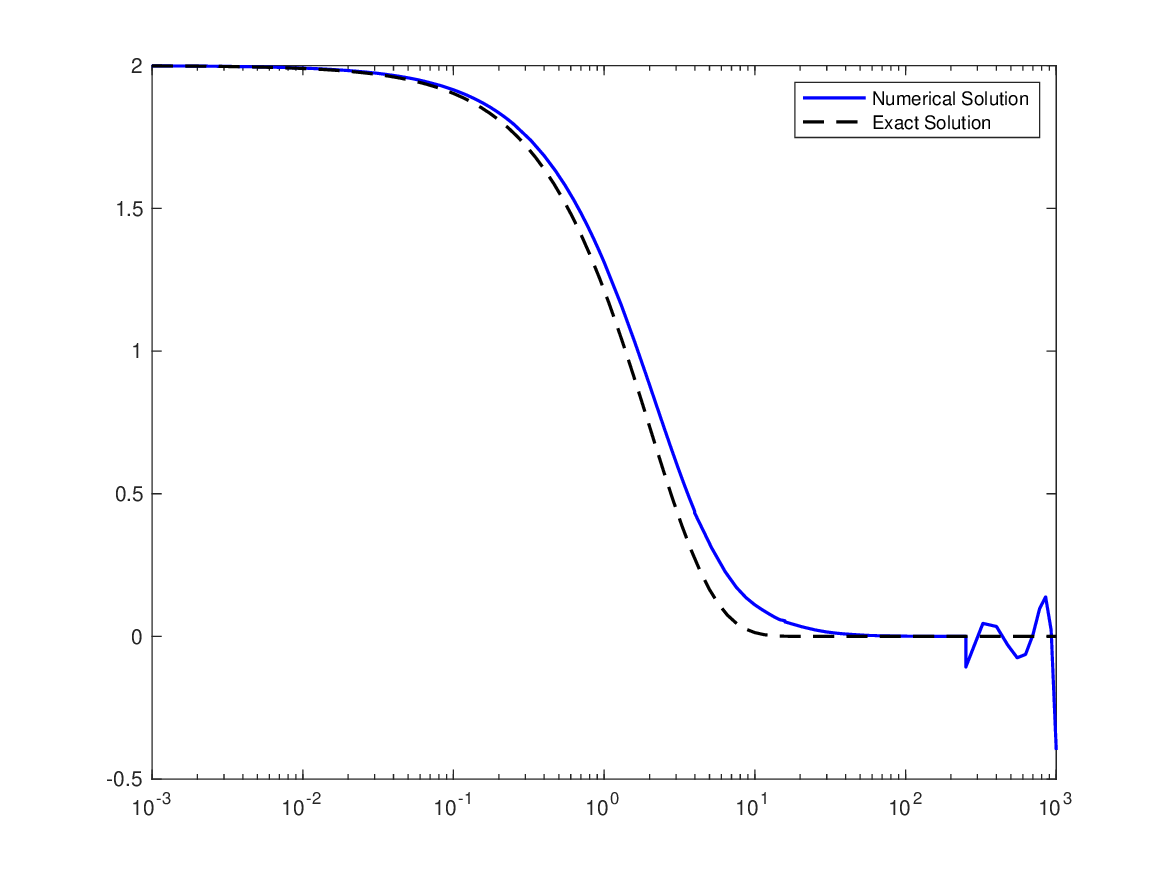}
  \caption{$n_h(v,t)$ at $t=5.04$}
 \end{subfigure}
 \begin{subfigure}[b]{0.32\textwidth}
  \includegraphics[width=\textwidth]{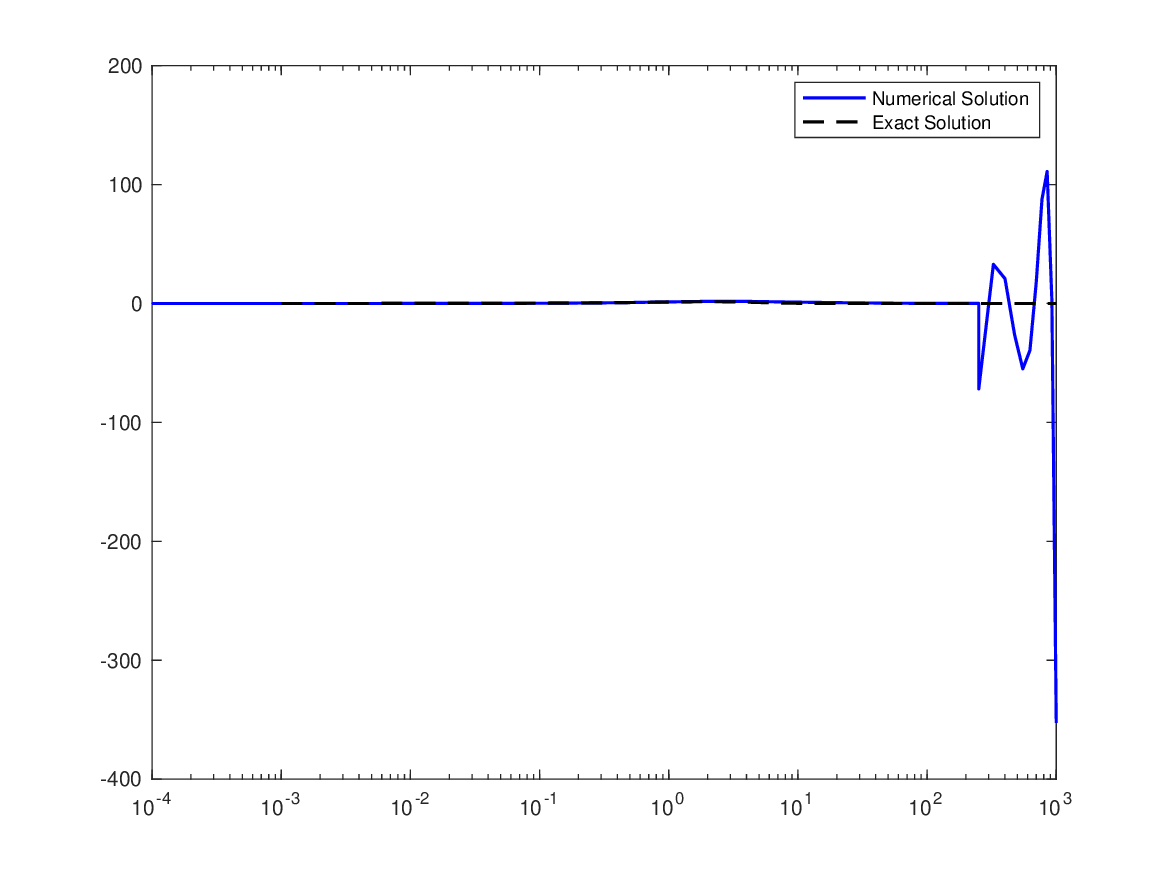}
  \caption{$v n_h(v,t)$ at $t=5.04$}
 \end{subfigure}
 \begin{subfigure}[b]{0.32\textwidth}
  \includegraphics[width=\textwidth]{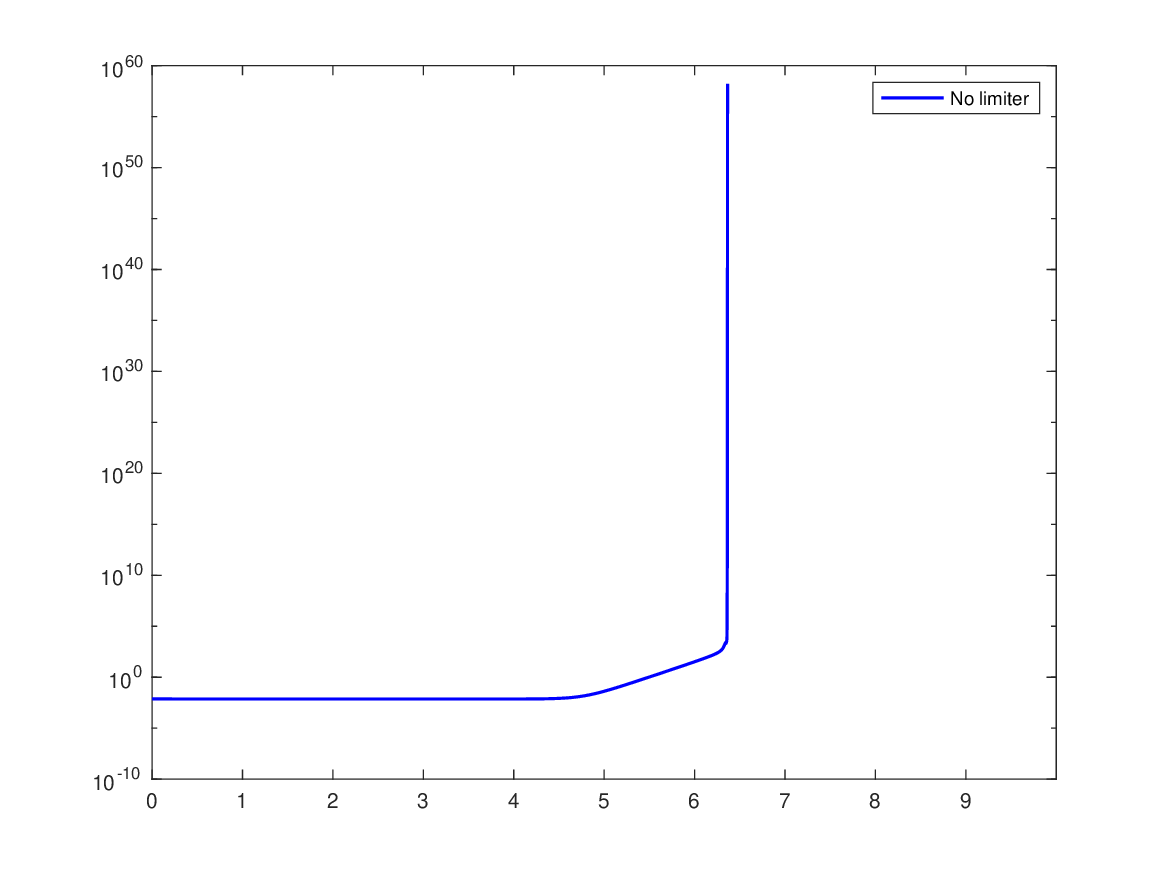}
  \caption{$L^1$-error over time}
 \end{subfigure}

 \begin{subfigure}[b]{0.32\textwidth}
  \includegraphics[width=\textwidth]{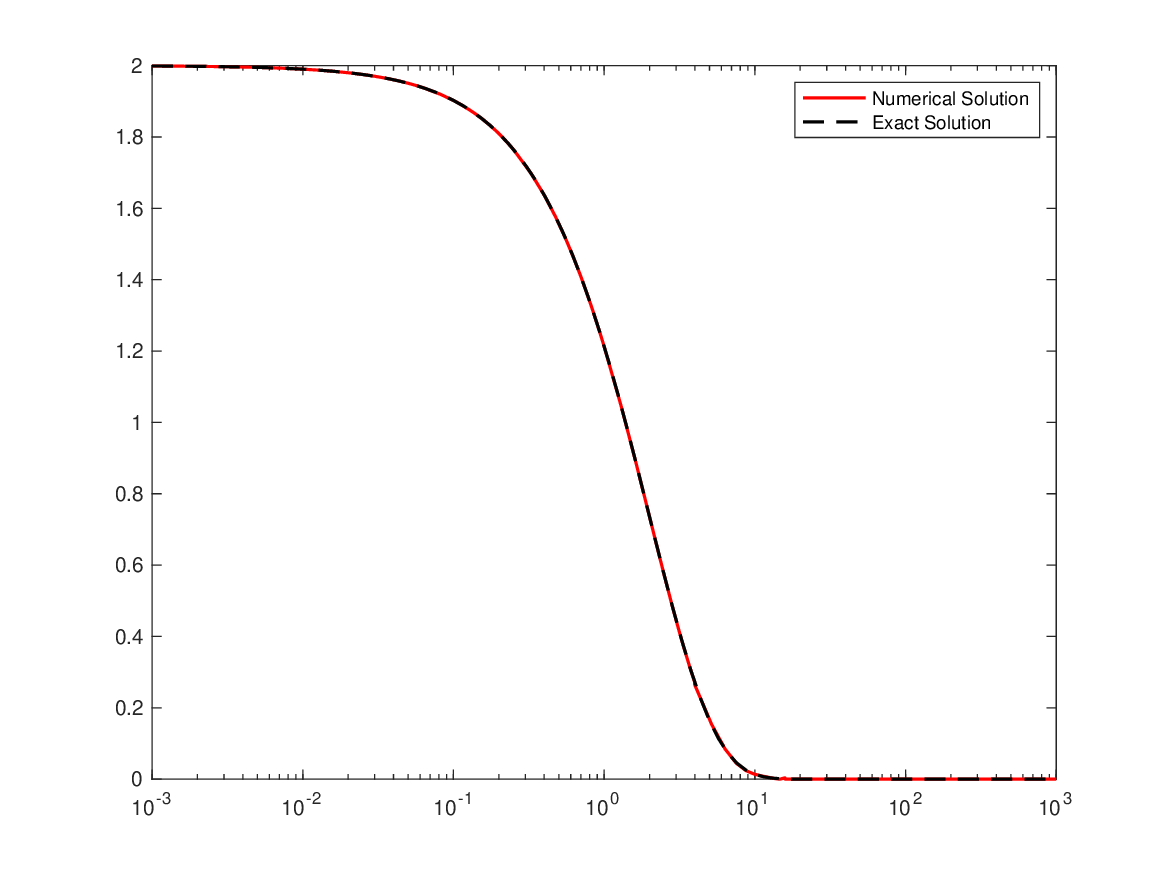}
  \caption{$n_h(v,t)$ at $t=10$}
 \end{subfigure}
 \begin{subfigure}[b]{0.32\textwidth}
  \includegraphics[width=\textwidth]{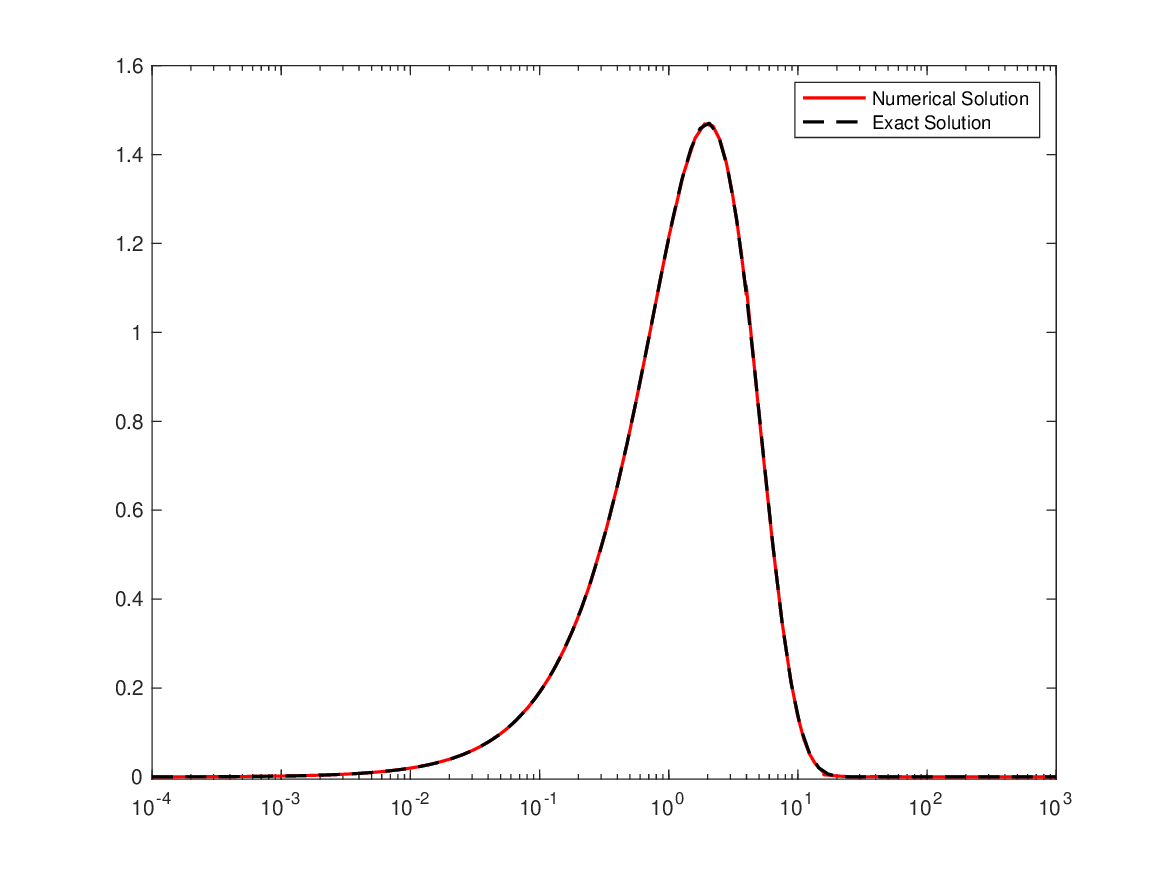}
  \caption{$v n_h(v,t)$ at $t=10$}
 \end{subfigure}
 \begin{subfigure}[b]{0.32\textwidth}
  \includegraphics[width=\textwidth]{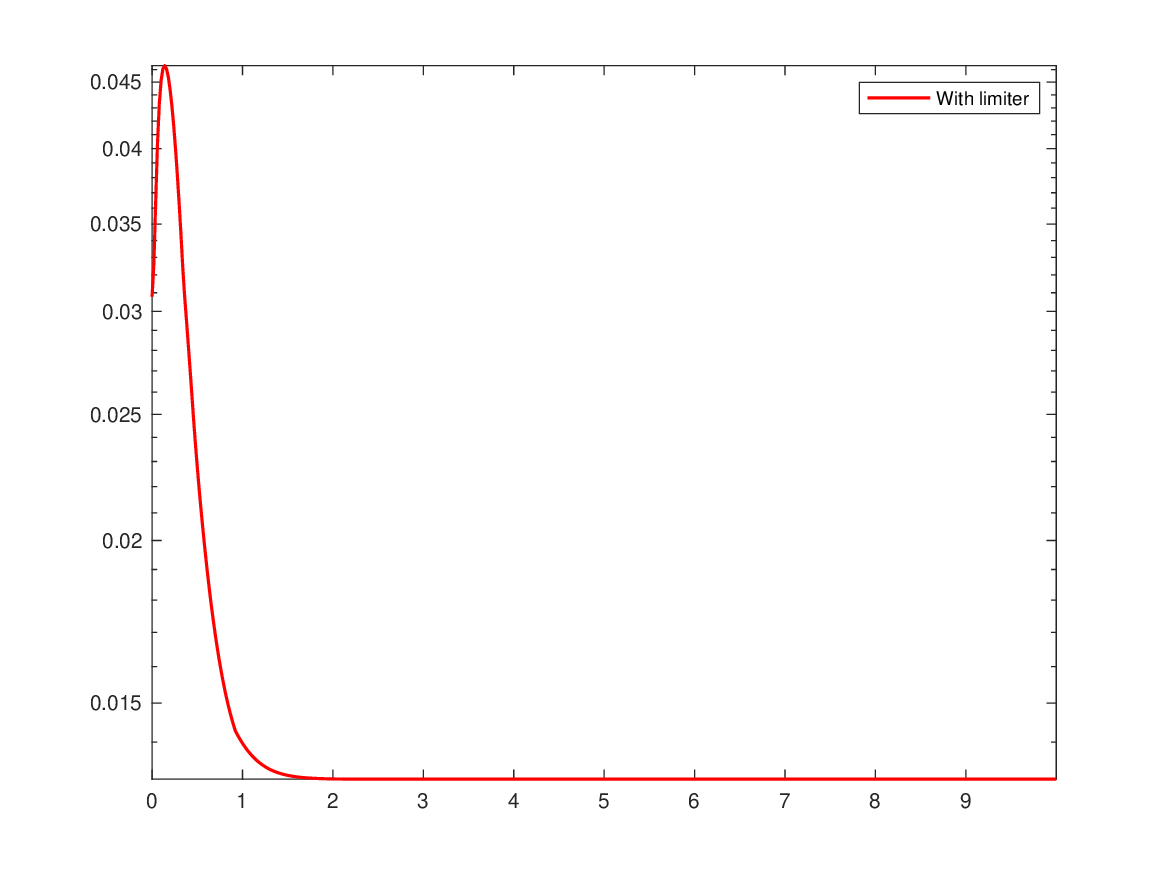}
  \caption{$L^1$-error over time}
 \end{subfigure}
\caption{\textbf{Example~\ref{ex:pptest}: Positivity test, Case~I.} Top row: DG solution ($\Delta t=0.003$) without limiter; Bottom row: DG solution ($\Delta t=0.02$) with limiter.}
\label{fig:pptest1}
\end{figure}

\begin{figure}[!htbp]
 \centering
 \begin{subfigure}[b]{0.32\textwidth}
  \includegraphics[width=\textwidth]{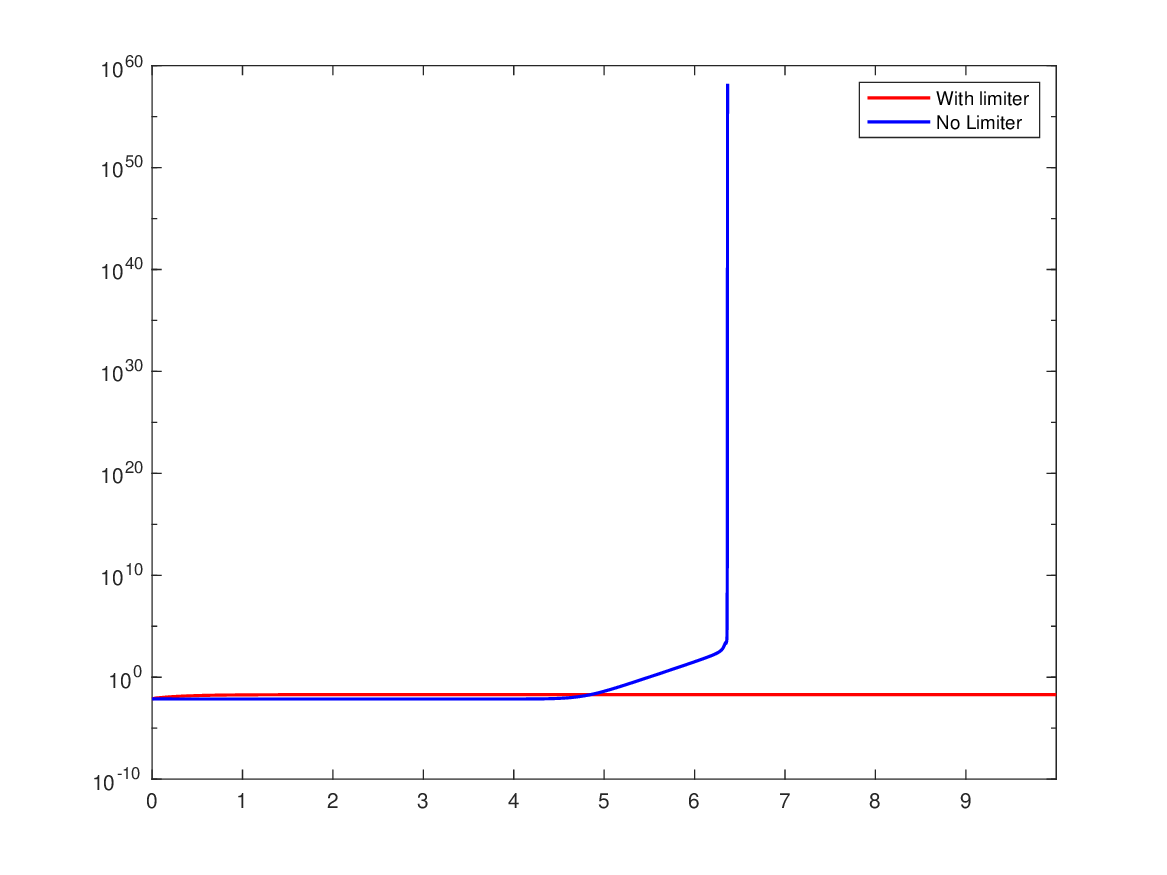}
  \caption{Case II: $L^{1}$-error over time}
 \end{subfigure}
 \begin{subfigure}[b]{0.32\textwidth}
  \includegraphics[width=\textwidth]{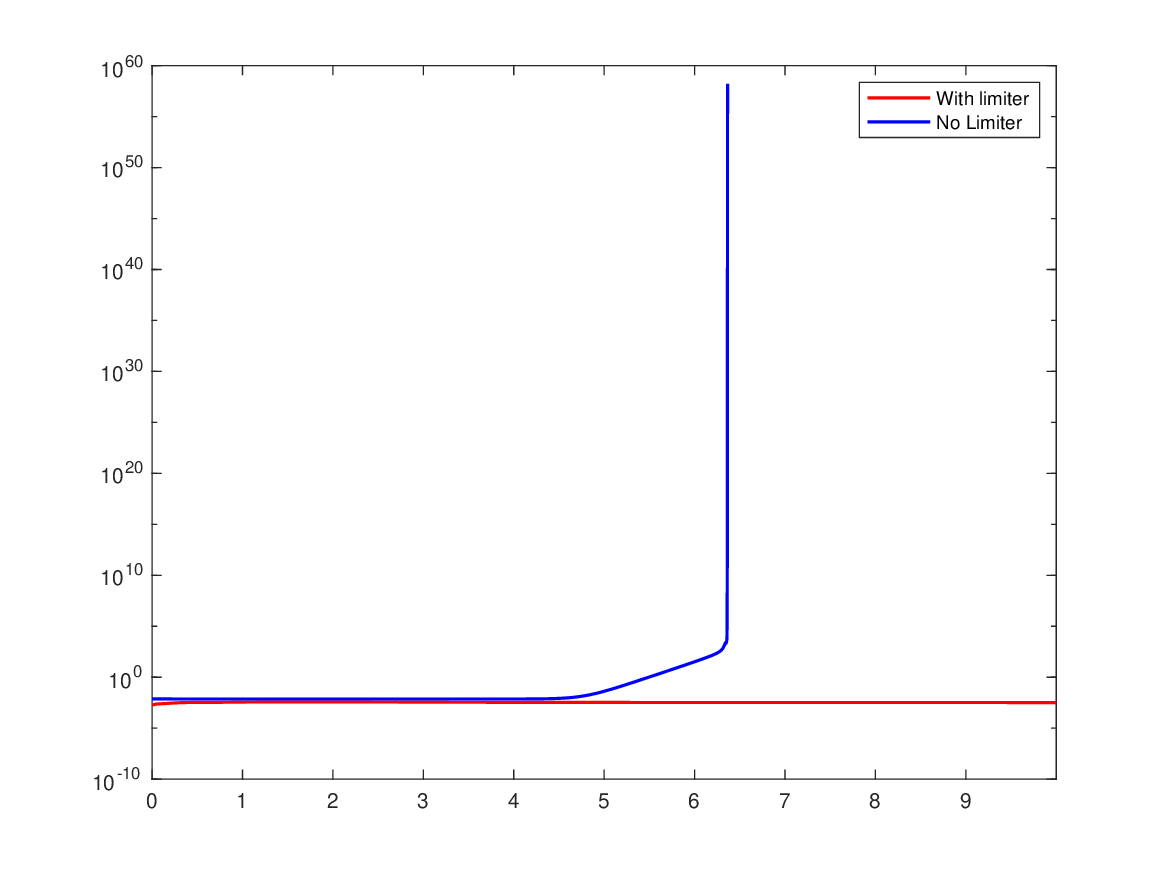}
  \caption{Case III: $L^{1}$-error over time}
 \end{subfigure}
\caption{\textbf{Example~\ref{ex:pptest}: Positivity test, Cases~II and III.}  
Comparison of $L^1$-errors over time for DG solutions without a limiter ($\Delta t = 0.003$) and with a limiter ($\Delta t = 0.02$).}
\label{fig:pptest2}
\end{figure}

% \section{Concluding Remarks}
% Possible extensions: Boltzmann equation? Nonlinear breakage effects -- collisional fragmentation \cite[]?

\section*{Conflict of Interest}
The authors declare that they have no conflicts of interest.

% \begin{appendices}
% \appendix
% \section{A}

% \end{appendices}

\end{document}